\documentclass[11pt,reqno]{amsart}    
\def\MR#1{}
\usepackage[margin=37mm]{geometry}
\usepackage{amsfonts,amsmath,amssymb,enumerate}    
\usepackage{amsthm, bm}    
\usepackage{mathrsfs}
\usepackage{setspace}
\usepackage{tikz}    
\usetikzlibrary{arrows,matrix}
\usetikzlibrary{decorations.markings,shapes.geometric,shapes.misc,cd} 

\usepackage{mathtools}    
\usepackage{todonotes}
\usepackage{hyperref}
\usepackage[capitalise]{cleveref}
\usepackage{upgreek}

\crefname{equation}{}{}
\Crefname{equation}{}{}

\theoremstyle{plain}
\newtheorem{thm}{Theorem}
\newtheorem{lem}[thm]{Lemma}
\newtheorem{prop}[thm]{Proposition}
\newtheorem{cor}[thm]{Corollary}

\newenvironment{thmbis}[1]
  {\renewcommand{\thethm}{\ref{#1}$'$}%
   \addtocounter{thm}{-1}%
   \begin{thm}}
  {\end{thm}}

\newtheorem*{prop*}{Proposition}
\theoremstyle{definition}

\theoremstyle{remark}
\newtheorem{rem}[thm]{Remark}

\newtheorem*{rem*}{Remark}
\newtheorem*{ack}{Acknowledgements}

\newcommand{\be}{\begin{equation}}    
\newcommand{\ee}{\end{equation}}    
\newcommand{\beu}{\begin{equation*}}    
\newcommand{\eeu}{\end{equation*}}    
\newcommand{\bea}{\begin{eqnarray}}    
\newcommand{\eea}{\end{eqnarray}}    
\newcommand{\beaa}{\begin{eqnarray*}}    
\newcommand{\eeaa}{\end{eqnarray*}}    
\newcommand{\bmx}{\begin{pmatrix}}    
\newcommand{\emx}{\end{pmatrix}}    

\newcommand{\uu}{\mathsf u}
 
\newcommand{\ol}{\overline}    
    
\newcommand{\del}{\partial}    
\newcommand{\g}{{\mathfrak g}}
\renewcommand{\b}{{\mathfrak b}}

\newcommand{\gh}{{\widehat \g}}
\newcommand{\hh}{{\widehat \h}}
    
\renewcommand{\a}{{\mathfrak a}}

\newcommand{\n}{{\mathfrak n}}    
\newcommand{\h}{{\mathfrak h}}    
\newcommand{\m}{{\mathfrak m}}    
    
\newcommand{\p}{{\mathfrak p}}

\newcommand{\mf}{\mathfrak}
\newcommand{\mc}{\mathcal}

\newcommand{\half}{\frac{1}{2}}

\newcommand{\nn}{\nonumber}

\newcommand{\8}{{\infty}}

\newcommand{\eps}{\epsilon}
\newcommand{\vareps}{\varepsilon}

\newcommand{\rank}{{\rm rank}}

\newcommand{\ad}{{\rm ad}}

\newcommand{\ZZ}{{\mathbb Z}}

\newcommand{\CC}{{\mathbb C}}

\newcommand{\bra}[1]{{\,\left<#1\right|}\,}    
\newcommand{\ket}[1]{{\,\left|#1\right>}\,}

\newcommand{\id}{{\textup{id}}}

\newcommand{\A}{\mc A}

\newcommand{\goi}[2]{=}    
\newcommand{\Hom}{\mathrm{Hom}}
\newcommand{\Homres}{\mathrm{Hom}^{\mathrm{res}}}

\newcommand{\lon}{\triangleright}
\newcommand{\on}{.}

\newcommand{\CCx}{\mathbb C^\times}

\renewcommand{\binom}[2]{{#1 \brack #2}}

\newcommand{\zFF}{\mathfrak z_{\mathrm{FF}}}

\newcommand{\btp}{\begin{tikzpicture}[baseline=0pt,scale=0.9,line width=0.25pt]}    
\newcommand{\etp}{\end{tikzpicture}}

\newcommand{\germ}{\mf}

\newcommand{\wt}{\widetilde}

\DeclareMathOperator{\res}{res}

\DeclareMathOperator{\Span}{span}

\newcommand{\nord}[1]{\,{:\!\!#1 \!\!:}\,}

\newcommand*{\longhookrightarrow}{\ensuremath{\lhook\joinrel\relbar\joinrel\rightarrow}}

\newcommand{\twist}{\varphi}
\newcommand{\la}{\langle}
\newcommand{\ra}{\rangle}
\newcommand{\shift}{\mathsf T}

\newcommand{\K}{\mc K}

\renewcommand{\O}{\mc O}

\newcommand{\sv}{\xi}
\newcommand{\Pisv}{{}^\sv \Pi^\eps}
\newcommand{\Pic}{\Pi}
\newcommand{\Rsv}{{}^\sv \! R}
\newcommand{\Hsv}{{}^\sv\hspace{-.3mm}\mathcal H^\eps}
\newcommand{\Hsvp}{{}^\sv\hspace{-.3mm}\mathcal H^{\eps,+}}
\newcommand{\Hsvpp}{{}^\sv\hspace{-.3mm}\mathcal H^{\eps,++}}
\newcommand{\hsv}{{{}^\sv\h}}
\newcommand{\Hsvloc}{\bigoplus_{i=1}^N \Hsv_{x_i}}

\newcommand{\Rnought}{{}^0 \! R}
\newcommand{\Hnought}{{}^0\hspace{-.3mm}\mathcal H^\eps}
\newcommand{\Hnoughtp}{{}^0\hspace{-.3mm}\mathcal H^{\eps,+}}
\newcommand{\Hnoughtpp}{{}^0\hspace{-.3mm}\mathcal H^{\eps,++}}
\newcommand{\triv}{\varphi}

\newcommand{\Vcrit}{{\mathbb V_0^{-h^\vee}}}

\newcommand{\M}{\mathcal M}

\newcommand{\vac}{\!\ket{0}\!}
\newcommand{\vacl}{\ket\lambda}

\newcommand{\Gaud}{\mathscr Z}

\DeclareMathOperator{\End}{End}

\DeclareMathOperator{\Fun}{Fun}

\DeclareMathOperator{\Der}{Der}
\DeclareMathOperator{\Aut}{Aut}
\newcommand{\ox}{\otimes}

\newcommand{\invlim}{\varprojlim}

\newcommand{\atp}[2]{\underset{\substack{$ $ \\ \tikz{
\draw[->] (0,0) -- (0,.13);} \\[-.2mm] {#1}}}{\smash{#2}}}

\newcommand{\into}{\hookrightarrow}

\newcommand{\onto}{\twoheadrightarrow}

\newcommand{\vpa}{\mathscr{P}}

\def\cocent{\mathsf d}

\newcommand{\lauleft}{\left\{\!\left\{}
\newcommand{\lauright}{\right\}\!\right\}}
\newcommand{\laurent}[1]{\!\lauleft #1 \lauright}

\newcommand{\picc}{\pi}
\newcommand{\pir}{{}^{\mathrm{(no\, der)}}\pi}
\newcommand{\pifin}{{}^{\mathrm{(fin)}}\pi}
\newcommand{\piaff}{{}^{\mathrm{(aff)}}\pi}
\newcommand{\Faff}{{}^{\mathrm{(aff)}}\mc F}
\newcommand{\Iaff}{{}^{\mathrm{(aff)}}\mc I}

\newcommand{\hr}{{}^{\mathrm{(no\, der)}}\h}

\newcommand{\haff}{{}^{\mathrm{(aff)}}\h}

\newcommand{\hfin}{{}^{\mathrm{(fin)}}\h}

\newcommand{\Qfin}{{}^{\mathrm{(fin)}}Q}

\def\lg{{{}^L\!\g}}
\def\lh{{{}^L\!\h}}

\def\ln{{}^L\!\n}

\def\lh{{}^L\h}

\def\hn{\hat\n}
\newcommand{\hb}{\hat\b}
\def\hg{\hat\g}
\def\hh{\hat\h}

\def\hN{\hat N}

\def\hB{ \hat B}

\def\H{ H}

\DeclareMathOperator{\diag}{diag}

\DeclareMathOperator{\Conn}{Conn}

\DeclareMathOperator{\op}{op}
\DeclareMathOperator{\Op}{Op}
\DeclareMathOperator{\MOp}{M\hspace{-.2mm}Op}

\newcommand{\F}{\mc F}
\newcommand{\Fc}{\ol{\mc F}}
\newcommand{\Wg}{W}
\newcommand{\Wgc}{W}

\DeclareMathOperator{\Lie}{Lie}

\newcommand{\Uloc}{\wt U_1(\hh^\eps)}

\newcommand{\Ulocc}{\wt U_1(L\h)}

\newcommand{\isom}{\xrightarrow\sim}
\newcommand{\GL}{GL}

\newcommand{\x}{{\bm x}}
\newcommand{\cp}[1]{{\mathbb P^{#1}}}
\newcommand{\conn}[1]{{}^{#1}\!\Conn}

\newcommand{\Diff}{\mathsf\Omega}

\newcommand{\Tk}{T^{\mathrm{(aff)}}}
\newcommand{\Tkn}{\nabla^{\mathrm{(aff)}}}
\newcommand{\vol}{\mathsf{dt}}

\newcommand{\U}{\mathscr U}
\newcommand{\largewedge}{\mbox{\Large $\wedge$}}
\newcommand{\hamd}{\mathbf v}
\newcommand{\coinv}{\mathsf F}
\newcommand{\nabt}{\nabla^T}
\newcommand{\confvec}{\omega}

\newcommand{\chal}{\check\alpha}
\newcommand{\al}{\alpha}

\newcommand{\bilin}[2]{\hspace{.5mm}\kappa\hspace{-.5mm}\big( #1|#2\big)}
\newcommand{\bilinvee}[2]{\hspace{.5mm}\kappa^\vee\hspace{-.8mm}\big( #1|#2\big)}

\newcommand{\coxeter}{h}
\newcommand{\dualcoxeter}{{h^\vee}}

\newcommand{\aaa}{a}
\newcommand{\chaaa}{\check a}

\newcommand{\cent}{\mathsf k}
\newcommand{\chcent}{\updelta}

\newcommand{\weyl}{\rho}
\newcommand{\chweyl}{{\check\rho}}

\newcommand{\La}{\Lambda}
\newcommand{\chLa}{\check\Lambda}

\newcommand{\trcart}{{}^tA}
\newcommand{\cart}{A}
\newcommand{\tildeal}{\widetilde\alpha}

\author{Charles Young}
\address{
Department of Physics, Astronomy and Mathematics, University of Hertfordshire, College Lane, Hatfield AL10 9AB, UK.}  \email{c.a.s.young@gmail.com}
\date{\today}

\begin{document} 
\title{Affine opers and conformal affine Toda}

\begin{abstract}
For $\mathfrak g$ a Kac-Moody algebra of affine type, we show that there is an $\textup{Aut}\, \mathcal O$-equivariant identification between $\Fun\Op_{\mathfrak\g}(D)$, the algebra of functions on the space of ${\g}$-opers on the disc, and $W\subset \pi_0$, the intersection of kernels of screenings inside a vacuum Fock module $\pi_0$. This kernel $W$ is generated by two states: a conformal vector, and a state $\updelta_{-1}\left|0\right>$. 
We show that the latter endows $\pi_0$ with a canonical notion of translation $T^{\mathrm{(aff)}}$, and use it to define the densities in $\pi_0$ of integrals of motion of classical Conformal Affine Toda field theory. 

The $\Aut\O$-action defines a bundle $\Pi$ over $\mathbb P^1$ with fibre $\pi_0$. We show that the product bundles $\Pi \ox \Omega^j$, where $\Omega^j$ are tensor powers of the canonical bundle, come endowed with a one-parameter family of holomorphic connections, $\nabla^{\mathrm{(aff)}} - \alpha T^{\mathrm{(aff)}}$, $\alpha\in \CC$. The integrals of motion of Conformal Affine Toda define global sections  $[\mathbf v_j dt^{j+1} ] \in H^1(\mathbb P^1, \Pi\ox \Omega^j,\nabla^{\mathrm{(aff)}})$ of the de Rham cohomology of $\nabla^{\mathrm{(aff)}}$.

Any choice of 
${\mathfrak\g}$-Miura oper $\chi$ 
gives a 
connection $\nabla^{\mathrm{(aff)}}_\chi$ on $\Omega^j$. Using coinvariants, we define a map $\mathsf F_\chi$ from sections of  $\Pi \ox \Omega^j$ to sections of $\Omega^j$. We show that $\mathsf F_\chi \nabla^{\mathrm{(aff)}} = \nabla^{\mathrm{(aff)}}_\chi \mathsf F_\chi$, so that $\mathsf F_\chi$ descends to a well-defined map of cohomologies. Under this map, the classes $[\mathbf v_j dt^{j+1} ]$ are sent to the classes in $H^1(\mathbb P^1, \Omega^j,\nabla^{\mathrm{(aff)}}_\chi)$ defined by the ${\g}$-oper 
underlying $\chi$. 
\end{abstract}


\maketitle
\setcounter{tocdepth}{1}
\tableofcontents

\section{Introduction and overview}
In this paper we generalize certain results about classical integrable systems associated to a finite-dimensional simple Lie algebra $\g$, to the case in which $\g$ is an affine algebra. 
These results concern the interplay of \emph{global} and \emph{local} pictures: here \emph{global} means associated to a copy of the Riemann sphere $\cp 1$ with some collection of marked points $\x = \{x_1,\dots,x_N\}$, while \emph{local} means associated to the formal disc $D$, which one should think of as a small disc about some point in $\cp 1\setminus\x$.

\subsection{}
To set the scene, let us begin with an overview of the situation when $\g$ is of finite type. 
The starting point is the \emph{free-field realization} of the \emph{classical W-algebra} associated to $\g$:
\be W\longhookrightarrow \pi_0 .\label{Wemb}\ee
Here $\pi_0$ is the vacuum Fock module for a system of $\rank(\g)$ free bosons and its subspace $W$ is defined to be the intersection of the kernels of certain \emph{screening operators} associated to the simple roots of $\g$. 
(Details are in \cref{sec: pieps}.) 
It turns out that $W$ is a Poisson vertex algebra generated by $\rank(\g)$ generators.
In the language of integrable systems, these generators correspond to the fields of the classical $\g$-KdV hierarchy, and the embedding \cref{Wemb} describes how the latter are expressed in terms of the fields of the classical $\g$-mKdV hierarchy. Equivalently, the generators of $W$ can be seen as the integrals of motion of the classical Toda field theory associated to $\g$, whose Hamiltonian $H$ in the light-cone formalism is nothing but the sum of the screening operators \cite{FFIoM}. 

The first of the generators of $W$ is a \emph{conformal vector} $\ol\confvec \in W\subset \pi_0$. Such a vector defines a preferred action, on $W$ and $\pi_0$, of the group $\Aut\O$ of changes of local holomorphic coordinate on the disc. 
Moreover, one has the following commutative diagram, which identifies $W$ and $\pi_0$ with geometrically-defined objects associated to the disc \cite[Theorems 11.2, 11.3]{Frenkel_2005}:
\be\label{pffw}
\begin{tikzcd}
\picc_0 \rar{\sim} & \Fun \MOp_\g(D) \\
\Wgc\uar[hook]  \rar{\sim}& \Fun \Op_\g(D) \uar[hook] 
\end{tikzcd}
\ee
Namely, $\Fun \Op_\g(D)$ is the algebra of functions on the space of $\g$-\emph{opers} on the disc \cite{DS}, and it embeds into $\Fun \MOp_\g(D)$, the algebra of functions on the space of $\g$-\emph{Miura opers} on the disc. We recall the definitions below: for the moment the important point is that the group $\Aut\O$ acts naturally on both these algebras, and the diagram \cref{pffw} is then $\Aut\O$-equivariant.

Now let us recall the passage from this local picture to the global picture. 
An oper is a certain gauge equivalence class of connections \cite{BDopers} (see \cref{sec: def oper}). 
It turns out that, for $\g$ of finite type, to give an $\g$-oper on an open subset $U$ of the Riemann sphere is the same thing as giving, on $U$, sections of certain tensor powers $\Omega^{j+1}$ of the canonical bundle $\Omega$ (i.e the bundle of holomorphic one-forms), together with a projective connection.\footnote{For the meaning of \emph{projective connection} see e.g. \cite[\S3.5.7]{Fre07}, but it is not crucial here.} 
(The relevant powers are given by the exponents of $\g$.) 
So opers form a sheaf, and one has the restriction map to the stalk of this sheaf at any point $x\in U$. 
This restriction map is an embedding. Indeed, relative to a choice of local holomorphic coordinate  $t:U\to \CC$, a meromorphic section $f_j(t) dt^{j+1}$ of $\Omega^{j+1}$ is given by a meromorphic function $f_j(t)$, and the germ of this section at a regular point $x$ is given by the Taylor expansion of this function at $x$. And, as usual, a meromorphic function can always be recovered from its Taylor series. (The resulting germs of sections correspond to the $\rank(\g)$ generators of $\Wgc$ mentioned above.) 
The upshot is that, for $\g$ of finite type, there is an embedding, global into local,
\be \Op_{\g}(\cp1)_{\x} \into \Op_{\g}(D) \label{emb}\ee
of the space of meromorphic opers on $\cp1$ holomorphic away from the marked points, into the space of opers on the formal disc about any point $x\in \cp1\setminus\x$. 

One can also understand the global picture in terms of the objects on the left of the diagram \cref{pffw}, the Poisson vertex algebras $W$ and $\pi_0$. Indeed, the $\Aut\O$ action defined by the conformal vector $\ol\confvec$ allows one to attach copies of $\pi_0$ to points in $\cp1$ (or any Riemann surface) in a coordinate-independent fashion, giving rise to a vector bundle over $\cp 1$ with fibre $\pi_0$ \cite{FrenkelBenZvi}. Global Miura opers then correspond to certain global sections of the dual bundle, as we shall recall in \cref{crs,sec: opers}. 

The main point to take away is that, for $\g$ of finite type, one has the surjection of the algebras of functions, 
\be \Fun\Op_{\g}(D) \onto \Fun\Op_{\g}(\cp1)_{\x},\label{surj}\ee 
coming from the embedding \cref{emb}. The algebra  $\Fun\Op_{\g}(\cp1)_{\x}$ of functions on the space of global opers is of great interest, for a reason we return to in \cref{gaudsec} below, and this surjection guarantees we get all such functions starting from states in $W \cong \Fun\Op_{\g}(D)$.

\subsection{}
Now, and for the rest of the paper, let us suppose that $\g$ is a Kac-Moody algebra of affine type. Our goal is to give the natural analogs, in this case, of the statements above.

Let us begin with the local situation on the disc. 
The first observation is that the definitions of all the objects appearing in the diagram \eqref{pffw} still make perfect sense. (They are given in \cref{sec: caf,sec: opers} below.)  In particular, one still has $\pi_0$, the Fock space for the loop algebra $\h\ox\CC((t))$ over the Cartan subalgebra $\h\subset \g$. Importantly, we mean the full Cartan subalgebra, including the central element $\cent$ and derivation element $\mathsf d$. For convenience, we shall identify $\h$ and $\h^*$ by means of the standard nondegenerate bilinear form. Then $\cent$ gets identified with the imaginary root $\chcent$.  

The first result of the  paper is the following.
\begin{thm}\label{propcd}
For $\g$ of affine type one (still) has the commutative diagram \eqref{pffw} and it is (still) $\Aut\O$-equivariant.
\end{thm}
However, whereas in finite types $\Wgc$ was generated by $\rank(\g)$ generators, in affine types we shall see that it is generated by just two states,
\be \chcent_{-1} \vac \qquad\text{and}\qquad \ol\confvec. \nn\ee

The state $\ol\confvec\in \Wgc\subset \picc_0$ is again a conformal vector, and it defines the action of the group of coordinate transformations $\Aut\O$ on $\picc_0$, exactly as above.\footnote{Though note that in the affine case the central charge is zero and the field corresponding to $\ol\confvec$ transforms like a germ of a section of $\Omega^2$ rather than the germ of a projective connection.} 

By contrast, the state $\chcent_{-1}\vac = \cent_{-1} \vac \in \Wgc\subset\picc_0$ is a new feature of the affine case. It too plays a vital role, as follows. 
We shall introduce the \emph{canonical translation operator} $\Tk$. It is given by 
\begin{align} 
\Tk &:= \sum_{m=0}^\8 \frac {(-1)^m}{(\coxeter)^m m!} \sum_{n_1,\dots,n_m \geq 1} \frac{1}{n_1\dots n_m} \chcent_{-n_1} \dots \chcent_{-n_m} L_{n_1+\dots+ n_m-1} \nn
\\&= L_{-1} - \frac {\chcent_{-1}}\coxeter L_0 + \left( \frac 12\frac{\chcent_{-1}}\coxeter\frac{\chcent_{-1}}\coxeter - \frac{\chcent_{-2}}{2\coxeter}\right) L_1 + \dots   ,\nn
\end{align}
where the $L_j$ are the generators of $\Der\O$. Here $\coxeter$ is the Coxeter number of $\g$. $\Tk$ is to be seen as a modification of the usual vertex algebra translation operator $T= L_{-1}$.
It is canonical in the sense that for all $j\geq 1$, $\left[ L_j, \Tk \right] = 0$ (\cref{lem: Lj}), which means it provides a notion of translation independent of our choice of local holomorphic coordinate on the disc (in contrast to $T$). 

The negative modes of any state $X= X_{-1} \vac$ in a vertex algebra can always be obtained by repeated application of the translation operator:  $X_{-2} = [T,X_{-1}]$, $X_{-3} = \frac 1 2[T,[T,X_{-1}]],\dots$ and so on. 
In our setting is natural to define \emph{canonical modes} of states using the canonical translation operator: we set 
\be X_{[-2]} := [\Tk, X_{-1}],\qquad X_{[-3]} := \frac 1 2[\Tk,[\Tk,X_{-1}]], \dots\nn\ee 
and so on. States constructed from the action of such canonical modes on the Fock vacuum $\vac$ will be conformal primaries. 

Armed with this canonical notion of translation, we establish the following. (For details see \cref{sec: tkdef}.) 
\begin{thm}\label{TkdiagIntro} We have the following $\Aut\O$-equivariant double complex:
\be \begin{tikzcd}
\CC\vol \\
\picc_0 \ox \CC\vol \rar{H\ox \id} \uar{\bra 0\ox \id}
& \bigoplus_{i\in I} \picc_{\al_i} \ox \CC\vol 
\\
\picc_0 \uar{\Tk\vol} \rar{H}
& \bigoplus_{i\in I} \picc_{\al_i} \uar{-\Tk\vol}\\
\CC\vac \uar
\end{tikzcd}
\nn\ee
\end{thm}
Here $H$ is the sum of the \emph{screening operators} corresponding to the simple roots of the affine algebra $\g$, $H := \bigoplus_{i\in I} \ol S_{\al_i}$,  
and by definition $\Wgc := \ker H$. 
In physics terminology, $H$ is the Hamiltonian of classical \emph{Conformal Affine Toda} field theory \cite{Babelon_1990,Bonora_1991,ACFGZ,Paunov1,PapadopoulosSpence} in the light-cone formalism. 

The double complex above is an $\Aut\O$-equivariant analog of a double complex from \cite{FFIoM}, which dealt with Affine Toda field theory.  
In Affine Toda field theory one works with the Fock module corresponding to the Cartan subalgebra of the underlying finite-type algebra, $\hfin := \left(\bigoplus_{i\in I} \CC\al_i\right)/ \CC\chcent$ in our conventions; we shall denote this Fock module by $\pifin_0$. 

We introduce a subspace $\piaff_0\subset \picc_0$ (\cref{sec: piaff}). It is defined using the canonical modes and so it is spanned by conformal primaries. We show that it is isomorphic to the Fock module $\pifin_0$ as a vector space: 
\be \pifin_0 \isom_\CC \piaff_0.\nn\ee
One can think that this isomorphism takes monomials in $\pifin_0$ and ``decorates'' them with appropriate terms involving negative modes of $\chcent$. For example, it turns out that for any $i,j\in I$,
\be [\al_i]_{-2}[\al_{j}]_{-1}\vac \mapsto
\left( \al_{i,-2} - \frac 1 \coxeter \chcent_{-1} \al_{i,-1}\right)\left(\al_{j,-1} - \frac 1 \coxeter \chcent_{-1} \right)\vac.\nn\ee

The classical screening operators obey the Serre relations of $\g$ and stabilize $\piaff_0\subset \picc_0$. This makes $\piaff_0$ into a module over $\n_+$, the subalgebra of $\g$ generated by the positive simple root vectors $e_i$, $i\in I$. We shall establish the following (\cref{afthm}).
\begin{thm}\label{afthmIntro} As $\n_+$-modules, $\piaff_0$ and $\pifin_0$ are isomorphic: \be \piaff_0 \,\,\,\cong_{\n_+} \pifin_0 .\nn\ee
\end{thm}
This is very useful because it allows us to import results from \cite{FFIoM} wholesale. 
We get the densities $\hamd_j$ of integrals of motion of Conformal Affine Toda field theory. 
They are nothing but the images of the usual densities of integrals of motion of Affine Toda field theory (or equivalently, the densities of the $\g$-mKdV Hamiltonians) under the ``decoration'' map above. See \cref{thm: Iaff} and \cref{vjthm} in \cref{sec: piaff}. 
For each exponent $j\in E$, $\hamd_j$ is a conformal primary in $\picc_0$ of conformal weight $j+1$.

\subsection{}
Now let us describe the passage from the local to the global picture in the affine case. Roughly speaking, the main idea is that rather than merely attaching copies of $\pi_0$ and its subspace $W=\ker H$ to points in the Riemann sphere, we should now attach copies of the whole $\Aut\O$-equivariant double complex from \cref{TkdiagIntro}. 

To explain why that is so, let us consider opers in affine types. While the definition of opers in affine types is itself in very close  analogy to the definition in finite types, the data needed to specify such an oper turn out to be of a different character.  Recall that an \emph{affine connection} $\nabla$ is by definition a connection on the canonical bundle $\Omega$. It allows one to differentiate sections of $\Omega$, and in fact sections of $\Omega^j$ for any integer $j$. 
When $\g$ is of affine type, to give an $\g$-oper on $U$ is the same as giving the following on $U$: first, an affine connection $\nabla$; second, a section of $\Omega^2$; and third, sections of the de Rham cohomologies 
\be H^1(\Omega^j,\nabla):= \Gamma(\Omega\ox \Omega^j)/ \nabla \Gamma(\Omega^j)\label{acc}\ee for the connection $\nabla$ with coefficients in $\Omega^j$, as $j$ ranges over the (now, countably infinite) set $E_{\geq 2}$ of exponents of $\g$. So ``most'' of the information in the oper now comes in the form of these cohomology classes. An important consequence is that the restriction map 
\be \Op_{\g}(\cp1)_\x\to \Op_{\g}(D_x)\nn\ee
to opers on the disc at $x\in \cp1\setminus \x$, is now very far from being an embedding, in contrast to the situation in \cref{emb} above. Indeed, on the (unpunctured) disc all the cohomologies are trivial. All that survives are the germs at $x$ of the affine connection and of the section of $\Omega^2$ (and these correspond to the states $\chcent_{-1}\vac$ and $\ol\confvec$).

Thus, in contrast the case of finite types, in affine types we certainly cannot expect to construct all functions on the space of global opers starting from states in $W \cong \Fun\Op_{\g}(D)$. Instead we must work in an appropriate de Rham cohomology, as follows. 

Using the action of $\Aut\O$ on $\picc_0$ defined by $\ol\confvec$, we define a vector bundle $\Pic$ over $\cp1$ with fibre $\picc_0$, following \cite{FrenkelBenZvi}. Such vertex-algebra bundles always carry a canonical flat holomorphic connection $(\nabt)_{\del_t}:= \del_t + T$, defined by the translation operator $T=L_{-1}$ relative to any local holomorphic coordinate $t$. In our setting though we get more: we show there is a flat holomorphic connection on $\Pic\ox \Omega^j$ for any integer $j$,
given by 
\be (\Tkn)_{\del_t} = \del_t + L_{-1} - j \frac{\chcent_{-1}}{\coxeter},\nn\ee
and, moreover, that this connection belongs to a one-parameter family of connections, as follows. (See \cref{sec: con}.)
\begin{thm}\label{act} Let $\alpha\in \CC$. 
For each $j\geq 0$ there is a flat holomorphic connection 
on $\Pic\ox\Omega^j$ given by 
\be (\Tkn)_{\del_t} + \alpha \Tk. \nn\ee
\end{thm}

Associated to any conformal primary state $v\in \picc_0$ of conformal weight $j$, one gets a global section of $\Pi\ox \Omega^j$, given locally, in any holomorphic coordinate $t$, by $vdt^j$. It is constant with respect to the connection $\Tkn - \Tk dt$, i.e. it obeys
\be \Tkn vdt^j = (\Tk v) dt^{j+1} .\nn\ee 
In particular, this applies to the densities $\hamd_j\in \picc_0$. These densities $\hamd_j$ defined classes in the cohomology for the vertical complex in \cref{TkdiagIntro}, i.e. they were defined only up to the addition of canonical translates $\Tk f_j$ with $f_j\in \picc_0$ primary of weight $j$. It will follow that the corresponding global sections $\hamd_jdt^{j+1}$ define classes
\be [\hamd_j dt^{j+1} ] \in H^1(\cp1, \Pic\ox \Omega^j,\Tkn) \label{ucc}\ee
in the de Rham cohomology of the flat holomorphic connection $\Tkn$ with coefficients in $\Pic\ox\Omega^j$.  These classes are to be seen as ``universal'' versions of the classes in \cref{acc} defined by any one particular oper on $\cp1$ -- in a sense we now make precise. 

\bigskip

To connect back to opers on $\cp 1$, we introduce spaces of coinvariants on $\cp1$. To each marked point $x_i\in \x$ we attach a one-dimensional module $\CC v_{\chi_i}$. Roughly speaking, it is a module over the loop algebra $\h\ox \CC((t))$ and it is specified by an element $\chi_i \in \h^* \ox \CC(t))$. (For the precise statement see \cref{sec: cchi}.) The space of coinvariants is non-trivial (and of dimension one) if and only if these $\chi_i$ are the germs of a global meromorphic section $\chi$ of a certain sheaf of connections, $\conn{-\chweyl}$. These statements are all exactly as in the case of $\g$ of finite type (for which see \cite{Fre07}). The only difference is that in affine types $\chweyl$ is a derivation element, i.e. it doesn't lie in the span of the simple roots. One upshot of that is that we get a preferred map 
\be \conn{-\chweyl}(\cp1)_\x \onto \Conn(\cp1,\Omega)_\x; \quad 
                     \chi \mapsto \Tkn_\chi \nn\ee
which takes our global section $\chi$ of $\conn{-\chweyl}(\cp1)_\x$ and produces an affine connection which we denote by $\Tkn_\chi$.
The connection $\chi\in \conn{-\chweyl}(\cp1)_\x$ is equivalent to a \emph{global Miura oper} $\nabla \in \MOp_{\g}(\cp1)_\x$. To such a Miura oper corresponds an underlying oper $[\nabla] \in \Op_{\g}(\cp1)_\x$ (see \cref{sec: opers}). The affine connection $\Tkn_\chi$ is nothing but the affine connection associated to that oper.

By considering a modified space of coinvariants in which we insert an additional module, isomorphic to $\picc_0$, at a point $x\in \cp1\setminus \x$, we then obtain, for each $\chi$, a map
\be \coinv_\chi: \Gamma(U,\Pic\ox \Omega^j)_\x \to \Gamma(U,\Omega^j)_\x . \nn\ee 
which takes meromorphic sections of $\Pic\ox\Omega^j$ to meromorphic sections of $\Omega^j$. 
We show that this map $\coinv_\chi$ is functorial in the following sense (see \cref{functhm,funccor})
\begin{thm}\label{tkIntro}
We have
\be \coinv_\chi \Tkn = \Tkn_\chi \coinv_\chi \nn\ee
and therefore we get a well-defined map of cohomologies
\be [\coinv_\chi]: H^1(U,\Pic\ox \Omega^j,\Tkn)_\x \to H^1(U,\Omega^j, \Tkn_\chi)_\x . \nn\ee
\end{thm}

Finally, we are in a position to apply the map of cohomologies $[\coinv_\chi]$ to the global sections $[\hamd_j dt^{j+1} ] \in H^1(\cp1, \Pic\ox \Omega^j,\Tkn)$, \cref{ucc}, coming from the Conformal Affine Toda integrals of motion. Their images under $[\coinv_\chi]$ are indeed the cohomology classes for that particular oper $[\nabla]$, as in \cref{acc}.  

\subsection{}\label{gaudsec}
To conclude this introduction, let us mention one of the main motivations for this paper. We do so purely for context -- nothing else in the paper depends on the contents of this section.

For Kac-Moody algebras $\g$ of finite type, there is a remarkable correspondence between the Gaudin/Bethe algebra associated to $\g$ and the algebra of functions on the space of opers on $\cp1$ for the Langlands dual Lie algebra $\lg$ \cite{FrenkelGaudinLanglands,Fopers}, and cf. \cite{Fre07,MTV1,MTVSchubert,RybnikovProof}.

It was conjectured in \cite{FFsolitons} that such statements should generalize to affine types: namely, in that paper the authors introduced a notion of affine opers and affine Miura opers and conjectured that they encode the spectra of affine Gaudin models for the Landlands dual Kac-Moody algebra.
Proving such a conjecture is an interesting but difficult open problem: interesting because there are close links between quantum Gaudin models in affine types and the ODE/IM correspondence \cite{DT,BLZ,BLZ4,BLZ5,DDT,MRV1,MRV2,FH,FJM,MR1,MR2} and integrable quantum field theories \cite{FFsolitons,V17,LacroixThesis,DLMV1,DLMV2,Vic2,Lacroix2019}; but difficult because it is not even clear what one should be diagonalizing. Indeed, while the quadratic Hamiltonians which define the Gaudin model can be written down for all symmetrizable Kac-Moody algebras, and are known to be diagonalizable by Bethe ansatz, \cite{SV,RV}, \cite[Appendix A]{LVY}, it is expected, but not known in general, that they belong to some large commutative algebra of integrals of motion.

The conjecture of \cite{FFsolitons} was eludicated somewhat in \cite{LVY,LVY2}: namely, the eigenvalues of the (local) higher Gaudin Hamiltonians should again be given by functions on a space of $\lg$-opers on $\cp1$, but these functions now take the form of certain integrals of hypergeometric type (prompting the further conjecture that the Hamiltonians are also of this form).

One approach to constructing the Hamiltonians is to work grade-by-grade in the principal gradation. 
Indeed, one is interested in the diagonalization problem in tensor products of highest weight irreducible $\g$-modules. That problem breaks up into a sequence of finite-dimensional diagonalization problems, labelled by the depth in the principal gradation (or equivalently by the number of Bethe roots). 
The starting point in that program, and setting of the present paper, is the \emph{vacuum case}, i.e. the case of no Bethe roots. So one can think that the content of this paper is to describe, in this vacuum case, how global objects arise from local ones in affine types. 

To explain that statement, let us recall some facts about the relationship between local and global objects for Gaudin models in finite types. 
For $\g$ of finite type, one has commutative diagrams of the following sort:  \cite[Theorem 2.7]{Fopers}\footnote{We are skirting over subtleties about singularities at $\8\in \cp1$ here. Note that, in the present paper, $\x= \{x_1,\dots,x_N\}$ will be the only allowed singularities; $\8$ will have no special status. Cf. the constraint in \cref{globsec}.}
\be\label{zzff}
\begin{tikzcd}
\zFF(\gh) \dar[two heads] \rar{\sim} & \Fun\Op_{\lg}(D)\dar[two heads]\\
\Gaud(\g)_\x \rar{\sim} & \Fun\Op_{\lg}(\cp1)_{\x}
\end{tikzcd}
\ee
Here, in the top line, the algebra $\Fun\Op_{\lg}(D)$ of opers on the disc is identified in an $\Aut\O$-equivariant fashion with $\zFF(\gh) \subset \Vcrit(\g)$, the space 
of singular vectors of the vacuum Verma module at critical level \cite{FFGD}. 
On the bottom line $\Fun\Op_{\lg}(\cp1)_{\x}$, the algebra of functions on the space $\Op_{\lg}(\cp1)_{\x}$ of meromorphic opers on $\cp1$ holomorphic away from the marked points, is identified with $\Gaud(\g)_\x\subset U(\g)^{\otimes N}$, the Gaudin/Bethe algebra.
The surjection on the right is the one from \cref{surj}. 
The surjection $\zFF(\gh) \onto \Gaud(\g)_\x$ is defined using coinvariants \cite{FFR}.
Thus, in finite types, $\zFF(\gh)$ plays the role of a sort of ``universal object'' from which all the Gaudin Hamiltonians can be constructed after supplying the extra data of the marked points $x_1,\dots,x_N$.\footnote{It also allows one to construct generalizations of Gaudin models including cyclotomic Gaudin models \cite{VY1} and Gaudin models with irregular singularities \cite{FFT,FFRyb,RCactus,VY3}.}
The proof of the existence of the $\Aut\O$-equivariant isomorphism $\zFF(\gh) \isom  \Fun\Op_{\lg}(D)$ in finite types uses the Wakimoto realization \cite{FF1990,Wakimoto}: an embedding of vertex algebras $\Vcrit(\g) \into M \ox \pi_0$, where the vertex algebras on the right are Fock spaces for certain systems of free fields. 
Under this embedding, the space of singular vectors $\zFF(\gh)$ is mapped into $\pi_0$, and its image in $\pi_0$ is precisely the subspace $W= \ker H \subset \pi_0$ from \cref{pffw}. 
One can go on to find joint eigenvectors for the Gaudin Hamiltonians by considering spaces of coinvariants of $\gh$-modules on $\cp1$. One assigns \emph{Wakimoto modules} $M \ox \CC v_{\chi_i}$ to each of the marked points $x_i$ and also to certain other points, the Bethe roots, and one assigns a copy of $M \ox \pi_0$ to an additional point $x$, distinct from these. For details on this perspective on the Bethe ansatz see \cite{FFR}. 

If one is merely interested in \emph{vacuum} eigenvalues, there are no Bethe roots, and one arrives at a much simpler space of coinvariants -- namely of $\hh$-modules on $\cp1$, where $\h$ is the Cartan subalgebra -- in which one assigns one-dimensional modules $\CC v_{\chi_i}$ to the marked points $x_i$, and the Fock module $\pi_0$ to some additional point $x$. It is this simpler space of coinvariants which we study, in the case of $\g$ of affine type, in the present paper (see \cref{crs}).

\subsection{}
This paper is structured as follows. 

In \cref{sec: pieps,sec: cls} we recall standard facts about the vertex algebra $\pi_0^\eps$ and its classical ($\eps\to 0$) limit $\picc_0$. In particular, we recall the definition of the screening charges and their classical limits. In \cref{sec: ct} we give the definitions of the group $\Aut\O$ of coordinate transformations and its action on $\picc_0$. 
Conventions about Cartan matrices, roots, coroots, etc. are fixed in \cref{sec: cartan}. 

In \cref{sec: caf} we introduce the canonical translation operator $\Tk$ and canonical modes, and use them to define the subspace $\piaff_0$. We establish the isomorphism of $\n_+$-modules $\pifin_0\cong\piaff_0$ and arrive at the densities of integrals of motion $\hamd_j \in \picc_0$. 

In \cref{crs} we move from local to global objects. We define the vertex algebra bundle $\Pisv$ with fibre $\pi_0^\eps$ and a sheaf of Lie algebras $\Hsv$ associated to this bundle. We define $\conn\sv$ and spaces of coinvariants, and the map $\coinv_\chi$. Then in \cref{sec: tto} we specialise back to our specific case and go on to define the connection $\Tkn$ and show that $\coinv_\chi \Tkn = \Tkn_\chi \coinv_\chi$. 

Finally in \cref{sec: opers} we recall the definitions of opers and Miura opers. We can then establish the $\Aut\O$-equivariant identifications $\picc_0 \cong \Fun\MOp_{\g}(D)$ and $\Wgc \cong \Fun\Op_{\g}(D)$, and the identification of $[\coinv_\chi( \hamd_j dt^j ) ]$ with the cohomology classes for the oper underlying $\chi$.

\begin{ack}
The author is glad to acknowledge the many helpful remarks of the referee, and in particular the thoughtful suggestions of improvements to the introduction.
\end{ack}

\section{Free fields and screening operators}\label{sec: pieps}
Our first objective is to recall the definitions of the objects in the embedding \cref{Wemb}:  the vacuum Fock module $\pi_0$, and the screening operators which act on it. The Fock module $\pi_0$ is a $\eps\to 0$ limit of a one-parameter family of vertex algebras $\pi_0^\eps$.  It is useful to start with this whole family, because doing so provides a natural way of understanding the semi-classical structures on $\pi_0$, as we shall see in the next section.

\subsection{Free fields}\label{sec: ff}
Let $\h$ be a finite-dimensional vector space over $\CC$, equipped with a non-degenerate symmetric bilinear form $\bilin\cdot\cdot$. We pick a basis $\{b_i\}_{i=1}^{\dim \h}$ of $\h$ and let $\{b^i\}_{i=1}^{\dim\h}\subset \h$ be its dual basis with respect to the form $\bilin\cdot\cdot$:
\be \bilin{b_i}{b^j} = \delta_i^j .\nn\ee

We shall regard $\h$ as a commutative Lie algebra. 
Let $\eps$ be a parameter and denote by $\hh^\eps$ the central extension of the loop algebra $\h((t))$, by a one-dimensional centre $\CC \bm 1$, defined by the cocycle $\eps \res \bilin f{dg}$. That is, $\hh^\eps$ has commutation relations
\be [b^i_m, b^j_n] = \eps \bm 1 m \delta_{n+m,0} \bilin{b^i}{b^j}  \label{hcom}\ee 
where the generators are $b^i_m := b^i \ox t^m$ for $i=1,\dots,\dim\h$ and $m\in \ZZ$. 

\subsection{Fock modules}\label{sec: fm}
For any $\lambda\in \h$, define the $\hh^\eps$-module 
\be \pi^\eps_\lambda := U(\hh^\eps) \ox_{U(\h[[t]]\oplus \CC \bm 1)} \CC \vacl \nn\ee
induced from the one-dimensional $(\h[[t]]\oplus \CC \bm 1)$-module $\CC \vacl$ spanned by a vector $\vacl$ obeying
\be b^i_0 \vacl = \eps\bilin \lambda {b^i} \vacl , \qquad b^i_n \vacl = 0,\quad n>0,\label{vldef}\ee
and $\bm 1 \vacl = \vacl$. These are called Fock modules. As a special case we have the vacuum Verma module, $\pi_0^\eps$. 
For every $\lambda$, we have the isomorphism
\be \pi^\eps_\lambda \cong \CC[b_{i,n}]_{i=1,\dots,\dim\h; n<0}, \nn\ee
of vector spaces, and of modules over $U(t^{-1}\h[t^{-1}])\cong \CC[b_{i,n}]_{i=1,\dots,\dim\h; n<0}$.

\subsection{Algebra of fields $\Uloc$}\label{sec: lc}
Let $U_1(\hh^\eps)$ denote quotient of the enveloping algebra of $\hh^\eps$ by the two-sided ideal generated by $1 - \bm 1$. For each $n\in \ZZ_{\geq 0}$ define the left ideal $J_n := U_1(\hh^\eps) \cdot (\h \ox t^n\CC[t])$. 
That is, $J_n$ is the linear span of monomials of the form $b^i_p\dots b^j_q b^k_r$ with $r\geq n$.
The quotients by these ideals $U_1(\hh^\eps) \big/ J_n$ form an inverse system, whose inverse limit we shall denote by $\Uloc$: 
\be \Uloc := \varprojlim_n U_1(\hh^\eps) \big/ J_n\nn\ee
It is a complete topological algebra. We call it the algebra of fields, or the local completion of $U_1(\hh^\eps)$. 
By definition, elements of $\Uloc$ are possibly infinite sums $\sum_{m\geq 0} X_m$ of elements $X_m\in U_1(\hh^\eps)$ which truncate to finite sums when one works modulo any $J_n$, \emph{i.e.} for every $n$, $X_m\in J_n$ for all sufficiently large $m$.

A module $\M$ over $\hh^\eps$ is \emph{smooth} if for all $a \in \h$ and all $v\in \M$, $a_nv=0$ for all sufficiently large $n$. For example the Fock modules $\pi_\lambda^\eps$ are smooth.  Any smooth module over $\hh^\eps$ on which $\bm 1$ acts as the identity is also a module over $\Uloc$.  

\subsection{Vertex algebra}\label{sec: vas}
For every $n\in \ZZ$, there is a linear map \be \pi_0^\eps \to \Uloc;\quad A\mapsto A_{(n)}\nn\ee sending any given state $A$ to its \emph{$n$th formal mode}, $A_{(n)}$. One can arrange these modes into  a formal series called the \emph{formal state-field map},
\be Y[A, x] := \sum_{n\in \ZZ} A_{(n)} x^{-n-1} \label{formalYmap},\ee
and they are defined recursively as follows. First, for all $a\in\h$, 
\be Y[a_{-1} \vac, u] :=\sum_{n\in \ZZ} a_n u^{-n-1}.\nn\ee 
Next, let $[T,\cdot]$ be the derivation on $U_1(\hh^\eps)$ given by $[T, a_n] := -n a_{n-1}$ and $[T, 1] := 0$; then by setting $T(X \vac) := [T, X] \vac$ for any $X \in U_1(\hh^\eps)$, one can regard $T$ also as a linear map  $\pi_0^\eps\to \pi_0^\eps$, the \emph{translation operator}. 
For any $A,B\in \pi_0^\eps$, one sets
\be Y[TA, u] := \del_u Y[A, u]\quad\quad\text{and}\quad\quad  
Y[A_{(-1)} B, u] := \!\!\!\!\!\quad \nord{Y[A, u] Y[B, u]} \nn\ee
where the \emph{normal ordered product} $\nord{Y[A, u] Y[B, u]}$ is given by
\be \nord{Y[A,u] Y[B,u]} \,\,\, := \Bigg(\sum_{m<0} A_{(m)} u^{-m-1} \Bigg)Y[B,u] + Y[B,u] \Bigg(\sum_{m\geq 0} A_{(m)} u^{-m-1}\Bigg).\label{nord}\ee 

Recall that $\pi_\lambda^\eps$ are modules over $\Uloc$. By sending each formal mode $A_{(n)}$ to its image in $\End(\pi_\lambda^\eps)$ we get the \emph{module maps}
\be  Y_{\pi^\eps_\lambda}(\cdot, x) : \pi_0^\eps \to \Hom\left(\pi_\lambda^\eps, \pi_\lambda^\eps((x))\right) \nn\ee
and as a special case the \emph{state-field map}
\be Y(\cdot,x) : \pi_0^\eps \to \Hom\left(\pi_0^\eps, \pi_0^\eps((x))\right) \nn\ee
These maps $Y(\cdot, x)$ and $Y_{\pi_\lambda^\eps}(\cdot,x)$ make $\pi_0^\eps$ into a \emph{vertex algebra} and the $\pi_\lambda^\eps$ into modules over this vertex algebra. The reader is referred to  \cite[\S4-5]{FrenkelBenZvi} for the definitions (see especially \S4.3.1 and \S5.1.5). 

\subsection{The Lie algebras $\F^\eps$, $\F^\eps_0$, $\F^\eps_{\geq 0}$}\label{sec: bla}
The algebra of fields $\Uloc$ is in particular a Lie algebra, with the Lie bracket given by the commutator 
\be [A,B] := AB-BA.\nn\ee 
It has $\hh^\eps$ as a Lie subalgebra. It also has a larger Lie subalgebra of interest. Let 
\be \F^\eps \equiv \Lie(\pi_0^\eps) := \Span_\CC\left\{ A_{(n)}: A\in \pi_0^\eps, n\in \ZZ\right\} \subset \Uloc \nn\ee
denote the linear subspace of $\Uloc$ spanned by all formal modes of all states in $\pi_0^\eps$. This is a Lie subalgebra because one has the commutator formula,
\begin{equation} \label{com Am Bn}
\big[ A_{(m)}, B_{(n)} \big] = \sum_{k \geq 0} \binom{m}{k} \big( A_{(k)} B \big)_{(m + n - k)}.
\end{equation}
Here $\binom mk$ is defined for all $m\in \ZZ$ and $k\geq 0$ as follows: 
\be \binom{m}{k} := \frac{m(m-1)\dots (m-k+1)}{k!},\quad k\neq 0,\qquad \binom m 0 := 0 \nn\ee

The Lie algebra $\F^\eps$ has Lie subalgebras $\F^\eps_0$ and $\F^\eps_{\geq 0}$ spanned by the formal zero modes and all formal  non-negative modes, respectively, of all states in $\pi_0^\eps$. 

\begin{rem}
$\F^\eps$ can be defined abstractly, but we regard it as a subalgebra of $\Uloc$. There is no loss in this, cf. \cite[Lemma 4.3.2]{FrenkelBenZvi}. 
\end{rem}

\subsection{Conformal vectors}\label{sec: cv}
For any $\sv\in \h$ depending polynomially on $\eps$, define the vector $\confvec_\sv\in \pi_0^\eps$ by 
\be \confvec_\sv := \frac 1 \eps \left(\half b^i_{-1} b_{i,-1} + \sv_{-2} \right)\vac. \nn\ee
Here and henceforth we employ summation convention on the index $i$. 

The vector $\confvec_\sv$ is a conformal vector and defines on $\pi_0^\eps$ the structure of a \emph{conformal vertex algebra} (or \emph{vertex operator algebra}) with central charge 
\be c = \eps \dim \h - 12\bilin\sv\sv.\nn\ee 
That is to say, the vector $\confvec_\sv$ obeys 
$(\confvec_\sv)_{(0)} \confvec_\sv = T\confvec_\sv$, $(\confvec_\sv)_{(1)} \confvec_\sv = 2 \confvec_\sv$, $(\confvec_\sv)_{(2)} \confvec_\sv = 0$, 
\begin{align} (\confvec_\sv)_{(3)} \confvec_\sv 
&= \frac 1 \eps \left( \half b^j_1 b_{j,1} + b^j_0 b_{j,2} + \dots - 3 \sv_2 \right) \left( \half b^i_{-1} b_{i,-1} + \sv_{-2} \right)\vac\nn\\
&= \frac 1 \eps \left( \half \eps^2 \dim \h - 6 \eps \bilin \sv\sv\right)\vac  = \frac c 2 \vac,\nn\end{align} 
and $(\confvec_\sv)_{(n)} \confvec_\sv =0$ for $n\geq 4$.

The corresponding generators  $L_n \in \Uloc$ of the Virasoro algebra are defined by  $Y(\confvec_\sv,x) = \sum_{n\in \ZZ} L_n x^{-n-2}$, i.e.
\begin{align} L_n  &:=  (\confvec_\sv)_{(n+1)} 
= \frac 1 {2\eps}\left(\sum_{m<0} b^i_{m} b_{i,n-m} + \sum_{m\geq 0} b_{i,n-m} b^i_m\right) - \frac{n+1}{\eps} \sv_n,\nn\end{align} 
and they obey
\be \left[ L_n,L_m \right] = (n-m) L_{n+m} + \frac{n^3-n}{12} \delta_{n,-m} c .\nn\ee

For any $a\in \h$, the state $a:=a_{-1} \vac\in \pi_0^\eps$ obeys $(\confvec_\sv)_{(0)} a = T a$,  $(\confvec_\sv)_{(1)} a = a$, and 
\be (\confvec_\sv)_{(2)} a =
\frac 1 \eps\left( b^j_0 b_{j,1} + b^j_{-1} b_{j,2} + \dots - 2 \sv_1 \right) a_{-1} \vac
= -2 \bilin \sv a \vac .\nn\ee
Thus, the state $a = a_{-1} \vac\in \pi_0^\eps$ is primary precisely if $\bilin \sv a = 0$. 

\subsection{Actions of $L_0$ and $L_{-1}=T$}\label{sec: LL}
The Virasoro zero mode $L_0=(\confvec_\sv)_{(1)}$ obeys
\be [L_0, b^j_{-m}] = m b^j_m,\qquad L_0 \vacl =  \vacl \Delta_\lambda \nn\ee
where
\be \Delta_\lambda:= - \bilin \sv \lambda + \eps \bilin \lambda\lambda .\label{Deltadef}\ee   
The space $\pi_\lambda^\eps$ is $\ZZ$-graded, with the vacuum $\vacl$ of grade $0$ and the generator $b^j_{-n}$ contributing $+n$ to the grade. For any vector $v\in \pi_\lambda^\eps$ of grade $m$ we see that
\be L_0 v = v (m + \Delta_\lambda).\nn\ee

For any state $A\in \pi_0^\eps$, we have, irrespective of the value of $\sv$,
\be L_{-1} A = (\confvec_\sv)_{(0)} A =  TA. \label{omT}\ee
Indeed, $L_{-1} = \frac 1 \eps\sum_{n\geq 0} b^i_{-n-1} b_{i,n}$ and hence $[L_{-1}, b^j_m] = - m b^j_{m-1}$, 
and $L_{-1} \vac=0$. More generally
\be L_{-1} \vacl = (\confvec_\sv)_{(0)} \vacl = \eps^{-1} b^i_{-1} b_{i,0} \vacl = b^i_{-1} \vacl \bilin\lambda {b_i}= \lambda_{-1} \vacl.\label{Lm1vacl}\ee

\subsection{Intertwiners and screenings}\label{sec: intertwiners}
In addition to the state-field map $Y$ and module maps $Y_{\pi_\lambda^\eps}$, there is also another structure, a linear map
\be V_\lambda(\cdot,x) : \pi_\lambda^\eps \to \Hom\left(\pi_0^\eps, \pi_\lambda^\eps((x))\right) \label{im}\ee
called the intertwining map,
\be V_\lambda(m,x) = \sum_n m_{(n)} x^{-n-1}.\nn\ee
We can take as its definition the following:
for any $m\in \pi_\lambda^\eps$ and $B \in \pi_0^\eps$, 
\be V_\lambda(m, x) B := e^{x L_{-1}} Y_{\pi_\lambda^\eps}(B, -x) m.\nn\ee 
In view of \cref{omT}, this expression can be motivated by comparison with the usual skew-symmetry property of the state-field map $Y$: for any $A, B \in \pi_0^\eps$, 
\be Y(A, x) B = e^{x T} Y(B, -x) A,\nn\ee
or equivalently
\be \label{vaskew}
A_{(n)} B = - \sum_{k \geq 0} \frac{(-1)^{n + k}}{k!} T^k \big( B_{(n + k)} A \big), \qquad n\in \ZZ.
\ee 
Just as the state-field map obeys the equation $\del_x Y(A,x) = Y(TA,x)$, the intertwiner map obeys an analogous equation, namely,
\be \del_x V_\lambda(m,x) = V_\lambda( L_{-1} m,x) .\nn\ee
We are particularly interested in the intertwining maps defined by the vacuum states $\vacl\in \pi_\lambda^\eps$, and for those there is the following explicit expression, which, formally, can be obtained by solving the equation above:
\begin{align}  V_\lambda(\vacl,x) = V_\lambda(x) &:= \shift_\lambda \exp\left( -  \bilin\lambda{b_i} \sum_{n<0} \frac{b^i_n x^{-n}}{n} \right)
\exp\left( - \bilin\lambda{b_i}  \sum_{n>0} \frac{b^i_n x^{-n}}{n} \right)\nn\\
   &=\shift_\lambda\exp\left( - \sum_{n<0} \frac{\lambda_n x^{-n}}{n} \right)
\exp\left( - \sum_{n>0} \frac{\lambda_n x^{-n}}{n} \right)\label{Vf}\end{align}
where $\shift_\lambda: \pi_0\to \pi_\lambda$ is the shift operator which sends $\vac\mapsto \vacl$ and commutes with $b^i_{n}$, $n\neq 0$.

Now define the \emph{screening operator} 
\be S_\lambda: \pi_0 \to \pi_\lambda\nn\ee 
to be the linear map given by the zero mode of the vacuum state $\vacl\in \pi_\lambda^\eps$,
\be S_\lambda A := \ket\lambda_{(0)} A.\nn\ee 

The interwiner maps obey natural analogs of Borcherds identity and its corollaries. (See e.g. \cite[\S7.2.1]{Fre07} and \cite[\S5.2]{FrenkelBenZvi}.) The property we shall need is the following.
\begin{lem}\label{scderlem} For all $m\in \pi_\lambda^\eps$, all $A,B\in \pi_0^\eps$, and all $k\in \ZZ$,
\be m_{(0)} (A_{(k)} B) = A_{(k)} ( m_{(0)} B) + (m_{(0)} A)_{(k)} B .\nn\ee 
\qed\end{lem}
\begin{cor}\label{kercl} For all $\lambda\in\h$, the kernel $$\ker S_\lambda\subset \pi_0^\eps$$ of the screening $S_\lambda$ is a vertex subalgebra of $\pi_0^\eps$. 
\end{cor}
\begin{proof} \Cref{scderlem} implies that $\ker S_\lambda$ is closed for all the $n$th products. From \cref{Vf} it is clear that  $\vac\in \ker S_\lambda$. For any $A\in \pi_0^\eps$ we have $TA = A_{(-2)} \vac$ (this is true in any vertex algebra). Hence if $(\vacl)_{(0)} A = 0$ then $(\vacl)_{(0)} TA = (\vacl)_{(0)}(A_{(-2)} \vac) = 0$, again by \cref{scderlem}. So $\ker S_\lambda$ is closed under the action of the translation operator $T$.  
\end{proof}

\begin{cor}\label{intcom} If $X\in \ker S_{\lambda}\subset \pi^\eps_0$ then the diagram
\be
\begin{tikzcd}
\pi^\eps_0\rar{S_\lambda} & \pi^\eps_{\lambda} \\
\pi^\eps_0\rar{S_\lambda}\uar{X_{(n)}} & \pi^\eps_{\lambda}\uar{X_{(n)}} 
\end{tikzcd}
\nn\ee
commutes for all $n\in \ZZ$, $\lambda\in \h$.
\end{cor}
\begin{proof} 
By \cref{scderlem} we have
\be [S_{\lambda}, X_{(n)} ] = [ \ket{\lambda}_{(0)}, X_{(n)} ] 
= \left(\lambda_{(0)} X \right)_{(n)} = (S_{\lambda} X)_{(n)} = 0.\nn\ee
\end{proof}

\subsection{Action of $S_\lambda$ on $\confvec_\sv$}\label{Sonom}
We shall need the action of the screening $S_\lambda$ on the conformal vector $\confvec_\sv\in \pi_0^\eps$. 
\begin{lem}\label{Somlem}
  \be S_\lambda \confvec_\sv =  \lambda_{-1} \vacl \left( -1 - \bilin \lambda \sv 
                     +  \half\eps \bilin{\lambda}\lambda \right) .\nn\ee
\end{lem}
\begin{proof}
Consider the formula \cref{Vf} for the intertwiner $V_\lambda(x)$. We have 
\be \exp\left( - \sum_{n>0} \frac{\lambda_n x^{-n}}{n} \right) 
= 1 -  \lambda_1 x^{-1} +  \half \left(-\lambda_2 + \lambda_1 \lambda_1\right)x^{-2} + \dots,\nn\ee
where $\dots$ are terms of higher order in $x^{-1}$ all of which kill $\confvec_\sv$ on grading grounds, and 
\be \exp\left( - \sum_{n<0} \frac{\lambda_n x^{-n}}{n} \right) 
= 1 +  \lambda_{-1} x 
+  \dots ,\nn\ee
where $\dots$ are terms of higher order in $x$. 
Recall that $\shift_{\lambda}^{-1}S_{\lambda}$ is the residue of $\shift_{\lambda}^{-1} V_{\lambda}(x)$. 
Thus, acting on $\pi_0^\eps$, 
\be \shift_{\lambda}^{-1}S_{\lambda}  =  - \lambda_{1} 
+\half \lambda_{-1} \left( -\lambda_2 
               +  \lambda_1 \lambda_1   \right) + \dots
\nn\ee
where $\dots$ are terms that kill $\confvec_\sv$, and so we obtain
\begin{align} \shift_{\lambda}^{-1}S_{\lambda} \confvec_\sv 
&= 
                 - \bilin{\lambda}{b_k} b^k_{-1} \vac 
                 + \lambda_{-1} \vac  
                     \left( - \bilin \lambda \sv 
                     +  \half\eps \bilin{\lambda}{b_k}\bilin{\lambda}{b^k}\right) \nn
\end{align}
and hence the result.
\end{proof}

\section{Classical limit}\label{sec: cls}
Having introduced the vertex algebra $\pi_0^\eps$ in the previous section, we now consider the limit $\eps\to 0$ to recover the Fock module $\pi_0$ which actually appears in \cref{Wemb}. 

\subsection{Classical limits $\picc_0$ and $\Ulocc$}\label{sec: cl}
Let us write $\Ulocc$ for the $\eps\to 0$ limit of the algebra of fields $\Uloc$. It is a Poisson algebra, i.e. a unital commutative associative algebra which is also a Lie algebra, with the compatibility condition that the Lie product (i.e. the Poisson bracket) is a derivation in both slots for the commutative product. The Poisson bracket on $\mc F$ is given by 
\be \{\cdot,\cdot\} := \lim_{\eps\to 0} \frac 1 \eps [\cdot,\cdot].\nn\ee 
Equivalently it is defined by its action on the generators,
\be \left\{ b^i_m, b^j_n \right\} =  1 m \delta_{m+n,0} \bilin{b^i}{b^j}, \nn\ee
cf. \cref{hcom}, together with the derivation condition.

The limit $\eps \to 0$ of the one-parameter family $\pi_0^\eps$ of vertex algebras is a commutative vertex algebra, which means one in which all non-negative products vanish. We shall denote it by $\picc_0$.
A commutative vertex algebra is equivalent to a \emph{differential algebra}; that is, a unital commutative associative algebra equipped with a derivation. Indeed, when all non-negative products vanish then the product defined by $A\cdot B := A_{(-1)} B$ is commutative by \cref{vaskew} and associative by Borcherds identity. The map $\del := T$ is a derivation and $1:= \vac$ is the unit element. In the present case we have
\be \picc_0 \cong \CC[ b_{i,n} ]_{i=1,\dots,\dim\h; n\in \ZZ_{<0}}, \label{picdef}\ee
as differential algebras, the latter equipped with the derivation $\del$ defined by $\del b_{i,n} = -n b_{i,n-1}$.
From its origin as the limit $\eps\to 0$ of a one-parameter family of vertex algebras, $\picc_0$ comes endowed with the structure of a vertex Poisson algebra. 
A \emph{vertex Poisson algebra} $\vpa$ is differential algebra $(\vpa,\cdot,\del,1)$ which is also a \emph{vertex Lie algebra} $(\vpa,T,\{_{\{n\}}\}_{n\in \ZZ_{\geq 0}})$, with the compatibility conditions that $T=\del$ and that the non-negative products are all derivations of the commutative product:
\be A_{\{n\}} (B\cdot C) = (A_{\{n\}} B) \cdot C + B \cdot (A_{\{n\}} C) \nn\ee
for all $n \geq 0$ and all $A,B,C\in \vpa$ \cite{BDChiralAlgebras,FrenkelBenZvi,Yamskulna}. We get this vertex Lie algebra structure, i.e. these non-negative products, by taking the order-$\eps$ term in the limits of the non-negative products on $\pi_0^\eps$:
\be A_{\{n\}} B := \lim_{\eps\to 0} \frac 1 \eps A_{(n)} B ,\nn\ee 
for all $n\geq 0$ and all $A,B\in \ol \pi_0$. 

\begin{rem}\label{identrem}
Here and below we identify $\picc_0$ and all $\pi_0^\eps$ as vector spaces in the obvious way, i.e. by the vector space isomorphisms $\pi_0^\eps \cong \CC[b_{i,n}]_{i=1,\dots,\dim\h, n<0}$. 
\end{rem}

There is a classical analog of the formal state-field map \cref{formalYmap}: for every $n\in \ZZ$ there is a linear map $\picc_0\to \Ulocc; A \mapsto \ol A_{(n)}$ sending $A\in \picc_0$ to its \emph{$n$th classical formal mode} $\ol A_{(n)}$. These modes can be arranged into a formal power series
\be \ol Y[A, x] := \sum_{n\in \ZZ} \ol A_{(n)} x^{-n-1} \label{formalYmapclass}\ee
and are defined by 
$\ol Y[a_{-1},x] = \sum_{n\in \ZZ} a_n x^{-n-1}$, $\ol Y[A\cdot B,x] = \ol Y[A,x] \ol Y[B,x]$, and $\ol Y[TA,x] = \del_x \ol Y[A,x]$. By analogy with the big Lie algebra $\F^\eps = \Lie(\pi_0^\eps)\subset \Uloc$ of \S\ref{sec: bla}, we have the linear subspace 
\be \Fc \equiv \Lie(\picc_0) := \Span_\CC\left\{ \ol A_{(n)}: A\in \picc_0, n\in \ZZ\right\} \subset \Ulocc \nn\ee
spanned by all classical formal modes of all elements of $\picc_0$. It is a Lie subalgebra of $\Ulocc$ with respect to the Poisson bracket. Indeed, notice that the commutator formula \cref{com Am Bn} depends only on the vertex Lie algebra structure of $\pi_0^\eps$, i.e. on the non-negative products. Here we have
\begin{equation} \label{class com Am Bn}
\big\{ \ol A_{(m)}, \ol B_{(n)} \big\} = \sum_{k \geq 0} \binom{m}{k} \ol{\big( A_{\{k\}} B\big) }_{(m + n - k)}.
\end{equation}
Define also the Lie subalgebras $\Fc_0$ and $\Fc_{\geq 0}$ of zero and of non-negative modes respectively, cf. \S\ref{sec: bla}. 

\begin{rem} What we call $\Fc$ is called $\hat{\mc F_0}$ in \cite{FFIoM}. It is the space of \emph{local functionals} (on $\h^*_1$).
\end{rem}

Finally recall that $\pi^\eps_\lambda$ (and in particular $\pi_0^\eps$) is a module over $\Uloc$.
One can ask what structures result from this in the $\eps\to 0$ limit.
First, $\picc_\lambda$ is a module over $\Ulocc$, the latter regarded as a commutative algebra. For this action, all non-negative modes  simply act as zero on all of $\picc_\lambda$. But $\pi^\eps_\lambda$ is also a module over the Lie algebra $\F^\eps\subset \Uloc$ and in particular over its Lie subalgebra $\F^\eps_{\geq 0}$. In the limit we get an action, call it $\lon$, of the Lie algebra $\Fc_{\geq 0}$ on $\picc_\lambda$:
\be \ol A_{(n)} \lon v = \lim_{\eps\to 0} \frac 1 \eps A_{(n)}  v\nn\ee
We can write any state in $\picc_\lambda$ as $X\vacl$ for some $X\in \Ulocc$. One then has 
\be \ol A_{(n)} \lon X \vacl = \left\{ \ol A_{(n)}, X\right\} \vacl + X\, (A_{(n)}\lon \vacl) .\label{olact}\ee
On the highest weight vector $\vacl$, one has
\be a_0 \lon \vacl = \vacl \bilin a \lambda,\qquad  a_n \lon \vacl =0 , \quad n\geq 1,\, a\in \h. \nn\ee
(We should stress that this is the Lie algebra action of $a_0\in \Fc_{\geq 0}$. As an element of the commutative algebra $\Ulocc$, $a_0  \vacl= \lim_{\eps\to 0} \eps \bilin a \lambda \vacl =0$.)

\subsection{$\picc_0$ as a conformal algebra}\label{sec: qc}
The conformal vector $\confvec_\sv\in \pi_0^\eps$ is divergent in the limit $\eps\to 0$. However, $\eps \confvec_\sv$ has a well-defined limit (assuming $s\in \h[\eps]$) which we shall denote by $\ol\confvec_\sv$:
\be \ol\confvec_\sv = \lim_{\eps\to 0} \eps\confvec_\sv .\nn\ee
It obeys
$(\ol\confvec_\sv)_{\{0\}} \ol\confvec_\sv = \del\ol\confvec_\sv$, $(\ol\confvec_\sv)_{\{1\}} \ol\confvec_\sv = 2 \ol\confvec_\sv$, $(\ol\confvec_\sv)_{\{2\}} \ol\confvec_\sv = 0$,
\begin{align} (\ol\confvec_\sv)_{\{3\}} \ol\confvec_\sv &= \lim_{\eps\to 0}\frac 1 \eps (\eps\confvec_\sv)_{(3)} \eps \ol\confvec_\sv\nn\\  
&= \lim_{\eps\to 0}\frac 1 \eps\left( \half b^j_1 b_{j,1} + b^j_0 b_{j,2} + \dots - 3 \sv_2 \right) \left( \half b^i_{-1} b_{i,-1} + \sv_{-2} \right)\vac\nn\\
&= \lim_{\eps\to 0}\frac 1 \eps \left( \half \eps^2 \dim \h - 6 \eps \bilin \sv\sv\right)\vac \nn\\ &= 6 \bilin \sv\sv  \vac,\nn\end{align} 
and $(\ol\confvec_\sv)_{\{n\}} \ol\confvec_\sv =0$ for $n\geq 4$. 
These relations say that, as a vertex Lie algebra, $\picc_0$ contains a subalgebra isomorphic to the \emph{Virasoro vertex Lie algebra} with central charge 
\be \ol c = 12 \bilin \sv\sv .\label{cclas}\ee
Such a vertex Lie algebra is called a \emph{conformal algebra} with central charge $\ol c$ (cf. \cite[16.1.14]{FrenkelBenZvi}).
Note that the $\dim \h$ term does not survive in the limit, which is as expected since it arises from a double contraction, i.e. from two uses of the commutator \cref{hcom}.

\subsection{Classical screenings $\ol S_\lambda$}
Recall the screenings $S_\lambda: \pi_0 \to \pi_\lambda$ from \S\ref{sec: intertwiners}. Now we consider their classical limits.
By definition, the \emph{classical screening operator} \cite{FFIoM}
\be \ol S_\lambda : \picc_0 \to \picc_\lambda\nn\ee
is the map given by
\be \ol S_\lambda = \lim_{\eps\to 0} \frac 1 \eps S_\lambda.\nn\ee

Recall the  explicit form \cref{Vf} of the intertwiner map:
\be V_\lambda(x) := \exp\left( - \sum_{n<0} \frac{\lambda_n x^{-n}}{n} \right)
\exp\left( - \sum_{n>0} \frac{\lambda_n x^{-n}}{n} \right).\label{Vf2}\ee
Consider the left exponential factor here. It contains lowering modes, $\lambda_{n}$, $n<0$. Such modes just act by multiplication on $\pi_0^\eps \cong \CC[b_{i,n}]_{i=1,\dots,\dim\h, n<0}$.  Define polynomials $V_{\lambda}[n]$ in the elements $\lambda_m$, $m\leq 0$, by
\be \sum_{n\leq 0} V_\lambda[n] x^{-n} = \exp\left( -  \sum_{n<0} \frac{\lambda_n x^{-n}}{n} \right) ,\label{Vdef}\ee
so that left exponential factor is $ \sum_{n\leq 0} V_\lambda[n] x^{-n}$.  

Now consider the action of the right exponential factor on $\pi_0^\eps$. 
The action of $b^i_n$, $n\geq 0$, on $\pi_0^\eps$ is given by $\eps n\frac{\del}{\del b_{i,-n}}$, for indeed
\be \left[ \eps n\frac{\del}{\del b_{i,-n}}, b_{j,m}\right]
     = \eps n \delta^i_j \delta_{n+m,0} = \left[ b^{i}_{n}, b_{j,m}\right]. \nn\ee
Thus, on $\pi_0^\eps$,
\be \exp\left(- \sum_{n>0} \frac{\lambda_{n} x^{-n}}{n}\right) 
= 1 - \eps  \sum_{n>0} x^{-n} \bilin\lambda {b^i} \frac{\del}{\del b_{i,-n}} + \O(\eps^2).\nn\ee  
Hence, on $\pi_0^\eps$, the polar part in $x$ of $V_{\lambda}(x)$ is given by 
\be -\eps \sum_{m\leq 0} \sum_{n>m} x^{-n} V_\lambda[m] \bilin\lambda{b_i} \frac{\del}{\del b_{k,-n+m}} + \O(\eps^2).\nn\ee
In particular, the leading term of the residue is $\eps \shift_{\lambda}^{-1} \ol S_{\lambda}$ where
\be \ol S_\lambda : \ol \pi_0 \to \ol \pi_\lambda\nn\ee
are the linear maps defined by
\begin{align} \ol S_\lambda
&:= - \shift_\lambda \sum_{m\leq 0} V_\lambda[m] \bilin \lambda{b_k} \frac{\del}{\del b_{k,-1+m}}. \label{csd}\end{align}
Call these \emph{classical screening operators}.
By construction, we have the following.
\begin{lem}\label{clcor} If $m\in \pi_0^\eps$ lies in the kernel of $S_\lambda$ for all $\eps$ then $m\in \picc_0$ lies in the kernel of $\ol S_\lambda$.\qed
\end{lem}

\begin{lem}\label{clkercl} Take any $\lambda\in \h$. The kernel $$\ker\ol S_\lambda\subset \picc_0$$ of the classical screening $\ol S_\lambda$ is a vertex Poisson subalgebra of $\picc_0$. 
\end{lem}
\begin{proof} Suppose $A,B\in \ker\ol S_\lambda$ and consider $A_{\{n\}} B$ for $n \geq 0$. By definition $A_{\{n\}} B$ is the order $\eps$ term in the product $A_{(n)} B$ of $A$ and $B$ regarded as elements of $\pi_0^\eps$; and hence  
$-\ol S_\lambda(A_{\{n\}}B)$ is the order $\eps^2$ term in $(\vacl)_{(0)}(A_{(n)} B)$. Now,  by \cref{scderlem} we have 
\be (\vacl)_{(0)}(A_{(n)} B) = \left( ((\vacl)_{(0)} A)_{(n)} B + A_{(n)} ((\vacl)_{(0)} B) \right). \nn\ee
But since $A,B\in \ker\ol S_\lambda$ we know that $(\vacl)_{(0)} A$ and $(v_{\lambda})_{(0)} B$ are both $\O(\eps^2)$, and hence both terms on the right here are $\O(\eps^3)$. So we have $\ol S_\lambda(A_{\{n\}}B)=0$, i.e. $\ker\ol S_\lambda$ is closed under the vertex Lie algebra products:
\be A_{\{n\}}B \in \ker\ol S_\lambda. \nn\ee
It is clear  $\vac\in \ker\ol S_\lambda$. And  $(\vacl)_{(0)} TA = (\vacl)_{(0)}(A_{(-2)} \vac) = ( (\vacl)_{(0)}A)_{(-2)} \vac$ is $\O(\eps^2)$ whenever $(\vacl)_{(0)}A$ is $\O(\eps^2)$. That is, $A\in \ker\ol S_\lambda$ implies $TA\in \ker\ol S_\lambda$. A similar argument shows that $\ker\ol S_\lambda$ is closed with respect to the commutative product, i.e. $A\cdot B \in \ker\ol S_\lambda$. 
\end{proof}

The proof of the following, starting from \cref{intcom}, is similar.
\begin{lem}\label{clintcom} Let $\lambda,\mu\in \h$. If $X\in \ker \ol S_{\lambda}\subset \picc_0$ then the following diagram commutes:
\be
\begin{tikzcd}
\picc_\mu     \rar{\ol S_\lambda} & \picc_{\mu+\lambda} \\
\picc_\mu\uar \rar{\ol S_\lambda} & \picc_{\mu+\lambda}\uar 
\end{tikzcd}
\nn\ee
where the vertical arrows can be either the action of the classical mode $X_{(n)}$ for any $n\in \ZZ$ (which is non-trivial only for $n\leq -1$) or the Lie algebra action $X_{(n)} \lon$ for any $n\geq 0$. 
\qed\end{lem}

\section{Coordinate transformations on $\picc_0$}\label{sec: ct}
Having defined $\pi_0$ and the classical screening operators, we now want to recall the definitions of the group $\Aut\O$, the Lie algebra $\Der\O$, and how they act on $\pi_0$. 

\subsection{Functions on the formal disc}\label{sec: auto}
Let $\O$ denote the complete topological $\CC$-algebra of formal power series in a variable $t$,
\be \O := \CC[[t]].\nn\ee 
One should think of $\O$ as the algebra of functions on a ``formal pointed disc'' $D$. It is a ``formal'' disc because we don't impose any convergence requirements on the functions beyond the demand that they be well-defined elements of $\CC[[t]]:= \invlim_n \CC[t]/t^n\CC[t]$, i.e. that they converge in the $t$-adic topology. And it is ``pointed'' because $\mc O$ has a unique maximal ideal $\m:= t\CC[[t]]$, which one thinks of as consisting of the functions that vanish at the ``point'' with coordinate $t=0$. 

\begin{rem}\label{rem: holomorphic} We work in the formal setting, but one could also work in the analytic setting, with $\O$ replaced by the algebra of germs of analytic functions at a point $p$ on a Riemann surface.
This amounts to imposing the additional requirement that each power series $f$ is convergent in the analytic topology on some sufficiently small (depending on $f$) neighbourhood of $p$.
\end{rem}

Let $\Aut\O$ be the group of continuous automorphisms of $\O$. Elements $\mu\in \Aut\O$ are changes of coordinate of the form
\be t = \mu(s) = c_1 s + c_2 s^2 + \dots \label{mudef}\ee
with $c_1\in \CCx$ and $c_n \in \CC$ for all $n\geq 2$, and the group operation is  
\be (\mu_1 \ast \mu_2)(s) = \mu_2(\mu_1(s)).\label{astdef}\ee
(Note the order here.)

The derivative
\be \mu'(s) = c_1 + 2c_2 s + \dots \nn\ee
of $\mu\in \Aut\O$ belongs to the group of units (i.e. invertible elements) of the ring $\O$. Denote the latter by $\U(\O)$. (It is the complement of the maximal ideal $\m$ in the local ring $\O$.)

Let $\Der\O$ denote the Lie algebra of continuous derivations of $\O= \CC[[t]]$,
\be \Der \O = \CC[[t]] \del_t.\nn\ee 
It has topological basis $L_n := - t^{n+1} \del_t$, $n\geq -1$.

\subsection{Quasi-conformal structure on $\picc_0$}\label{sec: qcs}
A commutative vertex algebra is called \emph{quasi-conformal} if it carries an action of $\Der\O$ such that the generator $L_{-1} = -\del_t$ acts as the translation operator $T$, the generator $L_0 = -t\del_t$ acts semisimply with integer eigenvalues, and the Lie subalgebra $\Der_+\O := t^2 \CC[[t]]\del_t$ acts locally nilpotently.\footnote{For general vertex algebras there is one extra condition, but this condition follows automatically from those listed in the case of commutative vertex algebras. See  \cite[\S6.3.5]{FrenkelBenZvi}.}

Recall the conformal algebra structure on $\picc_0$ from \S\ref{sec: qc}.
This structure makes $\picc_0$ into a quasi-conformal commutative vertex algebra. 
For indeed, one has generators  $\ol L_n \in \Fc \subset \Ulocc$ of a copy of the Virasoro algebra at central charge $\ol c$, defined by  $Y[\ol\confvec_\sv,x] = \sum_{n\in \ZZ} \ol L_n x^{-n-2}$, i.e. in our case
\begin{align} \ol L_n  
= \half\sum_{m\in \ZZ} b^i_{m} b_{i,n-m} - (n+1) \sv_n.\nn\end{align}
Their Poisson brackets with the generators $b^i_m$ are given by
\begin{subequations}\label{cld}
\be \left\{\ol L_n, b^j_{p} \right\} = - p b^j_{n+p} - n(n+1) \delta_{n+p,0} \bilin \sv {b^j} .\ee
We get an action of $\Der \O$ on $\picc_\lambda$ given by 
\be L_n \on m := \ol L_n \lon m \equiv (\ol\confvec)_{(n+1)} \lon m\label{derO} \ee
for all $n\geq -1$ and $m\in \picc_\lambda$.
Explicitly, $L_n\on \vacl = 0$ for $n\geq 1$,
\be L_0 \on \vacl = - \bilin\sv\lambda \vacl,\quad
 L_{-1} \on \vacl 
= \lambda_{-1} \vacl \nn\ee
together with, for all $m \in \picc_0$ and $n\geq -1$, 
\be L_n \on \left(b^j_{p} m\right) = \left\{\ol L_n, b^j_{p} \right\} m + b^j_{p} (L_n \on m) \ee 
(cf. \S\ref{sec: LL} and \cref{olact}).

In particular, we get an action of $\Der \O$ on $\picc_0$, given by the formulas 
\be
L_n \on b^j_{p} = 
\begin{cases} 
-p b^j_{n+p} & n < -p \\ 
-n(n+1) \bilin \sv {b^j} & n=-p \\ 
0 & n> -p \end{cases}
\ee
\end{subequations}
for $p\leq -1$ and $n\geq -1$. This action obeys the remaining conditions to make $\picc_0$ into a quasi-conformal vertex algebra.
(In particular, every eigenvalue of $L_0$ is an integer because it is nothing but the grade in the $\ZZ$-gradation introduced above; and on grading grounds, for every $m\in \picc_0$, $L_n\on m = 0$ for all sufficiently large $n$.)

Consider the Lie subalgebra of $\Der\O$, 
\be \Der_0\O := t\CC[[t]]\del_t, \nn\ee
consisting of those derivations which preserve the maximal ideal $\m = t\CC[[t]]$ of $\O$. It is the Lie algebra of $\Aut\O$. 
The defining conditions of a quasi-conformal vertex algebra ensure in particular that the action of $\Der_0\O\subset \Der\O$ can be exponentiated up to give an action of $\Aut\O$.
Later we will need explicit expressions for this action of $\Aut\O$ on $\picc_0$. To write them, let $\mu \in \Aut\O$ be the element in \cref{mudef},
\be \mu(s) = c_1 s + c_2 s^2 + \dots .\nn\ee
It can be written in the form 
\be \mu(s) = \exp\left(\sum_{n>0} v_n s^{n+1} \del_s\right) v_0^{s\del_s} \on s\nn\ee
for certain $v_n\in \CC$, $n>0$, uniquely defined by the coefficients $c_n$, $n\geq 1$. (One has  $c_1 = v_0$, $c_2 = v_0v_1$, $\dots$.) If one then lets
\be R(\mu) := \exp\left(-\sum_{n>0} v_n L_n\right) v_0^{-L_0}\label{Rsv}\ee
then the map $\mu \mapsto R(\mu)$ defines (see e.g. \cite[Lemma 6.3.2]{FrenkelBenZvi}) an action of $\Aut\O$ on $\picc_0$. 
We shall need the following in \cref{sec: pisvleq} below: for all $a\in \h$,
\begin{align} R(\mu) \on a_{-1}\vac &= (1 - v_1 L_1) v_0^{-1} a_{-1}\vac\nn \\
&= v_0^{-1} a_{-1}\vac + 2 v_1 v_0^{-1} \bilin \sv {a} \vac  \nn\\
&= \frac 1{\mu'(0)} \left(a_{-1}\vac  + \bilin{\sv}{a} \frac{\mu''(0)}{\mu'(0)} \vac \right). 
\label{btrans}\end{align}

\section{Cartan data}\label{sec: cartan}
Until this point $\h$ was merely a finite-dimensional vector space equipped with a symmetric non-degenerate bilinear form $\bilin\cdot\cdot$. But our interest is in the case in which $\h$ is a Cartan subalgebra of a Kac-Moody algebra $\g$ of affine type, and $\bilin\cdot\cdot$ is the restriction to $\h$ of an invariant symmetric bilinear form on $\g$. Let us recall the definitions from \cite{KacBook}. 

\subsection{} Suppose $\cart$ is an indecomposable Cartan matrix of affine type. 

Let $\g=\g(\cart)$ be the corresponding Kac-Moody algebra and $\h\subset \g$ a Cartan subalgebra.  Let $\{\al_i\}_{i\in I} \subset \h^*$ and $\{\chal_i\}_{i\in I} \subset \h$ be sets of simple roots and coroots of $\g$ respectively, where $I$ is an index set of cardinality the number of rows/columns of $\cart$. By definition, 
\be \la \al_i,\chal_j\ra = \cart_{ji},\nn\ee 
where $\la\cdot,\cdot\ra : \h^* \times \h \to \CC$ is the canonical pairing. 
Let $\lg= \g(\trcart)$ be the Langlands dual Lie algebra, i.e. the Kac-Moody Lie algebra with the transposed Cartan matrix $\trcart$. The Cartan subalgebra $\lh\subset\lg$ is naturally identified with the dual $\h^*$ of $\h$, $\lh = \h^*$ and $\{\al_i\}_{i\in I} \subset \h^* = \lh$ and $\{\chal_i\}_{i\in I} \subset \h = \lh^*$ are respectively sets of simple coroots and roots of $\lg$. 
The centres of $\g$ and $\lg$ are both of dimension one and are spanned by the elements
\be \cent = \sum_{i\in I} \chaaa_i \chal_i\in \h, \qquad \chcent = \sum_{i\in I}\aaa_i \al_i \in \lh\label{kddef}\ee
respectively, where $\{\chaaa_i\}_{i\in I}$ and $\{a_i\}_{i\in I}$ are the unique collections of relatively prime positive integers such that $\sum_{j\in I} \cart_{ij}\aaa_j  =0$ and $\sum_{j\in I} \chaaa_i \cart_{ij} = 0$. The dual Coxeter and Coxeter numbers of $\g$ are given respectively by
\be \dualcoxeter = \sum_{i\in I} \chaaa_i,\qquad \coxeter = \sum_{i\in I}\aaa_i. \nn\ee
We set
\be \vareps_i^{-1} = \chaaa_i\aaa_i^{-1}\label{epsa}\ee
and define $D^{-1} = \diag(\vareps^{-1}_i)_{i\in I}$. Then  $D^{-1} \cart$ is symmetric. 
We fix the symmetric nondegenerate bilinear form $\bilin\cdot\cdot$ on $\h$ defined by
\be \bilin{\chal_i} x := \la \al_i, x \ra \vareps_i \label{kdef}\ee
for $i\in I$ and $x\in \h$, together with the condition that 
\be \bilin yz = 0 \nn\ee 
for all $y,z$ belonging to a choice of complementary subspace $\h''$ in $\h$ to the subspace 
\be \h':= \sum_{i\in I} \CC\chal_i \subset \h \nn\ee
spanned by the simple coroots. It is the restriction of the \emph{standard} invariant symmetric bilinear form on $\g$ \cite{KacBook}. 
Since we assume $\g$ is of affine type, $\cart$ is of rank $|I|-1$ and $\h''$ has dimension one. We pick and fix some choice of $\h''$; the freedom here is not important because it is absorbed by the way we fix a derivation element in \cref{rhonull} below.
Let $\weyl \in \lh$ be the Weyl vector of $\g$, which is defined by the property that 
\be\la\weyl,\chal_i\ra = 1\nn\ee 
for each simple coroot $\chal_i$ of $\g$. 
It is unique up to the addition of an arbitrary multiple of $\chcent$.
Let $\chweyl \in \h$ be the Weyl vector of $\lg$, which is defined by the property that 
\be \la\al_i,\chweyl\ra = 1\nn\ee 
for each simple coroot $\al_i$ of $\lg$. 
It is unique up to the addition of an arbitrary multiple of $\cent$.
We fix these freedoms in $\weyl$ and $\chweyl$ by demanding that
\be\bilin{\chweyl}{\chweyl} = 0\quad\text{and}\quad \bilinvee\weyl\weyl = 0 \label{rhonull}\ee
where $\bilinvee\cdot\cdot$ is the induced nondegenerate bilinear form on $\h^*$.\footnote{One may check that $\bilinvee{\al_i}{x} = \la \chal_i , x\ra \vareps_i^{-1}$. Thus $\bilinvee\cdot\cdot$ is the restriction to $\lh=\h^*$ of the standard bilinear form on the Langlands dual algebra $\lg$, for which $a_i$ and $\check a_i$ are interchanged.}  
We note also that $\bilin{\chweyl}\cent 
= \coxeter$ and
$\bilinvee{\weyl}\chcent 
= \dualcoxeter$, and $\bilin\cent\cent = 0 =\bilinvee\chcent\chcent$.

Let $\{\chLa_i\}_{i=0}^\ell \subset \h$ and $\{\La_i\}_{i=0}^\ell \subset \lh$ be the unique elements such that 
\be \la \weyl, \chLa_i \ra = 0 , \qquad \la \La_i, \chweyl \ra = 0 \nn\ee
and 
\be 
  \la \La_i, \chal_j \ra = \delta_{ij},\qquad
  \la \al_j, \chLa_i \ra = \delta_{ij},
\qquad i,j\in I.\label{ld}
\ee
They are a choice\footnote{The choice is in the demand that $\la \La_i, \chweyl \ra = 0$ rather than e.g. $\la \La_i, \cocent\ra=0$ with $\cocent$ the derivation element corresponding to the homogeneous gradation.} of fundamental coweights of $\g$ and $\lg$ respectively.

\subsection{Principal gradation of $\g$}\label{sec: pg}
Let us introduce a Cartan decomposition
\be \g = \n_- \oplus \h \oplus \n_+ \nn\ee
of $\g$ and Chevalley-Serre generators $e_i\in \ln_+$, $f_i\in \ln_-$, $i\in I$. These latter obey
\begin{subequations} \label{KM relations}
\begin{alignat}{2}
\label{KM rel a} [x, e_i] &= \langle \al_i,x\rangle  e_i, &\qquad
[x, f_i] &= - \langle \al_i,x\rangle f_i, \\
\label{KM rel b} [x, x'] &= 0, &\qquad
[e_i, f_j] &= \chal_i \delta_{ij},
\end{alignat}
for any $x, x' \in \lh$, together with the Serre relations
\be \label{KM rel c}
(\text{ad}\, e_i)^{1- \cart_{ij}} e_j = 0, \qquad (\text{ad}\,f_i)^{1- \cart_{ij}} f_j = 0.
\ee
\end{subequations}
They generate the derived subalgebra $\g'\subset \g$.

We have the principal gradation of $\g$, namely the $\ZZ$-gradation defined by the adjoint action of $\chweyl$, or equivalently the unique $\ZZ$-gradation in which $e_i$ has grade $+1$ and $f_i$ has grade $-1$ for each $i\in I$. Let $\g_j$ denote the subspace of grade $j\in \ZZ$, so we have
\be \g = \bigoplus_{j\in \ZZ} \g_j \nn\ee
with
\be X\in \g_j \qquad \Longleftrightarrow \qquad [\chweyl, X] = j X. \nn\ee

\subsection{Exponents}\label{sec: adef} 
Define 
\be p_{-1} := \sum_{i\in I} f_i \in \n_-\label{pmdef}.\ee
The linear map
\be 
\g_{j+1} \to \g_{j}; \quad X \mapsto [p_{-1},X] \nn\ee
is injective for all $j\geq 1$. If it fails to be surjective for a given $j\geq 1$ then $j$ is called an \emph{exponent} of the Lie algebra $\g$. Let $E\subset \ZZ_{\geq 1}$ be the set of exponents of $\g$. 
In almost all cases, the codimension of $[p_{-1},\g_{j+1}]$ in $\g_j$ is one for every exponent $j$,\footnote{The exceptions are in types ${}^1\!D_{2k}$. In type ${}^1\!D_{2k}$ the codimension of $[p_{-1},\g_{2k+ (4k-2)n}]$ is two for every $n\geq 0$. In such cases one must pick two vectors to span a complementary subspace (and we say the exponent has multiplicity two).} and we may pick a vector $p_j\in \g_j$ such that
\be \g_j = [p_{-1},\g_{j+1}] \oplus \CC p_j,\qquad j\in E.\nn\ee
We make the following choice of these complementary subspaces $\CC p_j$.
First, there is a unique element $p_1\in \g_1$ such that
\be [p_1, p_{-1} ] = \chcent.\nn\ee
Then we may choose nonzero $p_j\in \g_j$, $j\in E$, such that
\be [p_1, p_j] = 0.\label{p1pj}\ee

Let $\a_+$ denote the span of the generators $p_j\in \n_+$, $j\in E$:
\be \a_+ := \bigoplus_{j\in E} \CC p_j \subset \n_+ .\nn\ee 
It is an abelian Lie subalgebra of $\n_+$.

\subsection{Identification of $\h$ and $\h^*$}\label{sec: identification}
From now on we shall simply identify $\h$ with its dual $\h^*$ by means of the linear isomorphism induced by the standard nondegenerate bilinear form $\bilin\cdot\cdot$:
\be \h\xrightarrow\sim \h^*;\qquad x \mapsto \bilin x \cdot.\nn\ee
For each $i\in I$ one then has, from \cref{kdef},
\be \chal_i = \vareps_i \al_i, \quad\text{and}\quad \vareps_i = \frac{\bilin{\chal_i}{\chal_i}} 2 = \frac{2}{\bilin{\al_i}{\al_i}}. \label{nudef}\ee
Note also that then
\be \cent = \sum_{i\in I} \check a_i \chal_i = \sum_{i\in I} \check a_i \eps_i \al_i = \sum_{i\in I} a_i \al_i = \chcent .\nn\ee

\section{Conformal affine Toda}\label{sec: caf}
Having got these conventions about Cartan matrices in place, in this section we describe the intersection of the kernels of the screening operators. We shall find, as we claimed in the introduction, that in affine types it is generated by the conformal vector and the state $\chcent_{-1}\vac$. 

We will then be in a position to state and prove the main local results of the paper, i.e. the main results associated to the disc. They are \cref{Tkdiag} and  \cref{thm: Iaff} (the latter by way of \cref{afthm}).

\subsection{The algebras $\Wg^\eps$ and $\Wgc$}
Define $\Wg^\eps\subset\pi_0^\eps$ to be the intersection of the kernels of the screening operators $S_{\al_i}: \pi_0^\eps \to \pi_{\al_i}^\eps$ corresponding to the simple roots of $\g$:
\be \Wg^\eps  := \bigcap_{i\in I} \ker S_{\al_i} \subset \pi_0^\eps.\nn\ee
Define $\Wgc\subset\picc_0$ to be the intersection of the kernels of the classical screening operators $\ol S_{\al_i}: \picc_0 \to \picc_{\al_i}$ corresponding to the simple roots of $\g$:
\be \Wgc  := \bigcap_{i\in I} \ker \ol S_{\al_i} \subset \picc_0.\nn\ee
By \cref{kercl} and \cref{clkercl} we have the following.
\begin{prop}$ $
\begin{enumerate}[(i)]
\item $\Wg^\eps$ is a vertex subalgebra of $\pi_0^\eps$. 
\item $\Wgc$ is a vertex Poisson subalgebra of $\picc_0$. 
\end{enumerate}\qed
\end{prop}

\subsection{Conformal algebra structure on $\picc_0$ and $\Wgc$}\label{sec: cvg}
As in \S\ref{sec: cv} we have a conformal vector $\confvec_\sv\in \pi_0^\eps$ associated to each element $\sv$ of $\h[\eps]$.
We can consider in particular setting $\sv = - \chweyl + \eps\weyl$. Let us write $\confvec$ for the resulting conformal vector:
\be \confvec := \frac 1 \eps \left( \half b^j_{-1} b_{j,-1} \vac - \chweyl_{-2} + \eps\weyl_{-2} \right) \vac \nn\ee
\begin{lem}\label{omegagker} $\confvec\in \Wg^\eps$. 
\end{lem}
\begin{proof} We must show that $\confvec\in\ker S_{\al_i}$ for each simple root $\al_i$.  
According to \cref{Somlem} we are to evaluate 
\be -1 + \bilin{\al_i}{\chweyl - \eps\weyl} + \frac \eps 2 \bilin{\al_i}{\al_i}.\nn\ee
This indeed vanishes, because $\bilin{\al_i}{\chweyl} = \la \chweyl, \al_i \ra = 1$ and, in view of \cref{nudef}
\be 1 = \la \chal_i, \weyl \ra = \bilin{\chal_i}{\weyl} = \frac{2}{\bilin{\al_i}{\al_i}} \bilin{\weyl}{\al_i} \implies  \bilin{\al_i}{\weyl} = \frac{\bilin{\al_i}{\al_i}} 2 \nn\ee
\end{proof}
Define the classical conformal vector $\ol\confvec\in \picc_0$ as in \S\ref{sec: qc},
\be \ol\confvec = \lim_{\eps\to 0} \eps \confvec = \left(\half b^i_{-1} b_{i,-1} - \chweyl_{-2}\right) \vac \in \picc_0.
\nn\ee

\begin{cor}\label{omin} $\ol\confvec \in \Wgc$.
\end{cor}
\begin{proof}
We must show that $\ol\confvec \in \ker \ol S_{\al_i}$ for each simple root $\al_i$. But in view of  \cref{clcor}, it is enough to check that $\eps\confvec\in \ker S_{\al_i}$, as we just did. 
\end{proof}

Recall \S\ref{sec: qc}. We have established the following.
\begin{prop} The vertex Poisson algebra $\Wgc\subset \picc_0$ is a conformal algebra with central charge 
\be \ol c = \bilin\chweyl\chweyl, \nn\ee 
and is therefore quasi-conformal as a commutative vertex algebra, with the action of $\Der\O$ given by the modes $\ol L_n$, $n\geq -1$, of the vector $\ol\confvec$. \qed
\end{prop}
 
This proposition actually holds as stated regardless of whether $\g$ is of finite or affine type. However, in our setting with $\g$ of affine type we have 
\be \ol c = \bilin\chweyl\chweyl = 0,\label{czero}\ee 
as in \cref{rhonull}. (In finite types $\ol c = \bilin\chweyl\chweyl \neq 0$.)
This is one important distinction between finite and affine types. There is another, to which we turn now. 

\subsection{Affine connection component $\chcent_{-1}\vac$}\label{sec: afc}
Consider the subspace of $\pi_0^\eps$, or of its classical limit $\picc_0$, of grade $+1$. It is isomorphic to $\h$ as a vector space, and is spanned by the states $a_{-1} \vac$ with $a\in \h$.  

\begin{lem} The subspace of $\Wg^\eps$ of grade $+1$ has dimension one, and is spanned by the state $\chcent_{-1} \vac$.
\end{lem}
\begin{proof} In $\pi_0^\eps$ we have $S_\lambda a_{-1} \vac = - \eps \bilin\lambda a \vac$ for all $a,\lambda\in \h$. (Compare \S\ref{Sonom}.) Thus $a_{-1}\vac \in \Wg^\eps$ if and only if $\bilin{\al_i}{a} = 0$ for every simple root of $\g$. 
Hence $a$ is proportional to $\chcent$.
\end{proof}

By \cref{clcor}, we have the classical corollary.
\begin{cor}$ $\label{kin} 
The subspace of $\Wgc$ of grade $+1$ has dimension one, and is spanned by the state $\chcent_{-1} \vac$.
\qed\end{cor}

\begin{rem} $ $
\begin{enumerate}[(i)]
\item In finite types, the subspace of $\Wgc$ of grade $+1$ has dimension zero. 
\item
In finite types, $\Wgc$ is the usual classical $W$-algebra associated to $\g$. It is generated, as a differential algebra, by $\rank(\g)$ states $S_1,S_2,\dots,S_{\rank\g}$, the first of which is $S_1=\ol\confvec$ \cite{FFIoM}.
\item
In affine types, by contrast, $\Wgc$ is generated as a differential algebra by the states $\chcent_{-1}\vac$ and $\ol\confvec$ \emph{only}, irrespective of the rank of $\g$. This will follow from an identification of the differential algebra $\picc_0$ with the space of \emph{$\g$-Miura opers} on the disc. (See \cref{opdisc} and \cref{Wpi} below.)
\end{enumerate}
\end{rem}

\subsection{Canonical translation operator $\Tk$}\label{sec: tkdef}
Let $\Tk$ be the endomorphism of $\picc_\lambda$ (for any $\lambda\in \h$) given by
\begin{align} 
\Tk &:= \sum_{m=0}^\8 \frac {(-1)^m}{(\coxeter)^m m!} \sum_{n_1,\dots,n_m \geq 1} \frac{1}{n_1\dots n_m} \chcent_{-n_1} \dots \chcent_{-n_m} L_{n_1+\dots+ n_m-1} 
\label{Tk}
\\&= L_{-1} - \frac {\chcent_{-1}}\coxeter L_0 + \left( \frac 12\frac{\chcent_{-1}}\coxeter\frac{\chcent_{-1}}\coxeter - \frac{\chcent_{-2}}{2\coxeter}\right) L_1 + \dots   .\nn
\end{align}
\begin{rem}\label{Vkrem}
Recall the formula \cref{Vf} for the intertwining operator $V_\lambda(x)$, $\lambda\in\h$ and the shift operator $\shift_\lambda$. Define $L^{(\geq -1)}(x) = \sum_{n\geq -1} L_n x^{-n-2}$.  On inspecting the expression above for $\Tk$ one sees that 
\be  \shift_{-\chcent/\coxeter} \circ \Tk =  \res_x V_{-\chcent/\coxeter}(x) L^{(\geq -1)}(x) \nn\ee
as an equality of linear maps $\picc_0 \to \picc_{-\chcent/\coxeter}$.
\end{rem} 

\begin{lem}\label{lem: Lj}
We have $v^{L_0} \Tk v^{-L_0} = v \Tk$ for any $v\in \CCx$, while
for all $j\geq 1$, 
\be \left[ L_j, \Tk \right] = 0 \nn\ee
as an equality of endomorphisms of $\picc_\lambda$, for any $\lambda$ such that $\bilin\chcent\lambda =0$. 
\end{lem}
\begin{proof} That $v^{L_0} \Tk v^{-L_0} = v^{-1} \Tk$ is clear. Consider $\left[ L_j, \Tk \right]$. We have the commutation relations (cf. \cref{cld}) 
\be [L_j, L_n] = (j-n) L_{j+n}, \qquad [L_j,\chcent_p] = -p \chcent_{j+p} + j(j+1) \delta_{j+p,0} \coxeter  \nn\ee
for all $j\geq 1$, $n\geq -1$, $p\in \ZZ$. (Recall that $\bilin{-\chweyl} \chcent = - \chweyl(\chcent) = -\coxeter$). 
Thus for all $m\geq 1$,
\begin{align} \label{lsum}
&\left[L_j,  \sum_{n_1,\dots,n_m \geq 1} \frac{1}{n_1\dots n_m} \chcent_{-n_1} \dots \chcent_{-n_m} L_{n_1+\dots+ n_m-1} \right] \\
&=
m \coxeter (j+1) \sum_{n_1,\dots,n_{m-1} \geq 1} \frac 1 {n_1\dots n_{m-1}}  \chcent_{-n_1} \dots \chcent_{-n_{m-1}} L_{n_1+\dots+ n_{m-1} + j -1}  \nn\\
&\quad+
\sum_{p=1}^m  \sum_{n_1,\dots,n_{m} \geq 1} \frac {n_p} {n_1\dots n_{m}}  \chcent_{-n_1} \dots \chcent_{-n_p+j} \dots \chcent_{-n_m} L_{n_1+\dots + n_m -1}  \nn\\
&\quad-
 \!\!\sum_{n_1,\dots,n_{m} \geq 1} 
\frac {(n_1+\dots+n_m-1-j)} {n_1\dots  n_m}  \chcent_{-n_1} \dots \chcent_{-n_m}  L_{n_1+\dots+ n_{m}+j-1}  .\nn
\end{align}
In the penultimate line here, $\chcent_{-n_p+j}$ acts as zero on $\picc_\lambda$ unless $-n_p+j<0$ (as does $a_{-n_p+j}$ for any $a\in \h$ such that $\bilin a\lambda=0$; here we use the assumption that $\bilin\chcent\lambda=0$). So we can drop the first several terms in the sum on $n_p$ and relabel:
\begin{align}
& \sum_{n_1,\dots,n_{m} \geq 1} \frac {n_p} {n_1\dots n_{m}}  \chcent_{-n_1} \dots \chcent_{-n_p+j} \dots \chcent_{-n_m} L_{n_1+\dots + n_m -1}\nn\\
& =\sum_{n_1,\dots,n_{m} \geq 1} \frac {n_p} {n_1\dots n_{m}}   \chcent_{-n_1} \dots \chcent_{-n_m} L_{n_1+\dots + n_m + j -1}.\nn
\end{align}
(The coefficient $\sim n_p/n_p$ is independent of $n_p$). Therefore the final two lines in \cref{lsum} above actually give
\begin{align}
&
\sum_{n_1,\dots,n_{m} \geq 1} \frac {n_1+\dots+n_m} {n_1\dots n_{m}}   \chcent_{-n_1} \dots \chcent_{-n_m} L_{n_1+\dots + n_m + j -1} \nn\\ 
&- \sum_{n_1,\dots,n_{m} \geq 1} 
\frac {(n_1+\dots+n_m-1-j)} {n_1\dots  n_m}  \chcent_{-n_1} \dots \chcent_{-n_m}  L_{n_1+\dots+ n_{m}+j-1}  \nn\\
& = (j+1) \sum_{n_1,\dots,n_{m} \geq 1} 
\frac {1 } {n_1\dots  n_m}  \chcent_{-n_1} \dots \chcent_{-n_m}  L_{n_1+\dots+ n_{m}+j-1}  \nn
\end{align}
Comparing this with the first line in \cref{lsum}, one sees that the sum on $m$ in $[L_j,\Tk]$ telescopes, and we have $[L_j,\Tk]=0$ as required.
\end{proof}

$\Aut\O$ has a subgroup consisting of the rescalings, generated by $L_0$.
Let $\CC\vol$ be the one-dimensional representation of $\Aut\O$ spanned by a vector $\vol$ obeying
\be v^{L_0} \vol = v^{-1} \vol, \quad v\in \CCx \qquad\text{and}\quad L_{j} \vol =0 \quad\text{for all $j\geq 1$}.\nn\ee 

Let $\bra 0\in \picc_0^*$ be the linear map $\picc_0 \to \CC$ which yields $1$ on $\vac$ and pairs as zero with all states of nonzero depth.

Define 
\be H := \bigoplus_{i\in I} \ol S_{\al_i}. \nn\ee
\begin{thm}\label{Tkdiag} We have the following $\Aut\O$-equivariant double complex:
\be \begin{tikzcd}
\CC\vol\\
\picc_0 \ox \CC\vol \rar{H\ox \id} \uar{\bra 0\ox \id}
& \bigoplus_{i\in I} \picc_{\al_i} \ox \CC\vol   \\
\picc_0 \uar{\Tk\vol} \rar{H}
& \bigoplus_{i\in I} \picc_{\al_i} \uar{-\Tk\vol}\\
\CC\vac \uar\\
\end{tikzcd}
\nn\ee
\end{thm}
\begin{proof}
First, let us establish that the horizontal maps are $\Aut\O$-equivariant. By \cref{clintcom}, the map  $\ol S_{\al_i}: \picc_0\to \picc_{\al_i}$ commutes with the Lie algebra action of the non-negative modes of $\ol\confvec\in\picc_0$. That is, it commutes with the action of $\Der\O$.
Recall from \S\ref{sec: qcs} the definition of this action of $\Der\O$ on the modules $\picc_{\lambda}$.
In the present case, in which $\sv = - \chweyl$, we have,
for any $\al$ in the root lattice, 
\be L_0 \on \ket{\al} = \la \chweyl,\al\ra \ket{\al} \in \ZZ \ket{\al}\nn\ee
and thus $L_0$ acts semisimply on $\picc_{\al}$ with integer eigenvalues. That means the action of $\Der_0\O$ on $\picc_{\al}$ can be exponentiated up to give an action of $\Aut\O$, just as can the action on $\picc_0$. So the  horizontal maps are indeed $\Aut\O$-equivariant.

By inspection, the image of each vertical map is contained in the kernel of the next, so we have a complex in the vertical direction. 

$H$ commutes with $\Tk$ by \cref{clintcom} (in view of \cref{omin}, \cref{kin} and the definition of $\Tk$). So $H(-\Tk) + \Tk H = 0$, making the diagram into a double complex.  \Cref{lem: Lj} together with the definition of $\vol$ ensure that the vertical maps are  $\Aut\O$-equivariant.
\end{proof}

\subsection{Canonical modes}\label{canmodes}
It is a standard observation that the negative modes $a_{-n}$, $n>1$, of any $a\in \h$ can be obtained by repeated application of the translation operator $T= L_{-1}$ on $a_{-1}$: 
\be a_{-2} = [T,a_{-1}],\qquad a_{-3} = \frac 1 2[T,[T,a_{-1}]],\dots,\nn\ee 
and in general
\be \sum_{n<0} a_n x^{-n-1} 
= \exp(x\ad_T) a_{-1} .\nn\ee
We now have a modified translation operator, $\Tk$, which commutes with the action of $\Der_+\O$, i.e. with the generators $L_j$, $j\geq 1$, as in \cref{lem: Lj}. We can use it to define a notion of negative modes of states which shares this property.
Define the formal power series
\be \sum_{n<0} a_{[n]} x^{-n-1} = \exp(x \ad \Tk) a_{-1} .\label{canmodedef}\ee
\begin{lem}\label{affdiff}
For all $j\geq 1$ and all $n>0$, \be [L_j,a_{[-n]}] = 0\nn\ee as endomorphisms of $\picc_\lambda$, for any $a\in \h$ and $\lambda\in \h$ such that $\bilin\lambda a =0$.
\end{lem}
\begin{proof} For all $j\geq 1$, $[L_j ,a_{-1}] = a_{-1+j}$ acts as zero on $\picc_\lambda$ (using the assumption that $\bilin\lambda a = 0$, when $j=1$). The result follows using \cref{lem: Lj}. 
\end{proof}
\begin{rem}
It follows that for all $i_1,\dots,i_m\in I$ and $n_1,\dots,n_m \in \ZZ_{>0}$, the state
\be b_{i_1,[-n_1]}\dots  b_{i_m,[-n_m]} \vac \in \picc_0 \nn\ee
is a conformal primary of conformal weight $n_1+\dots+n_m$, i.e. it is in the kernel of $L_j$ for all $j\geq 1$ and an eigenstate of $L_0$ with eigenvalue $n_1+\dots+n_m$.
\end{rem}

\begin{lem}\label{kaff} $\Tk \chcent_{-1} \vac = 0$, and thus  $\chcent_{[-n]} = 0$ for all $n\geq 2$. 
\end{lem}
\begin{proof}
We have 
\begin{align} \Tk \chcent_{-1} \vac &= 
\left(L_{-1} - \frac 1\coxeter \frac{\chcent_{-1}}{1} L_0 - \frac 1 \coxeter \frac{\chcent_{-2}}2 L_1 
+ \frac 1 {2{\coxeter}^2} \frac{\chcent_{-1}}{1}\frac{\chcent_{-1}}{1} L_1 
\right) \chcent_{-1} \vac\nn\\
& = 
\left(\chcent_{-2} - \frac 1 \coxeter \frac{\chcent_{-1}}{1} \chcent_{-1}  - \frac 1 \coxeter \frac{\chcent_{-2}}{2} 2  \coxeter  
+ \frac 1{2\coxeter^2} \frac{\chcent_{-1}}{1} \frac{\chcent_{-1}}{1} 2 \coxeter \right)\vac 
= 0.\nn
\end{align}
\end{proof}

\begin{lem}\label{lem: tkprim} 
If $v\in \picc_0$ is a conformal primary of conformal weight $\Delta\in \CC$ then 
\be \Tk v = L_{-1}v  - \frac\Delta\coxeter \chcent_{-1} v .\nn\ee
\qed
\end{lem}

\subsection{Subquotients of $\h$ and $\picc_0$}
Recall that, for us, after identifying $\h\cong \h^*$ as in \cref{sec: identification},  
\begin{align} \h &= \CC\weyl \oplus \bigoplus_{i\in I} \CC \al_i 
\nn\end{align}
Let $\hr\subset\h$ denote the subspace spanned by the simple roots, and $\hfin$ the quotient of $\hr$ by the span of $\chcent$:
\begin{align} \hr &:= \bigoplus_{i\in I} \CC \al_i 
\nn
\qquad \hfin := \hr/ \CC\chcent.\end{align}
Recall that $\picc_0$ is the vacuum Verma module over the loop algebra $\h((t))$ of $\h$ (and it arose as the $\eps\to 0$ limit of the vacuum Verma modules $\pi^\eps_0$ over the centrally extended loop algebras $\hh^\eps$). 
As a vector space
\be \picc_0 = \CC\left[\weyl_{n}\right]_{n<0} \ox
   \CC\left[\al_{i,n}\right]_{i\in I,n<0} 
\vac
.\nn\ee
Define the subspace $\pir_0$ and its quotient $\pifin_0$ as follows:
\be \pir_0 :=
   \CC\left[\al_{i,n}\right]_{i\in I,n<0}\vac\label{pirdef} \qquad \pifin_0 := \pir_0 \big/ \bigoplus_{n<0} \chcent_n\pir_0.\ee
These are, respectively, the vacuum Verma modules over the loop algebras $\hr((t))$ and $\hfin((t))$. 

\subsection{The action of $\n_+$ on $\picc_0$, $\pir_0$ and $\pifin_0$}
The classical screenings obey the Serre relations of $\g$, in the following sense. (Recall the shift operator $\shift_\lambda$ from \cref{Vf}.)
\begin{lem}\label{naction} The map 
\be \text{ }\qquad  e_i \mapsto Q_i := -\vareps_i^{-1} \shift_{-\al_i}\ol S_{\al_i},\qquad i\in I, \nn\ee
defines an action of the Lie algebra $\n_+$ on $\picc_0$.
\end{lem}
\begin{proof} 
These differential operators $Q_i$ are given by
\begin{align} Q_i   
&= - \sum_{m\leq 0}  \vareps_i^{-1} V_{\al_i}[m] \bilin {\al_i}{b_k} \frac{\del}{\del b_{k,-1+m}} \nn\\&
= -  \sum_{m\leq 0}  V_{\al_i}[m] \la\chal_i, b_k \ra \frac{\del}{\del b_{k,-1+m}}\nn\\
&= - \sum_{m\leq 0}  V_{\al_i}[m] \frac{\del}{\del \weyl_{-1+m}}
 -\sum_{\substack{j\in I\\j\neq 0}}  \sum_{m\leq 0} V_{\al_i}[m] \la\chal_i, \al_j \ra \frac{\del}{\del\al_{j,-1+m}},\label{eact}
\end{align}
where, for convenience, we go to the basis $\{b_k\}_{k=1,\dots,\dim\h} = \{\weyl,\chcent\} \cup \{\al_i\}_{i\in I}$ of $\h$. (The $\del/\del \chcent_{-1+m}$, $m\leq 0$, do not appear in this differential operator, since $\la\chal_i,\chcent\ra = 0$; we also used the fact that  $\la\chal_i, \weyl \ra=1$.) Notice that the coefficients, $V_{\al_i}[m]$, are polynomials in the modes $\al_{i,-j}$, $j\geq 1$, only, and do not contain factors $\weyl_{-j}$. Therefore the derivatives  $\del/\del\weyl_{-1+m}$, $m\leq 0$ in \cref{eact} never contribute in commutators $[Q_i,Q_j]$. It is enough to consider the differential operators 
\begin{align}
 -\sum_{\substack{j\in I\\j\neq 0}}  \sum_{m\leq 0} V_{\al_i}[m] \la\chal_i, \al_j \ra \frac{\del}{\del\al_{j,-1+m}}.\label{eactr}
\end{align}
The proof is then as in Proposition (2.2.8) of \cite{FFIoM} which -- cf. Remark (2.2.9) therein -- applies for any symmetrizable Kac-Moody algebra.
\end{proof}
Thus, $\picc_0$ is a module over $\n_+$. From the argument above we see also that the subspace $\pir_0$ of \cref{pirdef} is a submodule for the action of $\n_+$, and the generator $ e_i$ acts on $\pir_0$ by the differential operator in \cref{eactr}. 
Since  $\la\chal_i,\chcent\ra = 0$ for all $i\in I$, this action of $\n_+$ on $\pir_0$ in turn descends to a well-defined action on the quotient $\pifin_0$. Let us write $\Qfin_i$
for the differential operator realizing $ e_i$ on $\pifin_0$. 
It is given by the same formula, \cref{eactr}, but with $\al_{i}\in \hr$ replaced by the equivalence class 
\be [\al_i] := \al_{i} + \CC \chcent \in \hfin.\nn\ee 
Now, these screening operators $\Qfin_i$ were studied in \cite{FFIoM}. We can ``lift'' results from that setting to ours in the following fashion.

\subsection{The subspace $\piaff_0$}\label{sec: piaff}
Let us set
\be \tildeal_i := \al_i - \frac \chcent \coxeter, \qquad i\in I. \label{atdef}\ee
Let $\haff\subset \hr$ denote their span inside $\hr$. 
Define a vector subspace $\piaff_0\subset \pir_0\subset \picc_0$ by
\be \piaff_0 := 
 \CC[\tildeal_{i,[n]}]_{i\in I\setminus\{0\}, n<0}\vac.
 \nn\ee
Let us stress that the modes ${}_{[n]}$ here are those defined in \cref{canmodes}. 

In the remainder of this subsection we shall prove the following key result. 
\begin{thm}\label{afthm} The action of $\n_+$ stabilizes the subspace $\piaff_0\subset \picc_0$. Furthermore, as $\n_+$-modules, $\piaff_0$ and $\pifin_0$ are isomorphic:
\be \piaff_0 \,\,\,\cong_{\n_+} \pifin_0 .\nn\ee 
\end{thm}
\begin{proof}
First, observe that the canonical map $\piaff_0\to \pifin_0$ is actually a vector-space isomorphism, because the $\tildeal_i$ are chosen to obey the linear relation
\be \sum_{i\in I} \chaaa_i \tildeal_i = \sum_{i\in I} \chaaa_i \al_i - \frac 1 \coxeter \sum_{i\in I} \chaaa_i \chcent = \chcent -\chcent = 0.\label{linrel}\ee
So we have a vector-space isomorphism
\be \pifin_0 \isom_\CC \piaff_0.\label{finaff}\ee
One can think that this isomorphism takes monomials in $\pifin_0$ and ``decorates'' them with appropriate terms involving negative modes of $\chcent$. For example 
\begin{align} [\al_i]_{-2}[\al_{j}]_{-1}\vac &\mapsto
\tildeal_{i,[-2]} \tildeal_{j,[-1]} \vac
= \left(\al_{i,[-2]} - \frac 1 \coxeter \chcent_{[-2]} \right)\left(\al_{j,[-1]} - \frac 1 \coxeter \chcent_{[-1]} \right)\vac\nn\\
&= \al_{i,[-2]} \left(\al_{j,-1} - \frac 1 \coxeter \chcent_{-1} \right)\vac
= \left( \al_{i,-2} - \frac 1 \coxeter \chcent_{-1} \al_{i,-1}\right)\left(\al_{j,-1} - \frac 1 \coxeter \chcent_{-1} \right)\vac.\nn\end{align}

\begin{lem}\label{TkQ}
For each $i\in I$, 
\be [\Tk,Q_i] = -\tildeal_{i,-1} Q_i \nn\ee
and (hence)
\be [L_{-1},\Qfin_i] = -[\al_{i}]_{-1} \Qfin_i\nn,\ee
as endomorphisms of $\piaff_0$ and $\pifin_0$ respectively. 
\end{lem}
\begin{proof}
From the definition, \cref{Tk}, of $\Tk$, together with \cref{omin}, \cref{kin} and \cref{clintcom}, we have the commutativity of the diagram
\be \begin{tikzcd} 
\picc_0 \rar{\ol S_{\al_i}}& \picc_{\al_i} \nn\\
\picc_0 \rar{\ol S_{\al_i}}\uar{\Tk}& \picc_{\al_i} \arrow{u}[right]{\Tk}\nn
\end{tikzcd}
\nn\ee
(We used this already in \cref{Tkdiag}.) Now by definition, $Q_i =  -\vareps_i^{-1} \shift_{-\al_i}\ol S_{\al_i}$, so we need also to consider the commutator of $\Tk$ with the shift operator $\shift_{\al_i}$.  
On taking the classical limit of \cref{Deltadef}, one has
\be L_0 \ket{\al_i} = \bilin{\chweyl}{\al_i} \ket{\al_i} = \la\chweyl,\al_i\ra \ket{\al_i} = 1\ket{\al_i}.\nn\ee
In this way we see that $[L_0,\shift_{\al_i}] = \shift_{\al_i}$. It is clear that $[L_j,\shift_{\al_i}] = 0$ for all $j>0$. By definition $[\chcent_{-m},\shift_{\al_i}] = 0$ for all $m>0$. Hence, using \cref{Lm1vacl}, we have that
\be [\Tk,\shift_{\al_i}] = [L_{-1} - \frac 1 \coxeter \chcent_{-1} L_0, \shift_{\al_i}] 
= \left(\al_i - \frac 1 \coxeter \chcent_{-1} \right) \shift_{\al_i} = \tildeal_i \shift_{\al_i}.\nn\ee
Thus $[\Tk,Q_i] = - \tildeal_i Q_i$ as claimed. The statement that $[L_{-1},\Qfin_i] = -[\al_{i}]_{-1} \Qfin_i$ follows (or can be checked directly in a similar fashion).
\end{proof}

\begin{prop}\label{afffin}
The screening operators $Q_i$, $i\in I$, stabilize the subspace $\piaff_0$. Moreover the following diagram commutes for each $i\in I$:
\be \begin{tikzcd}
\piaff_0 \rar{\sim}& \pifin_0 \nn\\
\piaff_0 \rar{\sim}\uar{Q_i}& \pifin_0 \uar{\Qfin_i}\nn
\end{tikzcd} \nn\ee
\end{prop}
\begin{proof}
$Q_i$ acts a derivation, so to show that it stabilizes $\piaff_0$ it is enough to show that it sends any mode $\tildeal_{j,[-1-n]}$, $j\in I\setminus\{0\}$, $n\geq 0$, to some element of $\CC[\tildeal_{k,[m]}]_{k\in I\setminus\{0\},m<0}$. That follows by an induction on $n$, making use of \cref{TkQ}. Explicitly, we have 
\begin{align} 
\left[Q_i,  \tildeal_{j,[-1]}\right] &=  \left[Q_i,  \tildeal_{j,-1}\right] =  V_{\al_i}[0] \la \chal_i, \tildeal_{j,-1} \ra 1
=\la \chal_i, \tildeal_{j,-1} \ra 1 . \label{Qactfact}
\end{align}
and then for the inductive step\begin{align} \left[Q_i , \tildeal_{j,[-1-n]}\right] 
= \frac{1}{n}\left[ Q_i,  \left[\Tk, \tildeal_{j,[-n]} \right] \right]
&=  \frac{1}{n} \left[\left[ Q_i,\Tk\right], \tildeal_{j,[-n]} \right] +  \frac{1}{n} \left[\Tk, \left[ Q_i,\tildeal_{j,[-n]}\right] \right]\nn\\
&=  -\frac{1}{n} \tildeal_{i,-1} \left[ Q_i, \tildeal_{j,[-n]} \right] +  \frac{1}{n} \left[\Tk, \left[ Q_i,\tildeal_{j,[-n]}\right] \right]\nn\\
&=  -\frac{1}{n} \tildeal_{i,[-1]} \left[ Q_i, \tildeal_{j,[-n]} \right] +  \frac{1}{n} \left[\Tk, \left[ Q_i,\tildeal_{j,[-n]}\right] \right]\nn
\nn\end{align}
This induction amounts to giving a recursive definition of $Q_i$ on $\tildeal_{j,[-1-n]}$. One sees that there a recursive definition of $\Qfin_i$ on $[\al_{j}]_{-1-n}$ with exactly the same structure, with $\Tk$ replaced by $L_{-1}$ and $\tildeal_i$ by $[\al_i]$. It follows that the diagram given in the statement of the proposition indeed commutes.
\end{proof}
This completes the proof of \cref{afthm}. 
\end{proof}

\subsection{Structure of $\piaff_0$ and integrals of motion}
The following result about the structure of $\pifin_0$ as an $\n_+$ module was established in \cite{FFIoM}. By \cref{afthm} it applies also to $\piaff_0$.

\begin{thm}[\cite{FFIoM}]$ $\label{ffprop}
\begin{enumerate}[(i)]
\item
As an $\n_+$-module, $\pifin_0$ is isomorphic to the module coinduced from the trivial one-dimensional module over $U(\a_+)$:
\be \pifin_0 \cong \Homres_{\a_+}(U(\n_+), \CC) .\nn\ee
\item Consequently (by Shapiro's lemma) the cohomologies of the Lie algebra $\n_+$ with coefficients in $\pifin_0$ are given by exterior powers of $\a_+^*$:
\be H^\bullet(\n_+,\pifin_0) \cong H^\bullet(\a_+,\CC) \cong \largewedge^\bullet(\a_+^*). \nn\ee
\end{enumerate}
\qed\end{thm}
The first of these cohomologies, $H^1(\n_+,\pifin_0)$ ($\cong H^1(\n_+,\piaff_0)$) plays a vital role in \cref{thm: Iaff} below. To state the theorem, we first note the following analog of \cref{Tkdiag} for the subspace $\piaff_0$. Here we define $\piaff_{\lambda}\subset \picc_\lambda$ for any $\lambda\in \h$: 
\be \piaff_{\lambda} := \CC[\tildeal_{i,[n]}]_{i\in I\setminus\{0\}, n<0}\vacl. \nn\ee
\begin{cor}\label{Tkdiagaff} We have the following $\Aut\O$-equivariant double complex:
\be \begin{tikzcd}
\CC\vol\\
\piaff_0 \ox \CC\vol \rar{H\ox \id} \uar{\bra 0\ox \id}
& \bigoplus_{i\in I} \piaff_{\al_i} \ox \CC\vol   \\
\piaff_0 \uar{\Tk\vol} \rar{H}
& \bigoplus_{i\in I} \piaff_{\al_i}  \uar{-\Tk\vol}\\
\CC\vac \uar
\end{tikzcd}
\nn\ee
\end{cor}
\begin{proof}
Given \cref{Tkdiag} and \cref{afthm}, what has to be checked is that the subspace $\piaff_0$ is stabilized by the actions of $\Aut\O$ and $\Tk$. But this is immediate from, respectively, \cref{affdiff} and the definition, \cref{canmodedef}, of the modes ${}_{[n]}$. 
\end{proof}

Let $\Faff_0$ denote the quotient of the kernel of $\bra 0\ox \id$ in $ \piaff_0\ox \CC\vol$ by the image of $\Tk$.
We get the quotient linear map 
\be \mc H :  \Faff_0 \longrightarrow \bigoplus_{i\in I} \Faff_{\al_i}. \nn\ee
Define $\Iaff$ to be the kernel of this map,
\be \Iaff := \ker \mc H \subset \Faff_0.\nn\ee 
By construction, $\Faff_0$ admits an action of $\Aut\O$, and $\Iaff$ is an $\Aut\O$-submodule.

The diagram in \cref{Tkdiagaff} respects the $\ZZ$-gradations (i.e. the eigenspaces of $L_0$). So we get $\ZZ$-gradations on $\Faff_0$ and $\Iaff$. Note that $\vol$ has grade $-1$ by definition. 

\begin{thm}\label{thm: Iaff} We have the following isomorphisms of $\ZZ$-graded vector spaces:
\be \Iaff \cong_\CC 
 H^1(\n_+,\piaff_0) \cong_\CC \a_+^* \nn\ee
\end{thm}
\begin{proof} The second isomorphism is a special case of \cref{ffprop}(ii). 
Given \cref{afthm} and \cref{afffin}, the proof of the first isomorphism is the same as in the paper \cite{FFIoM}.
\end{proof}

It is helpful to restate the theorem in more concrete language:
\begin{thmbis}{thm: Iaff}$ $\label{vjthm}
For each exponent $j\in E$ there exists a nonzero conformal primary state of conformal weight $j+1$,
\be \hamd_j\in \piaff_0 \subset \picc_0, \nn\ee
such that $\Iaff$ is the span of the classes $[\hamd_j\ox \CC\vol]$:
\be \Iaff = \bigoplus_{j\in E} \CC [\hamd_j\ox \CC\vol] \nn\ee
\qed\end{thmbis}
(Each $\hamd_j$ is, of course, unique only up to the addition of the canonical translate $\Tk f_j$ of any state $f_j\in \piaff_0$ of grade $j$.)

We have the obvious analogs $\F_0$, $\mc H$ and $\mc I$ for $\picc_0$. In view of \cref{Tkdiagaff} we have $\Faff_0 \subset \F_0$ and $\Iaff \subset \mc I$.

\section{Coinvariants on the Riemann sphere}\label{crs}
So far all the objects we considered were associated with the algebra of functions $\O = \CC[[t]]$ on the disc. Now we want to think that they are local objects, attached to a point of the Riemann sphere. Then we shall define global objects on the Riemann sphere, namely certain spaces of coinvariants/conformal blocks. 

In order to separate concerns, in this section we return to the setting of \cref{sec: pieps} in which $\h$ was a just a finite-dimensional vector space equipped with a non-degenerate symmetric bilinear form. Later, in \cref{sec: tto}, we reintroduce the extra structure coming from the Cartan matrix.

\subsection{Germs of functions}
Let $\cp1$ denote a copy of the Riemann sphere.
Given a point $p\in \cp1$, recall that a \emph{germ} of a holomorphic function at $p$ is just a holomorphic function on some (sufficiently small) open neighbourhood of $p$, with two germs considered equal if they are equal on some (sufficiently small) open neighbourhood of $p$.
We shall write:
\begin{enumerate}
\item[$\O_p$ ---] the local ring of germs of holomorphic functions at $p$,
\item[$\m_p$ ---] the maximal ideal in $\O_p$, i.e. germs of holomorphic functions vanishing at $p$,
\item[$\K_p$ ---] the field of germs of meromorphic functions at $p$, i.e. the field of fractions of $\O_p$.
\end{enumerate}
For any open subset $U\subset \cp1$, we shall write $\K(U)$ for the field of meromorphic functions on $U$ and $\O(U)$ for the ring of holomorphic functions on $U$.
Recall that $U\mapsto \K(U)$ is a sheaf, whose stalk at $p$ is $\K_p$, and likewise for $\O$.

\subsection{Local copies $\pi^\eps_{0,p}$}\label{sec: local copies}
Roughly speaking, our goal is now to attach copies of the vacuum Verma module $\pi^\eps_0$ of \cref{sec: fm} to points in $\cp1$. There is one very natural way of doing so. It will turn out to have a limitation for our purposes, but it is nonetheless instructive to consider it first.

Let $\h$ and $\bilin\cdot\cdot$ be as in \S\ref{sec: ff}.
Associated to a point $p\in \cp1$, we have the loop algebra $\h \ox \K_p$. 
The centrally extended loop algebra $\hh^\eps_p$ is by definition the vector space
\be \hh^\eps_p := \h \ox \K_p \,\,\oplus\,\, \CC\bm 1_p,\nn\ee 
where $\bm 1_p$ is central, equipped with the Lie bracket given by
\be [a \ox f,  b\ox g] = [a,b] \ox fg + \eps \bm 1_p \bilin ab  \res_p 
fdg ,\nn\ee
for $a,b\in \h$, $f,g\in \K_p$. 
For any $\lambda\in \h$, we can define the $\hh^\eps_p$-module 
\be \pi^\eps_{\lambda,p} := U(\hh^\eps) \ox_{U(\h[[t]]\oplus \CC \bm 1)} \CC \vacl \nn\ee
induced from the one-dimensional $(\h\ox \O_p\oplus \CC \bm 1)$-module $\CC \vacl$ spanned by a vector $\vacl$ obeying
\be (a\ox 1) \vacl = \eps\bilin \lambda {a} \vacl , \qquad (a\ox \m_p) \vacl = 0,\label{pvldef}\ee
for all $a\in \h$, and $\bm 1 \vacl = \vacl$. As a special case we have the vacuum Verma module, $\pi^\eps_{0,p}$. 
These definitions are intrinsic, i.e. they involve no choice of coordinate or other data on $\cp1$. 
Now suppose we choose a coordinate patch $U\subset \cp1$ and a holomorphic coordinate 
\be t: U \to \CC.\nn\ee 
For each $p\in U$ let $t_p:= t-t(p)$ be the local coordinate at $p$. The choice of coordinate gives rise to an isomorphism 
\be \K_p \cong \CC\laurent{t_p} \cong \CC\laurent t,\label{cois}\ee 
where $\CC\laurent t$ is the field of Laurent series in $t$ (with nonzero radius of convergence). The latter embeds in the field $\CC((t))$ of formal Laurent series.
This gives an embedding of $\hh^\eps_p$ into the abstract algebra $\hh^\eps$ from \S\ref{sec: ff} and hence a vector-space isomorphism
\be \iota_{t_p} : \pi^\eps_{0,p} \isom \pi^\eps_0 \nn\ee 
between $\pi^\eps_{0,p}$ and the abstract vacuum Verma module $\pi^\eps_0$ from \S\ref{sec: fm}.\footnote{Recall that these latter objects have a $\ZZ$-gradation, coming from the $\ZZ$-gradation of $\CC\laurent t$ as $\CC$-algebra by powers of $t$.
So we get a $\ZZ$-gradation of $\hh^\eps_p$ and of $\pi^\eps_{0,\eps}$. It is coordinate dependent, but its underlying filtration is not. 
}

The choice of local coordinate $t_p$ gives in particular an embedding of $\mc O_p$ into $\mc O$ of \cref{sec: ct} (cf. \cref{rem: holomorphic}) and hence $\Aut \O_p \into \Aut \mc O$. By considering the coordinate transformation to a new local holomorphic coordinate $s_p$ at $p$, with $t_p= \mu_p(s_p)$ for some $\mu_p \in \Aut \O_p$, we get a vector-space isomorphism 
\be \Rnought(\mu_p) :=  \iota_{s_p}\circ \iota_{t_p}^{-1} : \pi^\eps_0 \isom \pi^\eps_0.\nn\ee
By construction, the map $\mu \mapsto \Rnought(\mu)$ is a homomorphism  $\Aut\O_p \to \GL(\pi^\eps_0)$, i.e. an action of $\Aut\O_p$ on $\pi^\eps_0$.\footnote{Indeed, suppose $t=\mu_1(s)$ and $s=\mu_2(r)$ so that $t=\mu_1(\mu_2(r)) = (\mu_2\ast\mu_1)(r)$; then $\Rnought(\mu_2)\circ \Rnought(\mu_1) = \left( \iota_{r}\circ\iota_{s}^{-1} \right) \circ \left( \iota_{s} \circ \iota_t^{-1}\right) = \iota_r \circ \iota_t^{-1} = \Rnought(\mu_2\ast\mu_1)$ as required.}

The problem is that it is not the action we want. Indeed, recall that each choice of conformal vector $\confvec_\sv$, $\sv\in \h$, defines an action $\Rsv$ of $\Aut\O$ on $\pi^\eps_0$, \S\ref{sec: cv} and \S\ref{sec: ct}. It is straightforward to check that the action defined above indeed corresponds to the choice $\sv =0$ -- whereas we are interested in the case $\sv = - \chweyl + \eps\weyl$ and its $\eps\to 0$ limit, as in \S\ref{sec: cvg}.

For that reason, we need a somewhat more general construction; we follow \cite[\S6 and especially \S8]{FrenkelBenZvi}. 

\subsection{The bundle $\Pisv$} First, we define a vector bundle $\Pisv$ over $\cp1$ with fibre isomorphic to $\pi^\eps_0$. 
We define it by giving the transition functions between a class of local trivializations. 
Let $U\subset \cp1$ be any coordinate patch. To each choice of holomorphic coordinate $t:U \to \CC$ we associate a local trivialization 
\be \triv_{U,t}: \Pisv_U \isom U \times \pi^\eps_0.\label{trivdef}\ee
Let $s: U \to \CC$ be another choice of holomorphic coordinate on $U$. To define the bundle it is enough to define suitable transition functions 
\be \triv_{U,s} \circ \triv_{U,t}^{-1} : U \times \pi^\eps_0 \to U \times \pi^\eps_0 \nn\ee 
between the corresponding trivializations. (One can think that $U$ is really the overlap of two coordinate patches.)

At each point $p\in U$ we have the local coordinates $t_p= t-t(p)$ and $s_p = s- s(p)$, and they are related by $t_p = \mu_p(s_p)$ for some $\mu_p\in \Aut\O_p$. After embedding $\O_p\into \O$ using our initial local coordinate $t_p$, we get an element $\mu_p \in \Aut\O$ for each $p\in U$. 
Let 
\be \Rsv: \Aut\O \to \GL(\pi^\eps_0) \nn\ee 
(or, in the $\eps\to 0$ limit, $\Aut\O \to \GL(\picc_0)$) be the homomorphism defined as in \cref{Rsv} in terms of the modes $L_n$ of the conformal vector associated to $\sv\in \h$. We define the transition function $\triv_{U,s} \circ \triv_{U,t}^{-1}$
between the trivializations $\triv_{U,t}$ and $\triv_{U,s}$ to be
\be \left(\triv_{U,s} \circ \triv_{U,t}^{-1}\right)(p,v) = (p, \Rsv(\mu_p)\on v) .\label{trfn}\ee 

\subsection{Sub-bundle $\Pisv_{\leq 1}$}\label{sec: pisvleq}
The action $\Rsv$ of $\Aut\O$ respects the depth filtration of $\pi^\eps_0$. Consequently, the bundle $\Pisv$ has a sub-bundle, which we denote by $\Pisv_{\leq 1}$, with fibre isomorphic to 
\be (\pi^\eps_0)_{\leq 1} \cong  \CC\vac \oplus \h. \nn\ee
From \cref{btrans} we read off the transition functions between trivializations corresponding to different holomorphic coordinates $t$ and $s$ with $t = \mu(s)$. By definition
\begin{align} \mu_p(s_p) &= t_p = t- t(p) 
= \mu(s) - \mu(s(p))
   = \mu\big(s_p + s(p)\big) - \mu(s(p))
\label{chc}\end{align}
so $\mu_p'(0) = \mu'(s(p))$ and $\mu_p''(0) = \mu''(s(p))$. So in the fibre at the point $p$,
\begin{align} \Rsv(\mu_p) \on a_{-1}\vac
&=  \left(a_{-1}\vac  + \bilin{\sv}{a} \frac{\mu''(s(p))}{\mu'(s(p))} \vac \right)\frac {1}{\mu'(s(p))}, \qquad a\in \h\nn\\
\Rsv(\mu_p) \on \vac &= \vac .
\label{btrans2}\end{align}

Let $\Omega$ denote the canonical bundle over $\cp1$, i.e. the holomorphic cotangent bundle. Its fibres are copies of $\CC$. In the local trivialization defined by a holomorphic coordinate $t$, sections of $\Omega$ look like $f(t) dt$, and the transition functions are given by $f(t) dt = \tilde f(s) ds$, i.e. $\tilde f(s) = \mu'(s) f(\mu(s))$. 

Consider the tensor product bundle 
\be \Pisv_{\leq 1} \ox \Omega. \nn\ee
Its fibres are copies of $(\pi^\eps_0)_{\leq 1}\ox \CC \cong (\pi^\eps_0)_{\leq 1}$. 
As a matter of notation, let us write $A(t)dt$, for the image of a meromorphic section of $\Pisv_{\leq 1} \ox \Omega$ in the local trivialization defined by a coordinate $t:U\to \CC$. (So the $dt$ will keep track of the choice of trivialization.) 
We see that the transition functions are given by
\begin{align} a_{-1} \vac f(t) dt  &\mapsto  \left(a_{-1}\vac  + \bilin{\sv}{a} \frac{\mu''(s)}{\mu'(s)} \vac \right) f(\mu(s)) ds\nn\\
\vac f(t) dt &\mapsto \vac f(\mu(s)) \mu'(s) ds .  \label{tfn}\end{align}
\begin{rem}\label{dtr}
If $\sv =0$ then the bundle $\Pisv_{\leq 1} \ox \Omega$ is isomorphic to the direct sum of $\Omega$ and the trivial vector bundle with fibre $\h$.
\end{rem}

\subsection{Connections on $\Pisv$}\label{sec: con}
Given any vector bundle $\mc V$ over $\cp1$, let $U \mapsto \Gamma(U,\mc V)$ denote the sheaf of meromorphic sections of $\mc V$. 

\begin{lem}\label{conlem}
Let $A\in \End(\pi^\eps_0)$. There is a well-defined flat holomorphic connection $\nabla^A: \Gamma(\cdot,\Pisv) \to \Gamma(\cdot,\Pisv \ox \Omega)$ on $\Pisv$ given by
\be \nabla^A \sigma(t) = \left(\sigma'(t) + A \sigma(t)\right) dt \nn\ee
if and only if, for all $\mu\in \Aut\O$,
\be R(\mu_s)^{-1} (\del_s R(\mu_s) ) + R(\mu_s)^{-1}A R(\mu_s)
=  \mu'(s) A 
 \nn\ee
as an equality of endomorphisms of $\pi^\eps_0$, where $\mu_s(x) := \mu(x+s) - \mu(s)$. 
\end{lem}
\begin{proof}
Let $s$ and $t = \mu(s)$ be two holomorphic coordinates on a patch $U\subset \cp1$. Let $\sigma: t(U) \to \pi^\eps_0; t\mapsto \sigma(t)$ be the image of a meromorphic section of $\Pisv$ in the trivialization corresponding to the coordinate $t$. By definition of $\Pisv$, the same section is given in the trivialization corresponding to the coordinate $s$ by $s(U) \to \pi^\eps_0; s\mapsto R(\mu_s)\on \sigma(\mu(s))$ where $\mu_s(x) := \mu(x + s) - \mu(s)$ (cf. \cref{chc}).  
In the $t$-trivialization, the derivative of this section is
\be \nabla^A \sigma(t) = \left(\sigma'(t) +A \sigma(t)\right) dt \nn\ee
which is, in the $s$-trivialization,
\be R(\mu_s) \on \nabla \sigma(t) =  \left( R(\mu_s) \on \left(\sigma'(\mu(s)) +A \sigma(\mu(s))\right)\right) \mu'(s) ds\nn\ee 
On the other hand, we can take the derivative in the $s$-trivialization:
\begin{align} 
\nabla^A R(\mu_s)\on \sigma(\mu(s)) 
&= \left(\del_s R(\mu_s)\on  \sigma(\mu(s)) +A R(\mu_s) \on \sigma(\mu(s))\right) ds \nn\\
&= \left(R(\mu_s)\on  \sigma'(\mu(s))\mu'(s) + (\del_s R(\mu_s) )\on \sigma(\mu(s))  +A  R(\mu_s) \on  \sigma(\mu(s))\right) ds \nn\end{align}
These quantities agree if and only if $A$ satisfies the condition given in the statement of the proposition. 
\end{proof}

The following standard result shows that holomorphic connections on $\Pisv$ exist. 
Recall the translation operator $T$ of the vertex algebra $\pi^\eps_0$, and the fact that $T = L_{-1} \equiv (\confvec_\sv)_{(0)}$ irrespective of the choice of $\sv\in \h$, as in \cref{omT}. 
\begin{lem}\label{lem: nab}
There is a well-defined flat holomorphic 
connection $\nabt: \Gamma(\cdot,\Pisv) \to \Gamma(\cdot,\Pisv \ox \Omega)$ on $\Pisv$ given by
\be \nabt \sigma(t) = \left(\sigma'(t) + T\sigma(t)\right) dt. \nn\ee
\end{lem}
\ifdefined\short
\begin{proof}
See \cite[\S6.6]{FrenkelBenZvi}.
\end{proof}
\else
\begin{proof} For completeness, let us include the following proof, taken from \cite[\S6.6]{FrenkelBenZvi}. We need to check that $T=L_{-1}$ is one solution to the condition in \cref{conlem}.

Now, we may let $\Der\O$ act on functions of an indeterminate $x$ via the realization $L_n \mapsto -x^{n+1} \del_x$. By definition of $R(\mu)$, in this realization $R(\mu)\on f(x) = f(\mu(x))$. Therefore $R(\mu_s)^{-1}\del_s R(\mu_s) \on x = R(\mu_s)^{-1} \del_s \mu_s(x) = \left(\del_s \mu_s(y)\right)|_{y=\mu_s^{-1}(x)}=  \left(\del_s \left( \mu(s+y) - \mu(s)\right)\right)|_{y=\mu_s^{-1}(x)} = \mu'(s+ \mu_s^{-1}(x)) - \mu'(s)$. And $R(\mu_s)^{-1} L_{-1} R(\mu_s) \on x = R(\mu_s)^{-1} (-\del_x) R(\mu_s) \on x = - R(\mu_s)^{-1}\del_x \mu_s(x) =  - R(\mu_s)^{-1}\mu_s'(x) = - \mu_s'(y)|_{y=\mu_s^{-1}(x)} = -\mu'(s+ \mu_s^{-1}(x))$. Thus 
\be \left(R(\mu_s)^{-1}\del_s R(\mu_s)+ R(\mu_s)^{-1} L_{-1} R(\mu_s)\right) \on x = - \mu'(s).\nn\ee 
And indeed, $\mu'(s) L_{-1} \on x = \mu'(s) (-\del_x)\on x = -\mu'(s)$ also. This is enough to establish the required equality: both sides are sums of the form $\sum_{n\geq -1} c_n L_n$, and one has $\sum_{n\geq -1} c_n (-x^{n+1} \del_x) \on x = 0$ only if $c_n=0$ for all $n$. 
\end{proof}
\fi
However, in our setting this connection will belong a family of such holomorphic connections, defined using the canonical translation operator $\Tk$ of \cref{sec: tkdef}. 
For that reason we shall need the following corollary of \cref{conlem} which describes the (affine space of) holomorphic connections on $\Pisv$.   

\begin{cor}\label{affcor}
Suppose
\be \nabla^A \sigma(t) := \left(\sigma'(t) + A \sigma(t)\right) dt \nn\ee
defines a flat holomorphic connection $\Gamma^A(\cdot,\Pisv) \to \Gamma(\cdot,\Pisv \ox \Omega)$ on $\Pisv$. Then
\be \nabla^{A+B} \sigma(t) := \nabla^A \sigma(t) + B\sigma(t) dt \nn\ee
is another well-defined flat holomorphic connection if and only if the element $B\in \End\pi^\eps_0$ obeys
\be R(\mu)^{-1}B R(\mu) =  \mu'(0) B \nn\ee
for all $\mu \in \Aut\O$.  
\end{cor}
\begin{proof}
Given the \cref{conlem}, certainly a necessary and sufficient condition is that 
$R(\mu_s)^{-1}B R(\mu_s) =  \mu'(s) B$
for all $\mu\in \Aut\O$. But, in turn, a necessary and sufficient condition for that is the same statement with $s=0$. 
\end{proof}

For the moment, we need only the connection $\nabt$ of \cref{lem: nab}. 
In the local trivialisation of $\Pisv \ox \Omega$ defined by a holomorphic coordinate $t: U \to \CC$, we have, for example,
\begin{align} \nabt \left( a_{-1} \vac f(t) \right) &=  a_{-1} \vac f'(t) dt + a_{-2} \vac f(t) dt, \nn\\
      \nabt \left( \phantom{a_{-1}}\vac f(t)\right) &=  \phantom{a_{-1}} \vac f'(t) dt ,\label{nact} \end{align}
for $a\in \h$ and $f(t)\in \K(U)$.

We then have the sheaf $H^1(\Pisv, \nabt)$ of the first de Rham cohomology of this flat connection, i.e. the sheaf which assigns to each open $U\subset \cp1$ the quotient vector space
\be H^1(U,\Pisv, \nabt) := \Gamma(U,\Pisv \ox \Omega) 
                     \big/\nabt \Gamma(U,\Pisv).\label{ndr}\ee
It may be shown that this is in fact a sheaf of Lie algebras \cite[\S9.2 and \S19.4]{FrenkelBenZvi}. 
For our purposes it is enough to consider a certain sub-sheaf, as follows.

\subsection{The sheaf of Lie algebras $\Hsv$}\label{hsvdef}
For any chart $U$ let 
\be \Hsv(U) \into H^1(U,\Pisv, \nabt) \nn\ee 
denote the image in $H^1(U,\Pisv, \nabt)$ of $\Gamma(U,\Pisv_{\leq 1} \ox \Omega)$.
Explicitly, in the local trivialisation of $\Pisv_{\leq 1} \ox \Omega$ defined by a holomorphic coordinate $t: U \to \CC$, $\Hsv(U)$ is the $\CC$-linear span of elements of the form $a_{-1} \vac f(t) dt$ and $\vac f(t)dt$, with $a\in \h$, $f(t)\in \K(U)$, modulo the $\CC$-linear span of elements of the form $\vac f'(t) dt$, cf. \cref{nact}. 
\begin{lem}\label{laslem}
There is a Lie bracket on $\Hsv(U)$ defined by  $[ \vac f(t) dt,\, \cdot\, ] := 0$ and
\begin{align} \left[ a_{-1} \vac f(t) dt,\, b_{-1} \vac g(t) dt \right] &:= \eps \bilin a b \vac f(t) g'(t) dt
, \nn\end{align}
for $a,b\in \h$, $f(t), g(t)\in \K(U)$.

This makes $U \mapsto \Hsv(U)$ into a sheaf of Lie algebras. There is a short exact sequence of sheaves of Lie algebras
\be 0 \to H^1 \to \Hsv \to \Hsv/H^1 \to 0,\nn\ee
where $H^1:= 
\Gamma(\Omega)/ d\K$ is the sheaf of meromorphic sections of the holomorphic de Rham cohomology of $\cp1$, regarded as a sheaf of commutative Lie algebras. 

\end{lem}
\begin{proof} First, recall that there is a well-defined central extension of the (commutative) Lie algebra $\h \ox \K(U)$ of $\h$-valued meromorphic functions on $U$, by a centre $H^1(U) := \Gamma(U,\Omega)/ d\K(U)$, defined by
\be [a\ox f, b\ox g] = \eps \bilin a b f dg. \label{hcen}\ee
In the trivialization defined by the coordinate $t$, this agrees with our bracket on $\Hsv(U)$, given the obvious identifications, $a_{-1} \vac$ with $a$ and $\vac$ with $1$. 
So we have defined a Lie bracket, in this trivialization. It remains to check that this Lie bracket is intrinsic, i.e. independent of the choice of trivialization. 
But this is clear because the modification to the transition functions is just by central elements with respect to the Lie bracket. 
(In the trivialization defined by $s$ with $t= \mu(s)$, the sections $a_{-1} \vac f(t) dt$ and $b_{-1} \vac g(t) dt$ are written as $\left(a_{-1}\vac  + \bilin{\sv}{a} \frac{\mu''(s)}{\mu'(s)} \vac \right) f(\mu(s)) ds$ and $\left(b_{-1}\vac  + \bilin{\sv}{b} \frac{\mu''(s)}{\mu'(s)} \vac \right) g(\mu(s)) ds$ respectively. The Lie bracket of these is  
$\eps \bilin a b \vac f(\mu(s)) d g(\mu(s))$, which, correctly, is the section $\eps \bilin ab \vac f(t) dg(t)$ in the new coordinate.)
\end{proof}

\subsection{Global Lie algebra}
Let us now pick a collection 
\be \bm x = \{x_1,\dots,x_N\} \subset \cp1 \nn\ee 
of $N\geq 1$ distinct points in $\cp1$. We call them the \emph{marked points}. 
Let $U\mapsto \Hsv(U)_\x$ denote the sheaf consisting of those of sections of $\Hsv$ that are holomorphic away from the marked points. We would like to describe the Lie algebra $\Hsv(\cp1)_{\x}$ of its global meromorphic sections. 
Pick a point $\8\in \cp1 \setminus \x$. Let $U_0 := \cp1 \setminus \8$. Pick a holomorphic coordinate 
\be z : U_0 \isom \CC.\nn\ee 
\begin{prop}\label{hpi} The Lie algebra $\Hsv(\cp1)_{\x}$ is ($\CC$-linearly) spanned by the elements
\be a_{-1}\vac dz  - \vac \frac{2\bilin \sv a dz}{z-z(x_i)} \qquad\text{and}\qquad \frac{a_{-1}\vac dz}{(z-z(x_i))^n},\qquad n\geq 1, \label{ddf1}\ee
with $a\in \h$ and $1\leq i\leq N$, together with the classes of the elements
\be  \frac{\vac d z}{z-z(x_i)} - \frac{\vac d z}{z-z(x_j)}, \qquad 1\leq i < j \leq N.\label{ddf}\ee
\end{prop}
\begin{proof}
Consider first the central extension: $\Hsv(\cp1)_{\x}$ fits into the short exact sequence, cf. \cref{laslem},
\be 0 \to H^1(\cp1)_\x \to \Hsv(\cp1)_{\x} \to \Hsv(\cp1)_{\x}/H^1(\cp1)_\x \to 0.\label{ges}\ee
Here $H^1(\cp1)_\x$ is the space of global meromorphic sections of the de Rham cohomology that are holomorphic away from the marked points (regarded as a  commutative Lie algebra). It is spanned by the given cohomology classes. (We need to take \emph{differences} of logarithmic differential forms as in \cref{ddf} to ensure that there is no pole at $\8$. Note that whenever $\sv\neq 0$, these differences are actually in the span of the elements in \cref{ddf1}.)

Now, for the remaining elements, fix an $a\in \h$ and consider the element of $\Hsv(U_0)_{\x}$ defined by the representative
\be a_{-1} \vac f(z) dz + \vac g(z) dz\label{gs}\ee 
in the local trivialization defined by $z$. We can ask under what conditions it corresponds to the restriction of a global section of $\Hsv(\cdot)_\x$. For that it must be regular at $\8$. Let $U_\8 := \cp1 \setminus 0$, where $0$ is the point with $z(0)= 0$. We have the holomorphic coordinate  $\zeta : U_\8 \isom \CC$ on $U_\8$ defined by $z= 1/\zeta$ on $U_0\cap U_\8$. In the trivialization of $\Hsv(U_\8){\x}$ defined by this coordinate $\zeta$, our section   is written
\be \left( a_{-1} \vac - \bilin \sv a \frac{2}\zeta \right)f\left(\frac 1 \zeta\right) d\zeta +
\vac g\left(\frac 1 \zeta\right) \frac {-1} {\zeta^2} d\zeta. \nn\ee 
For this to be regular at the point $\8$, where $\zeta(\8) = 0$, certainly $f$ must be regular at $\8$. (Consider the $a_{-1}\vac$ term). 
Moreover
\be \bilin \xi a  \frac{2}\zeta f\left(\frac 1 \zeta\right)+ g\left(\frac 1 \zeta\right) \frac {1} {\zeta^2} \nn\ee
must be regular at $\8$, modulo exact derivatives. We can regard this last as a condition on $g$, given $a$ and $f$: we need
\be g(z) \underset{z\to \8}\sim - 2 \bilin \sv a \frac{f(z)}{z} + h'(z) + \mc O(1/z^{2}), \nn\ee 
for some meromorphic function $h(z)$ with poles at most at the marked points. 

If $-2\bilin\sv a f$ is actually zero at $\8$ then any such $g$ is itself an exact derivative. 
If $-2\bilin\sv a f$ is not zero at $\8$, then (modulo exact derivatives) the solutions for $g$ are as in \cref{ddf1}.
\end{proof}

Let $\Hsv(\cp1)_\x^\8\subset \Hsv(\cp1)_{\x}$ denote the Lie subalgebra consisting of global sections that vanish at the point $\8$. 

Let $\hsv\subset \Hsv(\cp1)_{\x}$ denote the Lie subalgebra, isomorphic to $\h$, spanned by the elements
\be a_{-1}\vac dz  - \vac \frac{2\bilin \sv a dz}{z-z(x_1)},\nn\ee
with $a\in \h$. (Here, one could replace $x_1$ with $x_i$ for any one fixed $i$ with $1\leq i\leq N$.)

\begin{cor}\label{hpp}$ $  \begin{enumerate}[(i)]
\item As a Lie algebra, $\Hsv(\cp1)_{\x}$ is the direct sum 
\be \Hsv(\cp1)_{\x} = \Hsv(\cp1)_\x^\8 \oplus \hsv. \nn\ee
\item 
The Lie algebra $\Hsv(\cp1)_\x^\8$ does not depend on $\sv\in \h$.\\
Up to isomorphism, it is the central extension
\be 0 \to H^1(\cp1)_\x \to  \Hsv(\cp1)_\x^\8 \to \h\ox \K(\cp1)^\8_\x\to 0 \nn\ee 
(defined as in \cref{hcen}) of the Lie algebra $\h\ox \K(\cp1)^\8_\x$ of $\h$-valued meromorphic functions that vanish at $\8$ and have poles at most at the marked points $\x$, by the space $H^1(\cp1)_\x$ of global meromorphic sections of the de Rham cohomology with poles at most at the marked points $\x$. 
\end{enumerate}\qed
\end{cor}

\subsection{Local Lie algebras $\Hsv_{x_i}$}\label{sec: lla}
For this subsection, let us focus on one of the marked points $x_i$. Let $U$ be a small open disc containing $x_i$ and  $t:U\to \CC$ a local holomorphic coordinate at $x_i$. 

Let $\Hsv_{x_i}$ denote the stalk of the sheaf $\Hsv$ at $x_i$. 
Explicitly, in the coordinate $t$, $\Hsv_{x_i}$ is the $\CC$-linear span of elements of the form $a_{-1} \vac f(t) dt$ and $\vac f(t)dt$, with $a\in \h$ and now with $f(t)\in \CC\laurent t \cong \K_{x_i}$, modulo the $\CC$-linear span of elements of the form $\vac f'(t) dt$. The Lie bracket is given by the same formulas as in \cref{laslem}:
$[ \vac f(t) dt, \cdot ] := 0$ and
\begin{align} \left[ a_{-1} \vac f(t) dt,\, b_{-1} \vac g(t) dt \right] &:= \eps \bilin a b \vac f(t) g'(t) dt
. \nn\end{align}
The Lie algebra $\Hsv_{x_i}$ fits into the short exact sequence
\be 0 \to H^1_{x_i} \to \Hsv_{x_i} \to \Hsv_{x_i}/H^1_{x_i} \to 0,\nn\ee
where the central extension $H^1_{x_i} = \Gamma_{x_i}(\Omega)/d\K_{x_i}$ is of dimension one, with $\CC$-basis $dt/t$. We can identify $H^1_{x_i}$ with $\CC$ by taking the residue. 
In this way we see that there is a Lie algebra embedding
\be \Hsv_{x_i} \into \hh^\eps \label{locaff}\ee
given by 
\be a_{-1}\vac f(t) dt \mapsto a \ox f(t);\qquad \vac f(t) dt \mapsto \bm 1 \res_{x_i}  f(t) dt.\nn\ee
(It is an embedding rather than an isomorphism only because $\CC\laurent t \into \CC((t))$ is an embedding; cf. \cref{rem: holomorphic}.)

Let $\Hsvp_{x_i}$ (resp. $\Hsvpp_{x_i}$) denote the Lie subalgebra consisting of the germs at $x_i$ of sections of $\Hsv$ that are holomorphic (resp. vanishing) at $x_i$. There are Lie algebra embeddings
\be \Hsvp_{x_i} \into \h \ox \CC[[t]]_{hol}, \qquad  \Hsvpp_{x_i} \into \h \ox t\CC[[t]]_{hol}. \nn\ee
These embeddings are relative to our choice of local holomorphic coordinate $t$. In the special case $\sv = 0$ there are are canonical isomorphisms with the Lie algebras introduced in \cref{sec: local copies}: $\Hnought_{x_i} \cong \hh^\eps_{x_i}$, $\Hnoughtp_{x_i} \cong \h \ox \O_{x_i}$ and $\Hnoughtpp_{x_i} \cong \h \ox \m_{x_i}$, cf. \cref{dtr}. 

\subsection{Embedding of global into local}\label{sec: globloc}
For each marked point $x_i$, we have the canonical restriction homomorphism of Lie algebras $\Hsv(U) \to \Hsv_{x_i}$, 
for any open set $U$ containing $x_i$. In this way we get a canonical homomorphism of Lie algebras
$\Hsv(\cp1)_{\x} \to \bigoplus_{i=1}^N \Hsv_{x_i}$ of the Lie algebra $\Hsv(\cp1)_{\x}$ of global meromorphic sections into the direct sum of the $\Hsv_{x_i}$:
\be \Hsv(\cp1)_{\x} \to \bigoplus_{i=1}^N \Hsv_{x_i}. \nn\ee
\begin{lem} This is an embedding of Lie algebras.
\end{lem}
\begin{proof} We have to check injectivity. Injectivity is immediate for elements of the form in \cref{ddf1} (since no two distinct meromorphic functions define the same germ at $x_i$, for any one of the marked points $x_i$).  
For the central extension $H^1(\cp1)_{\x}$, i.e. the span of the elements in \cref{ddf}, we must be slightly more careful: the logarithmic differential form $d\log(z-z(x_i))$ maps to a nonzero cohomology class in the cohomology of the punctured disc at $x_i$, but is cohomologically trivial in the cohomology of the punctured disc about any other point. But since we included all the marked points in the direct sum on the right, we do get injectivity.
\end{proof}

\begin{prop}
As vector spaces,
\be \bigoplus_{i=1}^N \Hsv_{x_i} \cong_\CC \bigoplus_{i=1}^N \Hsvp_{x_i}\oplus \Hsv(\cp1)_\x^\8 \oplus \CC \mathsf c,\nn\ee
where $\mathsf c := \vac\sum_{j=1}^N d\log(z-z(x_j))$. 
\end{prop}

\begin{proof} 
We must show that every element $\bigoplus_{i=1}^N \Hsv_{x_i}$ decomposes uniquely into such summands. 
By linearity, it is enough to consider in turn elements of $\Hsv_{x_i}$, for each $i$. 
For its central extension, clearly $d\log(z-z(x_i))$ is in the span of the sum $\sum_{j=1}^N d\log(z-z(x_j))$ together with the differences in \cref{ddf}. 
It remains to consider the element of $\Hsv_{x_i}$ of the form $a_{-1} \vac f(z-z(x_i)) dz$ for some $a\in \h$ and Laurent series $f(t) \in \CC\laurent t$. The latter can be written uniquely as $f=f_++f_-$ for a series $f_+\in \CC[[t]]$ (with non-zero radius of convergence) and polynomial $f_- \in t^{-1}\CC[t^{-1}]$. Then $f_-(z-z(x_i))$ defines a meromorphic function on $\cp1$ that vanishes at $\8$ and so (in view of \cref{hpi}) $a_{-1} \vac f_-(z-z(x_i))dz$ belongs to $\Hsv(\cp1)_\x^\8$. So we get the unique decomposition of $a_{-1} \vac f(z-z(x_i)) dz\in \Hsv_{x_i}$ into the sum of $a_{-1}\vac f_+(z-z(x_i)) dz \in \Hsvp_{x_i}$ and $a_{-1} \vac f_-(z-z(x_i))dz \in \Hsv(\cp1)_\x^\8$.
\end{proof}

\begin{cor} We have
\begin{align} \bigoplus_{i=1}^N \Hsv_{x_i}\big/ \CC \mathsf c\,\, \cong_\CC\,\, &\bigoplus_{i=1}^N (\Hsv)_{x_i}^{+} + \Hsv(\cp1)_{\x}\nn\\\qquad\text{and}\qquad &\bigoplus_{i=1}^N (\Hsv)_{x_i}^{+} \cap \Hsv(\cp1)_{\x} = \hsv. \nn\end{align}
\end{cor}
\begin{proof} This follows from the proposition above and \cref{hpp}. \end{proof}

\subsection{The dual bundle $\left(\Pisv_{\leq 1}\right)^*$, and $\conn\sv$}\label{sec: conndef}
Consider the sheaf $U\mapsto \Gamma(U,\left(\Pisv_{\leq 1}\right)^*)$ of meromorphic sections of the dual bundle $\left(\Pisv_{\leq 1}\right)^*$. 
In the local trivialization of $\left(\Pisv_{\leq 1}\right)^*$ coming from a holomorphic coordinate $t:U\to \CC$, its sections look like $(F(t),\chi(t))$ with $F(t) \in \K(U)$ and $\chi(t) \in \h^*\ox\K(U)$, with the canonical pairing
\be\Gamma(U,\left(\Pisv_{\leq 1}\right)^*)\times\Gamma(U,\left(\Pisv_{\leq 1}\right)) \to \K(U)\nn\ee 
being given by 
\be \Big< (F(t), \chi(t)) , \left( f(t) \vac + g(t)_{-1} \vac\right)\Big> := F(t) f(t) + \la\chi(t),g(t)\ra \nn\ee
for $f(t) \in \K(U)$, $g(t) \in \h\ox\K(U)$.
Starting from the transition functions of the bundle $\Pisv_{\leq 1}$ from \cref{btrans2},  one finds\footnote{Using the transition functions \cref{btrans2}, we have that the pairing of these sections in the new coordinate is given by
\begin{align} \Big<  (\tilde F(s), \tilde\chi(s)) , \left( \tilde f(s) \vac + \tilde g(s)_{-1}\vac\right)\Big> 
& = \tilde F(s) \tilde f(s) + \la \tilde\chi(s), \tilde g(s) \ra \nn\\
& =  \tilde F(s) \left( f(\mu(s)) + \bilin\sv {g(\mu(s))} \frac{\mu''(s)}{\mu'(s)^2} \right) 
  + \la \tilde\chi(s), \frac{g(\mu(s))}{\mu'(s)}\ra  \nn\\
& =  \tilde F(s) f(\mu(s)) + \left( \frac{\la\tilde\chi(s),\cdot\ra}{\mu'(s)} + \tilde F(s)  \frac{\mu''(s)}{\mu'(s)^2}  \bilin\sv\cdot \right)\left(g(\mu(s))\right).
\nn\end{align}
This pairing is a function, so it must be the same as the pairing in the old coordinate, i.e. it must equal 
\be F(t) f(t) + \la\chi(t),g(t)\ra  = F(\mu(s)) f(\mu(s)) + \la\chi(\mu(s)),g(\mu(s))\ra .\nn\ee
Since this must be true for all $f$ and $g$, we conclude first that $\tilde F(s) = F(\mu(s))$ and hence that 
\be \tilde \chi(s)= \mu'(s)\chi(\mu(s))  - F(\mu(s))\frac{\mu''(s)}{\mu'(s)} \bilin\sv\cdot.\nn\ee 
} that in the local trivialization of $\left(\Pisv_{\leq 1}\right)^*$ coming from a new holomorphic coordinate $s:U\to \CC$ with $t=\mu(s)$, the same section of $\left(\Pisv_{\leq 1}\right)^*$ is written $(\tilde F(s), \tilde \chi(s))$ where
\be \tilde F(s) := F(\mu(s)) , \qquad \tilde \chi(s):= \mu'(s)\chi(\mu(s))  - F(\mu(s))\frac{\mu''(s)}{\mu'(s)} \bilin\sv\cdot.\nn\ee 
Since the coefficent function $F(t)=F(\mu(s))$ is the same in all such trivializations,  $\Gamma(\left(\Pisv_{\leq 1}\right)^*)$ fibres over $\K$. 
In particular, we have the fibre over the constant function $1\in \K$, consisting of sections whose trivializations in any local coordinate $t:U\to \CC$ take the form $(1,\chi(t))$, $\chi(t) \in \h^*\ox\K(U)$. Let us write $\conn\sv$ for the sheaf of such sections. Its transition functions are given by
\be \tilde \chi(s):= \mu'(s)\chi(\mu(s))  - \frac{\mu''(s)}{\mu'(s)} \bilin\sv\cdot.\label{chitrans}\ee

\subsection{The modules $\CC v_\chi$}\label{sec: cchi}
Let $x_i$ be any of the marked points on $\cp1$. Let $\chi$ be a germ at $x_i$ of a meromorphic section belonging to $\conn\sv\subset \Gamma(\left(\Pisv_{\leq 1}\right)^*)$. There is a one-dimensional representation $\CC v_\chi$ of the Lie algebra $\Hsv_{x_i}$ of \S\ref{sec: lla}, defined as follows. An element $f\in \Hsv_{x_i}$ is represented by a germ at $x_i$ of a meromorphic section of $\Pisv_{\leq 1}\ox \Omega$. Using the canonical pairing, we obtain a germ $\la \chi,f\ra$ of a meromorphic section of $\Omega$. It is ambiguous up to the addition of exact derivatives $dg$, $g\in \K_{x_i}$, because of the freedom to add terms of the form $\vac dg$ to $f$, cf. \S\ref{sec: lla}. Nonetheless, the residue is well-defined, and we set  
\be f \on v_\chi := v_\chi \res_{x_i} \la \chi, f\ra .\nn\ee
(Note that this would not work if $\chi$ were a germ of a section of $\left(\Pisv_{\leq 1}\right)^*$ belonging to the fibre over non-constant $F\in \K_{x_i}$, for then $\la \chi,f\ra$ would be defined only up to addition of terms of the form $Fdg$, whose residue could be nonzero.)

Suppose we pick a collection $\bm \chi := (\chi_i)_{i=1}^N$ of such germs, one for each marked point $x_i$. We obtain the one-dimensional representation
\be \CC v_{\bm\chi} := \bigotimes_{i=1}^N \CC v_{\chi_i} \nn\ee
of the Lie algebra $\Hsvloc$. 
On pulling back by the embedding of \S\ref{sec: globloc}, it is also a one-dimensional representation of $\Hsv(\cp1)_{\x}$. 

Let $U\mapsto \conn\sv(U)_\x$ denote the sheaf consisting of those of sections of $\conn\sv$ that are holomorphic away from the marked points $\bm x = \{x_1,\dots,x_N\}$. 
\begin{lem} As a module over $\Hsv(\cp1)_{\x}$, $\CC v_{\bm\chi}$ is trivial (i.e. $A\on v_{\bm\chi} = 0$ for all $A\in \Hsv(\cp1)_{\x}$) if and only if there exists a global section $\chi\in \conn\sv(\cp1)_\x$ such that $\chi_i$ is the germ at $x_i$ of $\chi$, for $1\leq i\leq N$. 
\end{lem}
\begin{proof} Consider the ``if'' direction. 
Suppose such a global section $\chi$ exists. Let $f\in \Hsv(\cp1)_{\x}$. By definition $\la \chi, f\ra\in \Gamma_\cp1(\Omega)$, i.e. it is a meromorphic one-form on $\cp1$, with no poles in $\cp1\setminus\{\x\}$. 
So we have
\be f \on v_{\bm \chi} = v_{\bm \chi} \sum_{i=1}^N \res_{x_i} \la \chi, f\ra =0 \nn\ee 
where the second equality holds by the residue theorem. The ``only if'' direction is the strong residue theorem (cf. \cite[\S9.2.9]{FrenkelBenZvi})  together with the non-degeneracy of the pairing.
\end{proof}

The quotient $\CC v_{\bm\chi} \big/ \Hsv(\cp1)_{\x}:= \CC v_{\bm\chi} \big/ \left(\Hsv(\cp1)_{\x}\on \CC v_{\bm \chi}\right)$ is called the space of coinvariants. 
\begin{cor} \label{cor: coinv}
\be \CC v_{\bm\chi} \big/ \Hsv(\cp1)_{\x} \cong_\CC \begin{cases} \CC & \text{if all $\chi_i$ are the restrictions of one global section $\chi\in \conn\sv(\cp1)_\x$} \\ 0 & \text{otherwise} \end{cases}\nn\ee 
\qed\end{cor}
Henceforth, we specialize exclusively to the case in which the $\chi_i$ are the restrictions of one global section $\chi\in \conn\sv(\cp1)_\x$.
\subsection{Global sections of $\conn\sv(\cp1)_\x$}
Let us describe these global sections $\chi$.
\begin{prop}\label{globsec} The set of global sections $\conn\sv(\cp1)_\x$ is in one-to-one correspondence with the set of meromorphic $\h^*$-valued functions of the form
\be \chi(z) = \sum_{i=1}^N \sum_{k=0}^{K_i} \frac{\chi_{i,k}}{(z-z(x_i))^{k+1}}, \qquad K_i \in \ZZ_{\geq 0},\quad \chi_{i,k}\in \h^*,\label{chiform}\ee
such that the elements $\chi_{i,0}\in \h^*$ satisfy the constraint 
\be \sum_{i=1}^N \chi_{i,0} =  2 \bilin\sv\cdot. \nn\ee
\end{prop}
\begin{proof} We have the patches $U_0$ and $U_\8$ and coordinates $z:U_0\isom \CC$, $\zeta: U_\8\isom \CC$ with $z=1/\zeta$ on $U_0\cap U_\8$, as above. A section of $\conn\sv(U_0)_\x$ is given in the trivialization corresponding to the coordinate $z$ by a meromorphic $\h^*$-valued function $\chi(z)$ of the form 
\be \chi(z) =   f(z) + \sum_{i=1}^N \sum_{k=0}^{K_i} \frac{\chi_{i,k}}{(z-z(x_i))^{k+1}}\nn\ee
where $K_i \in \ZZ_{\geq 0}$, $\chi_{i,k}\in \h^*$, and where $f(z) \in \h^*[z]$ is a polynomial. 
Over $U_0\cap U_\8$ the same section is given in the trivialization corresponding to the coordinate $\zeta$ by the function
\be \tilde\chi(\zeta) = - \frac{1}{\zeta^2} f\left(\frac 1\zeta\right) - \frac{1}{\zeta^2} \sum_{i=1}^N \sum_{k=0}^{K_i} \frac{\chi_{i,k}}{(1/\zeta-z(x_i))^{k+1}} +\frac 2 \zeta \bilin\sv \cdot \nn\ee
and this function is regular at $\zeta = 0$ (so that the section was in fact the restriction of a section in $\conn\sv(\cp1)_\x$) if and only if the given condition on the $\chi_{i,0}$ holds and $f(z) =0$.
\end{proof}

\subsection{Properties of coinvariants}
Let us consider a more general space of coinvariants. Suppose $\M_i$ is a smooth module over the Lie algebra $\Hsv_{x_i}$ for each marked point $x_i$. As shorthand, we shall write \be\bm 1_x := \vac d \log(z-z(x))\nn\ee for any $x\in \cp1$. Suppose that the central element $\bm 1_{x_i} \in \Hsv_{x_i}$ acts as 1 on $\M_i$, for each $i$. That is,  assume the modules $\M_i$ are all of level one. We have the space of coinvariants
\be H(\cp1,(x_i),(\M_i))_{i=1}^N := \left.\M_{1} \ox \dots \ox\M_{N} \right/ \Hsv(\cp1)_\x.\nn\ee
As an aside, note that according to \cref{hpp}, this space is non-trivial only if $\Delta^N a$  acts on $\M_{1} \ox \dots \ox\M_{N}$ by scalar multiplication by $2\bilin \sv a$, for all $a\in \h$,\footnote{Here, each $\M_i$ is in particular a module over the subalgebra of zero modes, $\h\into\Hsv_{x_i}; a\mapsto a_{-1} \vac d(z-z(x_i))$, and $\Delta^Na$ denotes the $N$-fold coproduct.} which is consistent with the constraint in \cref{globsec}. 

Now consider adding an additional marked point, $x$, in $\cp1\setminus\{\8\}$, distinct from the $\x=\{x_1,\dots,x_N\}$. 

The fibre $\Pisv_x$ of $\Pisv$ at $x$ is canonically a module over $\Hsv_x$. Namely, let $\CC\vac_x$ be the trivial one-dimensional module over the Lie algebra $\Hsvp_x$ of germs at $x$ of holomorphic sections of $\Hsv$. We make it into a module over $\Hsvp_x\oplus \CC \bm 1_x$ by declaring that $\bm 1_x$ acts as one. 
Then the fibre $\Pisv_x$ is the induced module
\be \Pisv_x \cong U(\Hsv_x) \ox_{U(\Hsvp_x\oplus \CC\bm 1_x)} \CC \vac_x. \nn\ee

We have the $(\Hsvloc\oplus \Hsv_x)$-module
\be  \M_{1} \ox \dots \ox\M_{N} \ox \Pisv_x  \nn\ee
and hence the space of coinvariants
\be  \left(\M_{1} \ox \dots \ox\M_{N} \ox \Pisv_x\right) \big/ \Hsv(\cp1)_{\x\cup\{x\}}.
\nn\ee
\begin{prop} There are canonical vector-space isomorphisms 
\begin{align}  \left(\M_{1} \ox \dots \ox\M_{N} \ox \Pisv_x\right) \big/ \Hsv(\cp1)_{\x\cup\{x\}} 
&\cong_\CC 
\left(\M_{1} \ox \dots \M_{N}\ox \CC\vac \right) \big/  \Hsv(\cp1)_\x \nn\\
&\cong_\CC 
\M_{1} \ox \dots \ox\M_{N} \big/ \Hsv(\cp1)_{\x}.  
\nn\end{align}
\end{prop}
\begin{proof}
As vector spaces, $\Hsv_x \cong_\CC \Hsv(\cp1)_{\{x\}}^\8 \oplus \left(\Hsvp_x\oplus \CC \bm 1_x\right)$. Therefore $\Pisv_x$ is free as a module over $\Hsv(\cp1)_{\{x\}}^\8$, and hence so too is $\M_{1} \ox \dots \ox\M_{N} \ox \Pisv_x$. Thus
\be \M_{1} \ox \dots \ox\M_{N} \ox \Pisv_x \cong_\CC U(\Hsv(\cp1)_{\{x\}}^\8) \on \left(\M_{1} \ox \dots \ox\M_{N} \ox \CC\vac_x\right).\nn\ee
Now, observe that
\begin{align} \Hsv(\cp1)_{\x \cup\{x\}}^\8 &\cong_\CC \Hsv(\cp1)_\x^\8 \oplus \Hsv(\cp1)_{\{x\}}^\8  \oplus \CC (\bm 1_{x_1} - \bm 1_x) . \nn\end{align}
Also, $\Hsv(\cp1)_\x^\8$ acts as zero on  $\vac_x$, and $(\bm 1_{x_1} - \bm 1_x)$ acts as $1-1=0$ on all of $\Pisv \ox \M_1\ox\dots\M_N$. It follows, as in e.g. the proof of \cite[Proposition 3.1]{VY2}, that there are canonical vector-space isomorphisms 
\begin{align}  \left(\M_{1} \ox \dots \ox\M_{N} \ox \Pisv_x\right) \big/ \Hsv(\cp1)_{\x\cup\{x\}}^\8 
&\cong_\CC 
\left(\M_{1} \ox \dots \M_{N}\ox \CC\vac \right) \big/  \Hsv(\cp1)_\x^\8 \nn\\
&\cong_\CC 
\M_{1} \ox \dots \ox\M_{N} \big/ \Hsv(\cp1)_{\x}^\8.  
\nn\end{align}
We have $\Hsv(\cp1)_{\x} = \Hsv(\cp1)_\x^\8 \oplus \hsv$ and $\Hsv(\cp1)_{\x\cup\{x\}} = \Hsv(\cp1)_{\x\cup\{x\}}^\8 \oplus \hsv$,  as in \cref{hpp}. So it remains to consider the action of $\hsv$. Since $a_{-1} \vac dz$, $a\in \h$, is a zero mode it acts trivially on $\Pisv_x$; and $- 2\bilin \sv a \bm 1_{x_1} $ by definition acts only on the tensor factor $\M_1$. So $\hsv$ does not act at all on the tensor factor $\Pisv_x$, and we have the result.
\end{proof}
Therefore we obtain a linear map 
\be \Pisv_x \to \End\left( H(\cp1,(x_i),(\M_i))_{i=1}^N\right),\nn\ee
which sends $X\in \Pisv_x$ to the endomorphism of the space of coinvariants given by
\be [v] \mapsto [v \ox X] .\nn\ee 
Given a meromorphic section $\sigma$ of $\Pisv$ over some open set $U$, we get for each $x\in U\setminus\x$ the linear map $[v] \mapsto [v\ox \sigma|_x]$. Note that, for a given section $\sigma$, this is by construction a well-defined function of $x$ (for example, it cannot depend on any choice of trivialization).

This linear map depends meromorphically on $x$ with poles at most at the marked points $\x$.  
(Cf. the explicit calculation below.)
\begin{cor} There is a well-defined map
\be  \Gamma(U,\Pisv) \to  \End\left( H(\cp1,(x_i),(\M_i))_{i=1}^N\right) \ox \K(U) \nn\ee
\qed\end{cor}

Now we specialise back to the case $\M_i = \CC v_{\chi_i}$ as above. 
The dimension of the space of coinvariants is then one, and the linear map is just a rescaling by a complex factor.
So we get a map $ \Gamma(U,\Pisv) \to \CC\ox \K(U) = \K(U)$. This map will also depend on our choice of $\chi\in \conn\sv$. Let us compute this map explicitly, working with the global coordinate 
\be z : U_0 \isom \CC\nn\ee 
and the corresponding identifications of $\Hsv_x$ with $\hh^\eps$ and $\Pisv_x$ with $\pi^\eps_0$ for all $x\in U_0 = \cp1\setminus \{\8\}$. Since, by definition of coinvariants,
\be 0 = \left[ \left( \frac{b_{j}}{(z-z(x))^n} \right) \on ( v\ox X) \right] \nn\ee
for all $n\geq 1$, we have
\begin{align} [v \ox b_{j,-n} X] 
&= - [v\ox X] \sum_{i=1}^N \res_{x_i}  \la \chi_i, \iota_{x_i} \frac{b_{j}}{(z-z(x))^n}\ra\nn\\
&= -[v\ox X] \frac 1 {(n-1)!} \left(\frac{\del}{\del z(x)}\right)^{n-1}  \sum_{i=1}^N \res_{x_i}  \la \chi_i, \iota_{x_i} \frac{b_{j}}{z-z(x)}\ra\nn\\
&= [v\ox X] \frac 1 {(n-1)!} \left(\frac{\del}{\del z(x)}\right)^{n-1} \sum_{i=1}^N \sum_{k=0}^\8  \frac{ \res_{x_i}   \la \chi_i, b_{j} (z-z(x_i))^k \ra}{(z(x) - z(x_i))^{k+1}}.
 \nn\end{align}
Recall that we are assuming that $\bm\chi = (\chi_i)_{i=1}^N$ are the restrictions of a global section in $\conn\sv(\cp1)_\x$, given, in the trivialization of $\conn\sv(U_0)_\x$ defined  by the $z$ coordinate, by a meromorphic $\h^*$-valued function $\chi(z)$ as in \cref{chiform}. So $ \res_{x_i}\chi_i  (z-z(x_i))^k = \chi_{i,k}$, and we have established that
\begin{align} [v \ox b_{j,-n} \atp z X] 
&= [v\ox X] \la \chi^{(n)}(z), b_j \ra,
 \label{bswap}\end{align}
where we adopt the shorthand $\chi^{(n)}(z) := \frac 1 {(n-1)!} \left(\frac{\del}{\del z}\right)^{n-1} \chi(z)$. 
Consider now an arbitrary section of $\Gamma(U_0,\Pisv)_\x$: such a section takes the form, in the trivialization coming from the coordinate $z$,
\be  \sum_{m=0}^M f^{j_1,\dots,j_m}_{n_1,\dots,n_m}(z) b_{j_1,-n_1}\dots b_{j_m,-n_m} \vac\nn\ee
where the functions $f^{j_1,\dots,j_m}_{n_1,\dots,n_m}(z)$ are meromorphic (and we continue to employ summation convention on the indices $j_k$). 
We get
\begin{align} &\sum_{m=0}^M f^{j_1,\dots,j_m}_{n_1,\dots,n_m}(z) [v \ox b_{j_1,-n_1}\dots b_{j_m,-n_m}\atp z \vac ] \nn\\
= &[v] \sum_{m=0}^M f^{j_1,\dots,j_m}_{n_1,\dots,n_m}(z) \la \chi^{(n_1)}(z), b_{j_1} \ra\dots\la \chi^{(n_m)}(z), b_{j_m} \ra .
 \nn\end{align}
(Here, to stress the point, both $f^{j_1,\dots,j_m}_{n_1,\dots,n_m}(z)$ and $\la \chi^{(n_1)}(z), b_{j_1} \ra$ have non-trivial transformation properties under changes in coordinate, but the expression above is by construction a function.)

We have established the following.
\begin{prop} \label{secpair}
There is a well-defined map, $\K(U)$-linear in the first slot,
\be \coinv : \Gamma(U,\Pisv)_\x \times \conn\sv(\cp1)_\x \to \K(U)_\x; \qquad (\sigma,\chi) \mapsto \coinv_\chi(\sigma)  \nn\ee
defined by 
\be [v_{\bm \chi} \ox \sigma|_x] =   [v_{\bm \chi}]\,\, \coinv_\chi(\sigma)|_x ,\nn\ee 
and given explicitly in the global coordinate $z:U_0\isom \CC$ by 
\begin{align}
&\coinv_\chi\left(\sum_{m=0}^M f^{j_1,\dots,j_m}_{n_1,\dots,n_m}(z) b_{j_1,-n_1}\dots b_{j_m,-n_m} \vac\right)
\nn\\&
= \sum_{m=0}^M f^{j_1,\dots,j_m}_{n_1,\dots,n_m}(z) \la \chi^{(n_1)}(z), b_{j_1} \ra\dots\la \chi^{(n_m)}(z), b_{j_m} \ra .\nn\end{align}\qed
\end{prop}

\section{The bundle $\Pic$ and the sheaf $\conn{-\chweyl}$}\label{sec: tto}
After the general discussion in the previous section about coinvariants of $\hh$-modules on the Riemann sphere, we can now specialise again to our real case of interest and prove the main global results, \cref{act} and \cref{tkIntro}, from the introduction.  

Namely we suppose that $\h$ is the Cartan subalgebra of a Kac-Moody algebra $\g$ of affine type; we go to the $\eps\to 0$ classical limit $\picc_0$ of $\pi^\eps_0$; and we choose the conformal structure coming from the choice $\sv = - \chweyl$.
 
To de-clutter the notation, let us write
\be \Pic := {}^{-\chweyl}\Pi^0 \nn\ee
for the bundle $\Pisv$ with this choice, $\sv = - \chweyl$, in the limit $\eps\to 0$. We also write, by a slight abuse of notation,
\be \conn{-\chweyl} := \conn{-\chweyl} \nn\ee 
for the sheaf defined in \cref{sec: conndef}. (In the next section, in \cref{mopconn}, we shall see that this sheaf is isomorphic to the sheaf of Miura opers.)

We shall show that the bundles $\Pi\ox \Omega^j$ admit a one-parameter family of flat connections, defined using our local notion of canonical translation from \cref{sec: caf} above; then we check that the map $\coinv_\chi$ behaves well with respect to these connections. 

\subsection{From  $\conn{-\chweyl}$ to affine connections}\label{sec: conrho}
By definition, \cref{chitrans}, $\conn{-\chweyl}$ is the sheaf of meromorphic sections of the vector bundle with fibre isomorphic to $\h^*\cong\h$ and transition functions given by
\be \chi(t) \mapsto \tilde \chi(s)= \mu'(s)\chi(\mu(s))  +  \frac{\mu''(s)}{\mu'(s)} \chweyl\label{conrhotrans}\ee
between the trivializations $\chi(t)$ and $\tilde\chi(s)$ coming from local holomorphic coordinates $t$ and $s$ respectively, with $t=\mu(s)$. 
We continue to write $\hr$ for the subspace of $\h$ spanned by the simple (co)roots:
\be \hr:= \bigoplus_{i\in I} \CC \chal_i.\nn\ee
Then we have the direct sum decomposition 
\be \h \cong \hr\oplus \CC\chweyl \nn\ee
of $\h$, and the corresponding decomposition of our section $\chi(t)$ of $\conn{\chweyl}$:
\be \chi(t) = \sum_{i\in I} \chi_i(t) \chal_i + \frac{\twist(t)}{\coxeter} \chweyl,\nn\ee
where
$\chi_i(t) :=  \la \chi(t),\La_i \ra$ and $\twist(t) := \la \chi(t), \chcent\ra$.
From \cref{conrhotrans} we see that $\chi_i(t)$ transforms like a section of $\Omega$ for each $i$, while $-\twist(t)/\coxeter$ transforms as follows:
\be -\frac{\twist(t)}{\coxeter} \mapsto -\frac{\tilde \twist(s)}\coxeter =  -\mu'(s)\frac{\twist(\mu(s))}\coxeter  -  \frac{\mu''(s)}{\mu'(s)}.\label{twisttrans}\ee
This is\footnote{Indeed, consider a section $f(t) dt = \tilde f(s) ds$ of $\Omega$. Here $\tilde f(s) = f(t) \mu'(s)$. Its derivative should be a section of $\Omega\ox \Omega$ so we must have $\tilde f'(s) + \tilde A(s) \tilde f(s) = (f'(t)) + A(t) f(t)) \mu'(s)^{2}$, whereas in fact
$\tilde f'(s) + \tilde A(s) \tilde f(s) 
= \del_s\big( f(t) \mu'(s)\big) + \tilde A(s) f(t) \mu'(s)
= f'(t) \mu'(s)^{2} + f(t) \mu''(s)+ \tilde A(s) f(t) \mu'(s)$.
On comparing the two, we obtain the transformation rule for the component $A(t)$ of the connection given in \cref{twisttrans}.} the transformation property of the component of a connection
\be \Gamma(U, \Omega) \to \Gamma(U, \Omega \ox \Omega) ,\qquad
 f(t) dt \mapsto \left(f'(t)-\frac 1 \coxeter\twist(t)f(t)\right) dt. \nn\ee
on the canonical bundle $\Omega$.
Let $U\mapsto \Conn(U,\Omega)$ denote the sheaf of meromorphic connections on $\Omega$, i.e. of meromorphic \emph{affine connections}. 
We have established the following.
\begin{lem} There is an isomorphism
\be \conn{-\chweyl}(U)\,\, \cong\,\, \hr \ox\Gamma(U, \Omega)\,\, \oplus\,\, \Conn(U,\Omega). \nn\ee
\qed\end{lem}
In particular, we have a surjection
\be \conn{-\chweyl}(\cp1)_\x \onto \Conn(\cp1,\Omega)_\x; \quad 
                     \chi \mapsto \Tkn_\chi \label{tknchidef}\ee
which takes a global section $\chi$ of $\conn{-\chweyl}(\cp1)_\x$ and produces a meromorphic affine connection (holomorphic away from the marked points $\x$) which we denote by $\Tkn_\chi$. 
It is easy to check that if $-\twist(z)/\coxeter$ is the component of a connection on $\Omega$ then $-j\twist(z)/\coxeter$ is the component of a connection on the $j$th tensor power $\Omega^j:= \Omega^{\ox j}$. In this way, the affine connection $\Tkn_\chi$ allows us to differentiate sections of $\Omega^j$ for all $j$.  

Now we will introduce an ``operator'' version $\Tkn$ of this connection $\Tkn_\chi$, which will not depend on a choice of $\chi$. 

\subsection{Connections on $\Pic\ox \Omega^j$} \label{sec: con}
Recall the canonical translation operator $\Tk$ of \cref{sec: tkdef}. 

\begin{thm} Let $\alpha\in \CC$. 
\begin{enumerate}[(i)]
\item
There is a well-defined flat holomorphic 
connection $\Gamma(\cdot,\Pic) \to \Gamma(\cdot,\Pic \ox \Omega)$ on $\Pic$ given by
\be  \sigma(t) \mapsto \left(\sigma'(t) + (L_{-1}- \alpha \Tk) \sigma(t)\right) dt .\nn\ee
\item\label{tkndef} More generally, 
 there is a flat holomorphic connection 
\be\Gamma(\cdot, \Pic \ox \Omega^j) \to \Gamma(\cdot, \Pic\ox \Omega^{j+1})\nn\ee
on $\Pic\ox\Omega^j$ for each integer $j$, given by 
\be \sigma(t) dt^j \mapsto \left(\sigma'(t) + (L_{-1} - \alpha\Tk - j \frac{\chcent_{-1}}\coxeter ) \sigma(t)\right) dt^{j+1}. \nn\ee
\end{enumerate}
\end{thm}
\begin{proof}
In \cref{sec: con} we described the affine space of flat holomorphic connections on $\Pisv$ and hence on $\Pic$. In view of the formula \cref{Rsv} for $R(\mu)$, we see that if $\mu\in \Aut\O$ then $v_0 = \mu'(0)$.  So \cref{lem: Lj} and \cref{affcor} together yield part (i).

Consider part (ii). Let $\tilde \sigma(s)$ be the section $\sigma(t)$ in the $s$ coordinate. We know from part (i) that $\sigma'(t)dt + (L_{-1} - \alpha\Tk) \sigma(t) dt$ is a well-defined section of $\Pic \ox \Omega$ given, in the $s$ coordinate, by $\tilde\sigma'(s)ds + (L_{-1} - \alpha\Tk) \tilde \sigma(s)ds$. 
Hence the would-be derivative, 
\begin{align} &\left(\sigma'(t) + (L_{-1} - \alpha\Tk - j \frac{\chcent_{-1}}\coxeter ) \sigma(t)\right) dt^{j+1}\nn\\
 &= \left( \sigma'(t)dt + (L_{-1} - \alpha\Tk) \sigma(t) dt \right) dt^j - j \frac{\chcent_{-1}}\coxeter \sigma(t) dt^{j+1} \nn\end{align}
is given in the $s$ coordinate by
\begin{align}
 & \left( \tilde\sigma'(s)ds + (L_{-1} - \alpha\Tk) \tilde \sigma(s)ds \right) \mu'(s)^j ds^j
      - j  \left(\frac{\chcent_{-1}}{\coxeter \mu'(s)} - \frac{\mu''(s)}{\mu'(s)^2} \right) \tilde\sigma(s) \mu'(s)^{j+1} ds^{j+1} \nn.
\end{align}
Here we used the fact that $R(\mu_s) \chcent_{-1} R(\mu_s)^{-1} = \frac{\chcent_{-1}}{\coxeter \mu'(s)} - \frac{\mu''(s)}{\mu'(s)^2} + \dots$ where $\dots$ are terms with non-negative modes of $\chcent$; these act as zero on $\picc_0$. 
On the other hand, in the $s$ coordinate, $\sigma(t) dt^j$ becomes $\tilde \sigma(s) \mu'(s)^j ds^j$ and its derivative should, therefore, be
\begin{align}
\left(\tilde\sigma'(s)\mu'(s)^j  + j \tilde\sigma(s)\mu'(s)^{j-1} \mu''(s) 
+ (L_{-1} - \alpha\Tk - j \frac{\chcent_{-1}}\coxeter ) \tilde \sigma(s) \mu'(s)^j \right) ds^{j+1}.
\nn\end{align}
The expressions above agree, which completes the proof of part (ii). 
\end{proof}

Define $\Tkn$ to be the flat holomorphic connection obtained by taking $\alpha = 0$ in part (ii) of the theorem. That is, $\Tkn$ is the connection
\be \Tkn: \Gamma(\cdot,\Pic\ox\Omega^j) \to \Gamma(\cdot,\Pic \ox \Omega^{j+1}) \nn\ee
given by
\begin{align} \Tkn \sigma(t) dt^j &:= \left(\sigma'(t) + \left(L_{-1} - j \frac{\chcent_{-1}}\coxeter\right) \sigma(t)\right) dt^{j+1} .\nn\end{align}
Equivalently $\Tkn \sigma = \nabt \sigma - j \frac{\chcent_{-1}}\coxeter \sigma dt$
where $\nabt$ is the connection from \cref{lem: nab}. 

\subsection{Functoriality}
By \cref{secpair} we have, for each $\chi \in \conn{-\chweyl}(\cp1)_\x$, $\K(U)$-linear maps
\be \coinv_\chi: \Gamma(U,\Pic)_\x \to \K(U)_\x  \nn\ee
and hence for all $j\in \ZZ$,
\be \coinv_\chi: \Gamma(U,\Pic\ox \Omega^j)_\x \to \Gamma(U,\Omega^j)_\x . \nn\ee
Recall the definition of $\Tkn_\chi$ from \cref{tknchidef}.

\begin{thm} \label{functhm}
For any section $\sigma\in \Gamma(U,\Pic\ox \Omega^j)$, $j\in \ZZ$, and for any connection $\chi\in \conn{-\chweyl}(\cp1)_\x$,
\be \coinv_\chi \left(\Tkn \sigma\right)  = \Tkn_\chi \coinv_\chi\left(\sigma\right). \nn\ee
\end{thm}
\begin{proof} 
It is enough to consider a coordinate patch $U$ such that $U \subset U_0 = \cp1\setminus \{\8\}$ or $U \subset U_\8 = \cp1 \setminus \{0\}$. Without loss of generality (by our choice of what to call $0$ and what $\8$) suppose $U\subset U_0$.

Let us work in the restriction to $U$ of the global coordinate $z:U_0\isom \CC$. 
From the explicit expression in \cref{secpair} we see that $\coinv_\chi\left(\nabt_{\del_z} \sigma\right) = \del_z\coinv_\chi(\sigma)$. At the same time, by \cref{bswap} we know that $\coinv_\chi(\chcent_{-1} v dt^j) = \la \chi(z),\chcent\ra \coinv_\chi(vdt^j) = \twist(z) \coinv_\chi(vdz^j)$. 
Hence indeed
\begin{align} \coinv_\chi(\Tkn_{\del_z} \sigma(z)dz^j ) &= \coinv_\chi(\nabt_{\del_z} \sigma(z)dz^j - \frac j\coxeter \chcent_{-1} \sigma(z)dz^j) \nn\\
&= \left(\del_z - \frac j\coxeter \twist(z)\right) \coinv_\chi(\sigma(z)dz^j) = \left(\Tkn_\chi\right)_{\del_z} \coinv_\chi(\sigma(z)dz^j).\nn\end{align} 
\end{proof}
\begin{rem}\label{flem}
In the special case $j=0$ we recover the statement that, for any section $\sigma\in \Gamma(U,\Pic)$ and for any connection $\chi\in \conn{-\chweyl}(\cp1)_\x$,
\be \coinv_\chi(\nabt \sigma) = d\coinv_\chi(\sigma) .\nn\ee
\end{rem}

Now, for $j\geq 0$, let $H^1(\cdot,\Pic\ox\Omega^j,\Tkn)$ denote the sheaf of the first de Rham cohomology of the connection $\Tkn$ with coefficients in $\Pic\ox\Omega^j$:
\be H^1(U,\Pic\ox \Omega^j, \Tkn) := \Gamma(U,\Pic \ox \Omega^j\ox \Omega) 
                     \Big/\Tkn \Gamma(U,\Pic\ox \Omega^j);\nn\ee
and let $H^1(\cdot,\Omega^j,\Tkn_\chi)$ denote the sheaf of the first de Rham cohomology of the connection $\Tkn_\chi$ with coefficients in $\Omega^j$:
\be H^1(U,\Omega^j, \Tkn) := \Gamma(U,\Omega^j\ox \Omega) 
                     \Big/\Tkn_\chi \Gamma(U, \Omega^j).\nn\ee
\begin{cor}\label{funccor}
For each $\chi\in \conn{-\chweyl}(\cp1)_\x$ we have well-defined $\K(U)$-linear map
\be [\coinv_\chi]: H^1(U,\Pic\ox \Omega^j,\Tkn)_\x \to H^1(U,\Omega^j, \Tkn_\chi)_\x . \nn\ee
\qed
\end{cor}

\begin{rem}\label{twistrem} Recall we denote by $-\twist(t)/\coxeter :=  -\la \chi(t),\chcent\ra/\coxeter \nn$ the component, in a local holomorphic chart $t: U \to \CC$, of the affine connection $\Tkn_\chi$ defined by a global section $\chi \in \conn{-\chweyl}(\cp1)_\x$, and we write $\twist^{(n)}(t) := \frac 1 {(n-1)!} \left(\frac{\del}{\del t}\right)^{n-1} \twist(t)$. 
\Cref{functhm} implies that
\be \coinv_\chi( \chcent_{-1} \vac )(t) = \coxeter\twist(t)   \label{swaptwist},\ee 
and, more generally, we have
\be \coinv_\chi( \chcent_{-n} v  )(t) = \coxeter\twist^{(n)}(t) \coinv_\chi(v)(t)  \nn\ee 
for any $v\in \picc_0$ and any $n\geq 1$.

Let us check these statements explicitly.
We know from \cref{bswap} that these statements hold in the restriction to $U$ of the global coordinate $z:U_0\isom \CC$. Let $z=\mu(s)$ where $s:U\to \CC$ is another local holomorphic coordinate. Let $\tau = \mu^{-1}$. So $z= \mu(\tau(z))$, $1= \mu'(s) \tau'(z)$ and hence $0=\mu''(s) \tau'(z) + \mu'(s)^2 \tau''(z)$. Recall the transition functions from \cref{btrans2}. The section given by $\chcent_{-1}\vac$ in  the coordinate $s$  is written, in the trivialization coming from the coordinate $z$,
\be \frac{1}{\tau'(z)} \left( \chcent_{-1} \vac - \coxeter \frac{\tau''(z)}{\tau'(z)} \vac \right)
  = \mu'(s) \left(  \chcent_{-1} \vac + \coxeter \frac{\mu''(s)}{\mu'(s)^2} \vac \right). \nn\ee
After applying $\coinv_\chi$, we get $\mu'(s) \coxeter\twist(z) + \coxeter \mu''(s)/\mu'(s)$ and, in view of \cref{twisttrans}, we recognise this as $\coxeter\tilde \twist(s)$. This establishes the statement about $\coinv_\chi( \chcent_{-1} \vac)(t)$ and more generally about $\coinv_\chi( \chcent_{-1} v)(t)$ for all $v\in \picc_0$, since the correction to the transition functions are by terms with non-negative modes of $\chcent$, which act as zero on all of $\picc_0$. For the same reason, to obtain the statement about $\coinv_\chi( \chcent_{-n} v )(t)$ it is enough to consider $\chcent_{-n}\vac$. And the fact that $\coinv_\chi( \chcent_{-n} \vac  )(t) = \coxeter\twist^{(n)}(t) $ follows from \cref{swaptwist} and \cref{flem}. 
\end{rem}

\subsection{Conformal primaries and global constant sections}
\begin{prop}\label{tktk}
Suppose $v\in \picc_0$ is a conformal primary of conformal weight $j$. Associated to $v$ is a well-defined global section of $\Pic\ox\Omega^j$ given, in the trivialization defined by any local holomorphic coordinate $t: U \to \CC$, by $v dt^j$.

This section $vdt^j$ obeys
\be \Tkn vdt^j = \left(L_{-1} - \frac j\coxeter\chcent_{-1}\right)  vdt^{j+1} = (\Tk v)dt^{j+1}\nn\ee
(so it is constant with respect to the connection $\Tkn-\Tk$). 
\ifdefined\short
\qed
\else
\begin{proof}
Consider the section of $\Pic$ which is given by $v$ in the trivialization coming from a holomorphic coordinate $t:U\to \CC$ (i.e the section $\triv_{U,t}^{-1}(U \times \{v\})$). In the coordinate $s$ with $t=\mu(s)$ the same section looks like $v \mu'(s)^{-j}$ (since $v$ is primary). At the same time, $ds^j = \mu'(s)^j dt^j$. So indeed the section of $\Pic\ox \Omega^j$ given by $v dt^j$ in the $t$ coordinate is given by $vds^j$ in the $s$ coordinate.
It is clear that $\Tkn vdt^j = (L_{-1} - j/\coxeter\chcent_{-1})  vdt^j$ by definition of $\Tkn$. The final equality follows by \cref{lem: tkprim}.
\end{proof}
\fi
\end{prop}
Recall that inside $\picc_0$ we have the subspace $\piaff_0 \subset \picc_0$
which is stable under the action of $\Aut\O$, and is in fact spanned by conformal primaries.
We defined 
\be \Faff_0 \subset \left(\piaff_0\ox \CC\vol\right)\Big/ (\Tk\vol) (\piaff_0)\nn\ee
cf. \cref{Tkdiagaff}. 
\begin{cor}\label{tktkcor} 
Suppose $v\in \piaff_0$ is a conformal primary of conformal weight $j+1$. 
It defines a class $[v\ox \vol]\in \Faff_0$. Associated to this class $[v\ox \vol]$ is a well-defined global section 
\be [v dt^{j+1}] \in  H^1(\cp1 ,\Pic\ox \Omega^j, \Tkn).\nn\ee 
\qed\end{cor}

Now we can apply the results above to the classes from \cref{vjthm}, 
\be [\hamd_j\ox \vol] \in \Iaff_0 = \ker \mc H \subset \Faff_0 .\nn\ee
Using \cref{tktkcor} we obtain well-defined global holomorphic sections of the cohomology of $\Tkn$,
\be [\hamd_j dt^{j+1} ] \in H^1(\cp1 ,\Pic\ox \Omega^j, \Tkn). \nn\ee
Then, for any meromorphic connection $\chi\in \conn{-\chweyl}(\cp1)_\x$, we obtain using \cref{funccor} well-defined global meromorphic sections of the cohomology of $\Tkn_\chi$:
\be [ \coinv_\chi(\hamd_j dt^{j+1}) ] \in  H^1(\cp1 ,\Omega^j, \Tkn_\chi)_\x \nn\ee

Now we can ask what special properties these cohomology classes possess, by virtue of the fact that $[\hamd_j\ox \vol]$ lay not just in $\Faff_0$ but in the kernel of $\mc H$; 
in other words, by virtue of the fact that these classes are invariant under the screening flows $Q_i$, $i\in I$, on $\picc_0$.

\section{$\g$-Opers and $\g$-Miura Opers}\label{sec: opers}
In this final section, we recall the definitions of opers and Miura opers, both local and global. That will allow us to establish the remaining results of the paper, \cref{Wpi} and \cref{fc}.  
We continue to suppose that $\g$ is a Kac-Moody algebra of affine type. 
Opers were first introduced on the disc in \cite{DS} and for a general Riemann surface in \cite{BDopers}. Miura opers for simple Lie algebras were introduced in \cite{Frenkel_2005}; affine opers and affine Miura opers were introduced in \cite{FFsolitons}. We follow the conventions from \cite[\S6]{LVY}.  

It is convenient to treat in parallel the cases of opers on the disc, and meromorphic opers on a coordinate patch $U\subset \cp1$.
Thus, let $\A$ denote either the ring $\CC[[t]]$ of formal power series in $t$, or the field $\K(t(U))$ of meromorphic functions on the image $t(U)\subset \CC$ of a holomorphic coordinate $t:U\to \CC$. 
In either case we have a notion of coordinate transformations. On the disc, these form the group $\Aut\O$ defined in  \cref{mudef,astdef}. In the meromorphic setting, we have the changes of holomorphic coordinate $\mu(s)$ between any pair of holomorphic coordinates $t,s:U\to \CC$ on $U$, which we think of as sending $f(t)\in \K(t(U))$ to $f(\mu(s))\in \K(s(U))$. (Since $s(U)$ and $t(U)$ need not coincide as subsets of $\CC$, the latter do not form a group, but only a groupoid.)

Let $\U(\A)$ denote the group of units of $\A$. We saw $\U(\O)$, the multiplicative group of invertible formal power series, already above. Since $\K(t(U))$ is a field, $\U(\K(t(U))) = \K(t(U)) \setminus\{0\}$.

\subsection{Completion of $\n_+(\A)$}

For any Lie algebra $\p$, we have the Lie algebra 
\be \p(\A):= \p \ox_\CC \A\nn\ee 
(of $\p$-valued functions on the disc, or $\p$-valued meromorphic functions on $t(U)$). 

Let $\n_k$ denote the Lie ideal in $\n_+$ spanned by elements of grade $n\geq k$. We get a descending $\ZZ_{>0}$-filtration of $\n_+(\A)$ by the Lie ideals $\n_k(\A)$ such that the quotient Lie algebras $\n_+(\A)\big/ \n_k(\A)$ are nilpotent. The inverse limit is a Lie algebra we shall write as
\be \hn_+(\A) := \invlim_k \n_+(\A)\big/ \n_k(\A) ,\nn\ee
whose elements are by definition infinite sums $\sum_{n>0}u_n$, with $u_n\in \g_n(\A)$ for each $n$, which truncate to finite sums when working in any given quotient $\n_+(\A)\big/ \n_k(\A)$, $k\in \ZZ_{>0}$. 

Define also
\begin{align}
 \hb_+(\A) &:= \h(\A) \oplus \hn_+(\A)\nn\\
 \hg(\A) := \n_-(\A) \oplus \hb_+(\A) &=  \n_-(\A) \oplus\h(\A) \oplus \hn_+(\A)  . \nn\end{align}
These are Lie algebras. 

\subsection{The group $\hN_+(\A)$} \label{sec: group lhN}
Let $\exp \big( \n_+(\A) / \n_k(\A) \big)$ denote a copy of the vector space $\n_+(\A)/ \n_k(\A)$ and let $m\mapsto \exp(m)$ be the map into this copy, for each $k>0$. As the notation is supposed to suggest, $\exp \big( \n_+(\A) / \n_k(\A) \big)$ gets a group structure, given by
\be  \exp(x)\exp(y) := \exp\left(x+y+ \frac 12 [x,y] + \dots\right) ,\nn\ee
where the expression on the right is the usual Baker-Campbell-Hausdorff formula, from which we need only finitely many terms since each $\n_+(\A) / \n_k(\A)$ is nilpotent. We have the commutative diagram
\be\begin{tikzcd}
\n_+(\A) / \n_m(\A)\rar[two heads]\dar{\sim} & \n_+(\A) / \n_k(\A) \dar{\sim}\\
\exp \big(\n_+(\A) / \n_m(\A) \big) \rar[two heads] & \exp \big( \n_+(\A) / \n_k(\A) \big)
\end{tikzcd}
\nn\ee
where the vertical maps are the formal exponential maps, which are by definition bijections, and the horizontal maps are the canonical projections for all $m\geq k$. 
The group $\hN_+(\A)$ is then the inverse limit
\be  \hN_+(\A) := \invlim_k \exp \big( \n_+(\A) / \n_k(\A) \big). \nn\ee
and the diagram above defines an exponential map $\exp : \hn_+(\A) \to \hN_+(\A)$.

\subsection{The group $\hB_+(\A)$}
Recall the fundamental coweights $\{\chLa_i\}_{i\in I}$ from \cref{ld}.  
Let $P:= \bigoplus_{i=0}^\ell \ZZ \chLa_i\subset \h$ and define $\H(\A)$ to be the abelian group generated by elements of the form $\phi^\lambda$ with $\phi \in \U(\A)$ and $\lambda\in P$, subject to the relations $\phi^{\lambda} \psi^{\lambda} = (\phi\psi)^\lambda$ and $\phi^{\lambda+\mu} = \phi^\lambda \phi^\mu$ for all $\phi,\psi\in \U(\A)$ and $\lambda,\mu\in P$. We have the adjoint action of $\H(\A)$ on the Lie algebra $\hn_+(\A)$, defined weight space by weight space, with 
\be 
\phi^\lambda n \phi^{-\lambda} :=  \phi^{\langle \lambda, \al\rangle} n,
\label{aaH}\ee
for all $n$ in the subspace of $\hn(\A)$ of weight $\al\in \bigoplus_{i\in I}\ZZ\al_i$. (Note that $\langle \lambda, \al\rangle$ is an integer and hence $\phi^{\langle \lambda, \al \rangle} \in \U(\A)$.) The adjoint action of $\H(\A)$ on the group $\hN_+(\A)$ is then given by $\phi^\lambda\exp(n)\phi^{-\lambda}:= \exp(\phi^\lambda n\phi^{-\lambda})$. Finally, we define $\hB_+(\A)$ to be the semi-direct product
\be
\hB_+(\A) := \hN_+(\A) \rtimes \H(\A),
\nn\ee
so elements of $\hB_+(\A)$ are of the form $\exp(n) \phi^\lambda$, $n\in \hn_+(\A)$, $\phi\in \U(\A)$ and $\lambda\in P$, and the product is given by
\be
(\exp(n) \phi^{\lambda})( \exp(m) \psi^{\mu}) := \Big( \! \exp(n) \exp\big( \phi^{\lambda} m \phi^{-\lambda}\big) \Big) \big(\phi^\lambda \psi^\mu \big).
\nn\ee

\subsection{Definition of $\g$-opers} \label{sec: def oper}
Let $\wt\op_\g$ be the set of all connections of the form
\be \label{tnf}
\nabla = d + \left(\sum_{i=0}^\ell \psi_i(t) f_i + b(t)\right) dt
\ee
with $\psi_i(t)\in \U(\A)$ for each $i\in I$, and $b(t)\in \hb_+(\A) := \h(\A) \oplus \hn_+(\A)$. 

In calling these connections, we mean that there is an action of the group $\hB_+(\A)$ on $\wt\op_\g$ (by gauge transformations) and that the set $\wt\op_\g$ transforms in a well-defined manner under changes in coordinate. Let us describe these transformation properties. 

The action of the coordinate transformation $t=\mu(s)$ on $\nabla$ is given by
\be \mu \on \nabla := d + \left(\sum_{i=0}^\ell \psi_i(\mu(s)) f_i + b(\mu(s))\right) \mu'(s) ds. \label{cch}\ee
(Here we choose to think of $\mu$ as a passive coordinate transformation, i.e. to think that $\mu\on \nabla$ is the same connection $\nabla$ in the new coordinate $s$ given by $t=\mu(s)$.)

The action of an element $g=e^m \phi^\lambda \in \hB_+(\A)$ on $\nabla$ is defined as follows. 
First, \cref{aaH} defines the adjoint action of $\phi^\lambda\in\H(\A)$ on the Lie algebra $\hg(\A):= \n_-(\A) \oplus \hb_+(\A)$, weight space by weight space, while the adjoint action of $e^m \in \hN_+(\A)$ on an element $u\in \hg(\A)$ is given by
\be u\mapsto e^m u e^{-m}  := \sum_{k \geq 0} \frac 1{k!} \ad_m^k u, \nn\ee
where $\ad_m u := [m,u]$.
In this way we get the adjoint action of $\hB_+(\A) = \hN_+(\A) \rtimes \H(\A)$ on $\hg(\A)$. 
For any $g=e^m \phi^\lambda \in \hB_+(\A)$, define 
\be (\del_t g) g^{-1} :=  \lambda \phi^{-1} \del_t\phi + \sum_{k \geq 1} \frac{1}{k!} \ad_m^{k-1} \del_tm ,\nn\ee
an element of $\hb_+(\A)$. One can then check that there is a well-defined action of $\hB_+(\A)$ on the affine space (over $\hg(\A)$) of connections valued in $\hg(\A)$, given by
\be d + v dt \mapsto d + \left(g v g^{-1} - (\del_t g)g^{-1}\right)dt,\qquad v\in \hg(\A). \nn\ee
This action stabilizes the set $\wt\op_\g$. We shall write $g \nabla g^{-1}\in \wt\op_\g$ for the image under the action of $g$ of the connection $\nabla\in \wt\op_\g$. We call this the gauge action of $\hB_+(\A)$ on $\wt\op_\g$.
 
The \emph{space of $\g$-opers}, is by definition the quotient of the set of connections $\wt\op_\g$ by the gauge action of $\hB_+(\A)$:
\be \Op_{\g} := \wt\op_{\g} \big/ \hB_+(\A). \ee

Recall the definition of the element $p_{-1}\in \n_-$ from \cref{pmdef}: $p_{-1} := \sum_{i\in I}  f_i$.
Let $\op_\g\subset \wt\op_\g$ denote the set of connections of the form \cref{tnf} with $\psi_i=1$ for each $i\in I$, i.e. the set of connections of the form
\be \nabla = d + (p_{-1} + b) dt, \qquad b\in \hb_+(\A).\nn\ee
(It is an affine space over $\hb_+(\A)$.)
\begin{lem}\label{oplem} 
Each $\H(\A)$-orbit in $\wt\op_\g$ contains a representative in $\op_\g$. This representative is unique, and  the stabilizer of $\op_\g$ in $\wt\op_\g$ is $\hN_+(\A)$.
Hence
\be \Op_\g \cong \op_\g \big/ \hN_+(\A).\nn\ee\qed
\end{lem}

\subsection{The (quasi-)canonical form} \label{sec: qcf}
Recall from \cref{sec: adef} the definition of the exponents $j\in E$ of $\g$ and of the generators $p_j\in \a_+$. The following result was established in \cite{LVY}, following \cite{GelfandDorfman,DS}.  (The proof in \cite{LVY} is in the meromorphic setting, but the same proof goes through the case of the disc.) 
\begin{thm}[\cite{LVY}]\label{thm: quasi-canonical form}
Suppose $\g$ is of affine type. 
Fix a formal coordinate $t$ on $D$ (resp. a holomorphic coordinate $t$ on $U$).

Every oper $[\nabla] \in \Op_{\g}$ has a \emph{quasicanonical} representative of the form 
\be
\nabla = d + \Bigg( p_{-1} - \frac{\twist(t)}{\coxeter} \chweyl + \sum_{j \in E_{\geq 1}} v_j(t) p_j  \Bigg) dt,
\ee
where $\twist(t)$ and $v_j(t)$, for each positive exponent $j \in E$, are formal series in $t$ (resp. meromorphic functions on $t(U) \subset \CC$). 

The gauge transformations preserving quasi-canonical form are those of the form $\exp\big( \sum_{j \in E_{\geq 2}} f_j p_j \big)$ for formal series (resp. meromorphic functions) $f_j(t)$. The effect of such gauge transformations is to send 
\begin{equation}
v_j(t) \longmapsto v_j(t) - f_j'(t)+ \frac{j \twist(t)}{\coxeter} f_j(t)\label{vuptof}
\end{equation}
for all $j\in E_{\geq 2}$, and to leave $\twist$ and $v_1$ invariant.
\qed\end{thm}

\subsection{Coordinate transformations}\label{sec: op ct}

One can now study the behaviour under coordinate transformations of the functions $v_j$ and $\twist$ appearing in the quasi-canonical form of the oper. 
That is, one can start with a connection $\nabla$ of quasi-canonical form with respect to the coordinate $t$ and then transform to a new coordinate given by $t=\mu(s)\neq s$, as in \cref{cch}. The resulting connection will not be in quasi-canonical form in the new coordinate -- but it can necessarily be brought to this form by means of a gauge transformation. One interprets the resulting power series $\tilde v_j$ and $\tilde \twist$ as the images of $v_j$ and $\twist$ under the coordinate transformation. 
One finds the following. (For details see e.g. \cite[\S6]{LVY}, or the proof of \cref{prop: doao} below where an example of this sort of calculation occurs.)
\begin{align} \label{phipt}
- \frac{1}{\coxeter} \tilde\twist(s) &= - \frac{1}{\coxeter} \twist(\mu(s))\mu'(s) - \frac{\mu''(s)}{\mu'(s)},\\
\tilde v_j(s)  &= v_j(\mu(s)) \mu'(s)^{j+1}, \qquad j\in E.\label{vpt}
\end{align}
In the meromorphic setting, we know from \cref{crs} above how to interpret \cref{phipt,vpt,vuptof}, and we recover the following.
\begin{thm}[\cite{LVY}]\label{opu}$ $
$\Op_{\g}(U)$ fibres over the affine space $\Conn(U, \Omega)$ of meromorphic connections on the canonical bundle and we have the isomorphism 
\be
\Op_{\g}(U)^{\overline{\nabla}} \cong \Gamma(U,\Omega^2) \times
\prod_{j\in E_{\geq 2}} H^1(U,\Omega^j, \overline{\nabla})
\nn\ee
for the fibre over any connection $\overline{\nabla} \in \Conn(U, \Omega)$. \qed
\end{thm}

To state the analogous result on the disc, we need some notations.
A \emph{$j$-differential (on D)} is an object of the form $f(t) dt^j$, specified by a series $f(t) \in \CC[[t]]$, its component in the coordinate $t$. By definition (and as the notation suggests) this component must transform as  $f(t) \mapsto f(\mu(s)) \mu'(s)^j$ under the coordinate transformation $t=\mu(s)$, $\mu\in \Aut\O$. Let us write $\Diff^j$ for the $\O$-module of $j$-differentials. 

For each $j\in \ZZ_{\geq 1}$, we write $\Conn(\Diff^j)$ for the space of connections on $j$-differentials. It is an an affine space over $\Diff$. That is to say, as an abelian group under addition, $\Diff$ acts freely transitively on $\Conn(\Diff^j)$ (and this action is $\Aut\O$ equivariant). Just as in the meromorphic case, cf. \cref{twisttrans} and the discussion following, $\Conn(\Diff) \cong \Conn(\Diff^j)$ for all $j\geq 1$

On comparing \cref{phipt} with \cref{twisttrans} we see that $- \frac{1}{\coxeter}\twist(t)$ transforms as the component of an affine connection, call it  $\nabla^\twist$. For each $j\in E$, the power series $v_j(t)$ transforms as a $(j+1)$-differential. As in the meromorphic case, we can now interpret \cref{vuptof} as saying that the $v_j(t)$ define classes in the cohomology $H^1(\Diff^j, \nabla^\twist) := \Diff^{j+1} \big/ \nabla^\twist \Diff^j$
of the connection $\nabla^\twist$. But these cohomologies on the disc are all trivial.\footnote{Indeed, let $F(t) = \sum_{n<0} F_n t^{-n-1} \in \CC[[t]]$ be the component of a $(j+1)$-differential $F(t)dt^j\in \Diff^{j+1}$. We must show that there exists $G(t) =\sum_{n\leq 0} G_n t^{-n} \in \CC[[t]]$ such that $F(t) = G'(t) - \frac j\coxeter\twist(t) G(t)$. That is,
\begin{align} 
F_{-1} &= G_{-1} - \frac j \coxeter \twist_{-1} G_0 \nn\\
F_{-2} &= 2G_{-2} - \frac j \coxeter \left(\twist_{-2} G_0 + \twist_{-1} G_{-1}\right)  \nn\\
&\vdots\nn\\
F_{-n} &=  n G_{-n} - \frac j\coxeter \sum_{k=-n}^{-1} \twist_k G_{-n-k}.\nn
\end{align}
where $\twist(t) = \sum_{n<0} \twist_n t^{-n-1}$ is the component of the connection. 
We can pick $G_0=0$ and then solve the $n$th equation above by choice of $G_{-n}$, for each $n\geq 1$.}

We have arrived at the following description of the space $\Op_{\g}(D)$ of opers on the disc. 
\begin{prop}\label{opdisc} We have the following ($\Aut\O$-equivariant) isomorphism:
\be \Op_{\g}(D) \cong \Conn(\Diff) \times \Diff^2.\nn\ee\qed
\end{prop}

\subsection{Miura opers}\label{sec: Miura oper def}
An \emph{$\g$-Muira oper} is a connection $\nabla\in \op_\g \subset \wt\op_\g$ of the special form
\be \nabla = d + (p_{-1} + u(t)) dt, \qquad u(t)\in \h(\A).\label{mop}\ee
Let $\MOp_\g$ denote the set of such connections. It is an affine space over $\h(\A)$.

We have the canonical map 
\be \MOp_\g \to \Op_\g \label{moptoop}\ee
which associates to each Miura oper $\nabla$ the underlying oper $[\nabla]$, i.e. its $\hB_+(\A)$-gauge equivalence class in $\wt\op_\g$. 

\begin{rem} For us (and also for \cite{MVopers,VW,VWW}) a Miura oper is thus not an oper, strictly speaking -- but one can think of a Miura oper $\nabla$ as consisting of its underlying oper $[\nabla]\in \Op_\g$ together with the extra data of a representative $\nabla$ of the special form \cref{mop}. That is the right intuition, because in the original, more geometric, definition of opers \cite{BDopers} and Miura opers \cite{Frenkel_2005} on Riemann surfaces, a Miura oper is by definition an oper together with extra data. (See \cite[\S5]{Fopersontheprojectiveline}. Note also that, for compatibility with that definition, what we call Miura opers here should really be called generic Miura opers.) 
\end{rem}

\subsection{Identification of $\picc_0$ with $\Fun\MOp_{\g}(D)$}
We write $\MOp_\g(D)$ for the space of Miura opers on the disc, i.e. in the case $\A = \CC[[t]]$. 
In the coordinate $t$ we can write the series expansion of the element $u\in\h(\A)$ in \cref{mop} as 
\be u = \sum_{n<0} u_{n} t^{-n-1},\nn\ee
with $u_n\in \h$ for each $n$.
Recall $\{b^i\}_{i=1,\dots,\dim\h}$ and $\{b_i\}_{i=1,\dots,\dim\h}$ are dual bases of $\h$ with respect to the form $\bilin\cdot\cdot$. Define
\be u^i_n := \la u_n, b^i \ra ,\qquad u_{i,n} := \la u_n, b_i\ra.\nn\ee
Let $\uu_{i,n}\in \MOp_{\g}(D)^*$ be the linear map which sends $\nabla$ to $u_{i,n}$. We can regard $\uu_{i,n}$ as coordinate functions on the space $\MOp_{\g}(D)$. They generate the $\CC$-algebra 
\be \Fun\MOp_\g(D) := \CC[\uu_{i,n}]_{i=1,\dots,\dim\h,n<0} \nn\ee 
of polynomial functions on $\MOp_\g(D)$. 
In view of \cref{picdef}, by identifying $\uu_{k,n}$ with $b_{k,n}$, we get an isomorphism differential algebras
\be \Fun\MOp_\g(D) \cong \picc_0 .\nn\ee
We endow $\picc_0$ with the structure of a conformal algebra as in \S\ref{sec: cvg}.

\begin{prop}\label{prop: doao}
This isomorphism
\be \Fun\MOp_\g(D) \cong \picc_0 .\label{mopisom}\ee
of differential algebras is $\Aut\O$- and $\Der\O$-equivariant.
\end{prop}
\begin{proof} First, just as in the calculation sketched in \S\ref{sec: op ct} above, one can find the behaviour of the power series
\be u_i := \la u,b_i\ra = \sum_{n<0} u_{i,n} t^{-n-1} \nn\ee
under the coordinate transformation $t\mapsto s$ with $t=\mu(s)$, $\mu\in \Aut\O$. Indeed, in the new coordinate $s$ the connection in \cref{mop} becomes
\be \nabla = d + p_{-1} \mu'(s) ds + u(\mu(s)) \mu'(s) ds.\label{nmop}\ee
This can be brought back into the Miura form by performing a gauge transformation by $\mu'(s)^\chweyl\in \H(\O)$:
\be \mu'(s)^\chweyl \,\nabla\, \mu'(s)^{-\chweyl} =   d + (p_{-1} + \tilde u(s))ds \nn\ee
where
\be \tilde u(s) = u(\mu(s)) \mu'(s) - \chweyl \frac{\mu''(s)}{\mu'(s)} .\label{utrans}\ee
Let us adopt the viewpoint that this is an active transformation, transforming the Miura oper $\nabla= d+ (p_{-1} + u(t))dt$ to the new Miura oper $\mu\on \nabla = d+(p_{-1} + \tilde u(t))dt$. This gives the action of $\Aut\O$ on $\MOp_{\g}(D)$. The action of $\Aut\O$ on $\Fun\MOp_{\g}(D)$ is then given by  $(\mu\on \uu_{i,n})(\nabla) := \uu_{i,n}(\mu^{-1} \on \nabla)$ (so that the pairing $\uu_{i,n}(\nabla)\in \CC$ is, correctly, $\Aut\O$-invariant). 


Consider an infinitesimal transformation,  $\mu(s) = (1+\eps L_k) s = (1 - \eps s^{k+1} \del_s) s = s- \eps s^{k+1}$, so that $\mu'(s) = 1-\eps (k+1) s^k$. We have 
\begin{align}
\tilde u(s)  &= \left( u(s) - \eps s^{k+1} u'(s)\right)\left(1-\eps (k+1) s^k\right) - \eps  k (k+1) s^{k-1}\chweyl + \dots\nn\\
 &= u(s) - \eps\left( s^{k+1} u'(s) + (k+1) s^k u(s) + k (k+1) s^{k-1} \chweyl\right) + \dots.\nn
\end{align}
Let us arrange the coordinate functions $\uu_{i,n}$ into power series in the same way as $u_{i,n}$, i.e. $\uu(s) := \sum_{n<0} \uu_n t^{-n-1}$ with $\uu_n := \uu_{i,n} b^i$. 
By definition 
\be \tilde u(s) = (\mu^{-1}\on \uu)(\nabla) = ((1-\eps L_k+\dots)\on \uu)(\nabla).\nn\ee 
On comparing coefficients, we see that 
\be L_k\uu_{n} = \res ds s^n \left( s^{k+1} \uu'(s) + (k+1) s^k\uu(s) + k (k+1) s^{k-1}\chweyl\right). \nn\ee  
That is
\be L_k \uu_n = \begin{cases} -n \uu_{n+k} & k\leq -n, \\ k(k+1) \chweyl & k=-n,\\ 0 & k>-n. \end{cases}\nn\ee
In view of \cref{cld}, and recalling that $\sv = -\mu^{-1}(\chweyl)$, this establishes the $\Aut\O$-equivariance of the isomorphism \cref{mopisom}. The Lie algebra $\Der\O$ is slightly larger than the Lie algebra $\Der_0\O= t\CC[[t]]\del_t$ of $\Aut\O$, because it contains also the generator $L_{-1} = -\del_t$. But the calculation above still goes through for $k=-1$, i.e. for this generator $L_{-1}$.
\end{proof} 
\subsection{Identification of $\Wgc$ with $\Fun\Op_{\g}(D)$}\label{sec: wop}
Now we shall show that the action of the screening operators on $\picc_0$ coincides with the action of infinitesimal gauge transformations by simple root vectors on $\Fun\MOp_{\g}$. 
Let us start with a Miura $\g$-oper $\nabla$ as in \cref{mop}.
We shall consider the effect of a gauge transformation by the element $\exp(a e_i) \in \hN_+(\O)$. 
We have (noting in the second line that $[e_i,[ e_i,u]] =0$ since $[e_i,u] \propto e_i$),
\begin{align} e^{a e_i} \nabla  e^{-a e_i}
&= d - a' e_i dt + e^{a \ad e_i)}\left( p_{-1} + u\right)dt \nn\\
&= d - a' e_i dt + \bigg( p_{-1} + a\al_i + \half a^2 \left[ e_i,\al_i\right]  
+u + a \left[ e_i, u \right]\bigg)dt\nn\\
&= d + \bigg( p_{-1} + u + a \al_i 
+ e_i \left( a^2 - a' - a \la \al_i,u\ra \right) \bigg)dt.\label{scf}
\end{align}
Note $\left[ e_i , u\right] = - \la u, \al_i\ra e_i$. Therefore this gauge transformation preserves the Miura form if and only if $a$ obeys the (``Ricatti-type'') equation 
\be a^2 - a' - a\la\al_i,u\ra=0. \label{ric}\ee

(In the meromorphic setting on $\cp1$, solutions of this equation give rise to the \emph{reproduction procedure} which generates families of new solutions to the Bethe equations starting from a given solution; see \cite{ScV,MV1,MVopers}.)

In our setting of the formal disc, $u$ and $a$ are power series. We can choose to write them as
\be u = \sum_{n<0} u_{n} t^{-n-1}, \qquad a = \sum_{n\leq 0} a_n t^{-n} ,\nn\ee
with $u_n , a_n\in \h$ for each $n$. Given a power series $f=\sum_{n\geq 0} f_n t^n $, set $\int f := \sum_{n\geq 0} f_n t^{n+1} / (n+1)$. 
For any value of a parameter $c\in \CC$, there is a solution to \cref{ric} given by
\be a = -\frac{g}{ c^{-1} +\int g },\quad\text{where}\quad g=\exp\int \left(-\la u,\al_i\ra\right),\nn\ee
for then indeed
\begin{align} a^2 - a' &= \frac{g^2}{\left(c^{-1}+\int g\right)^2} - \left(-\frac{g'}{c^{-1}+\int g} + \frac{g^2}{\left(c^{-1}+\int g\right)^2}\right) \nn\\
&= \frac{g'}{c^{-1}+\int g} = \frac{g}{c^{-1}+\int g} \frac{g'}{g} = a \la u,\al_i\ra;\nn\end{align}
moreover this is the unique solution with leading behaviour $a= -c+\dots$. 

Thus the solutions to \cref{ric} form a one-parameter family containing the trivial solution $(a=-c=0)$.\footnote{One can ask about the point at infinity, $c^{-1}=0$. Generically there is no corresponding solution because $0+\int g$ is not invertible in $\O$. But there is a special case when $\la u,\al_i\ra$ is nonzero and constant. Then $g= \exp(- \la u,\al_i\ra t)$ and we have the constant nonzero solution $a=-g/\int g =  \la u ,\al_i\ra$. That is important because it gives rise to the action of the (discrete) Weyl group on ``finite Miura opers'', i.e. elements of $\g$ of the form $p_{-1} + u$.}

Consequently, it is enough to consider infinitesimal gauge transformations $\exp(\eps a  e_i) \sim 1+ \eps a  e_i$. 
In that case $a$ solves the equation 
\be a'=  -\la \al_i,u\ra a, \label{iric}\ee
whose solution, unique up to scale, is 
\be a = \exp\int\left(-\la u,\al_i\ra\right)
 = \exp\left( -\sum_{n<0}\frac{ \la u_n,\al_i \ra t^{-n}}{n} \right) 
 = \exp\left( -\bilin{\al_i}{b_j} \sum_{n<0}\frac{ u^j_n  t^{-n}}{n} \right) .
\label{adef}\ee 
In view of \cref{Vf}, \cref{Vf2} and \cref{Vdef}, 
if we identify $u_{k,n}$ with $b_{k,n}$ then we have
\be a = \sum_{n\leq 0} V_{\al_i}[n] t^{-n}.\nn\ee 
From \cref{scf} we see that $\delta u = \eps a\al_i$. That is, $\delta u_{k,n} := \eps \la \delta u_n,b_k\ra= \eps a_{n+1} \la \al_i , b_k\ra$. 
Thus this infinitessimal gauge transformation acts on polynomials in the $u_{k,n}$, by the derivation 
\be \sum_{m<0} a_{m+1} \la \al_i , b_k \ra \frac{\del}{\del u_{k,m}}
= \sum_{m\leq 0} V_{\al_i}[m] \la \al_i, b_k \ra \frac{\del}{\del u_{k,-1+m}}.\label{Qdef}\ee
Recall from \cref{kdef} that $\la \al_i , \cdot \ra = \vareps_i^{-1} \bilin{\al_i}\cdot$. We have established the following.
\begin{prop} \label{flowprop} 
After identifying $\uu_{k,n}$ with $b_{k,n}$, the action of 
the flow generated by 
$ e_i$ on $\MOp_\g(D) \cong \pi_0$ is given by
\be Q_i = -\vareps_i^{-1} \shift_{\al_i}^{-1} \ol S_{\al_i}.\nn\ee
\qed
\end{prop}
Let us write $\Fun\Op_\g(D)$ for the space of polynomial functions on $\Op_{\g}(D)$:  
\be \Fun\Op_{\g}(D) \cong \CC[\twist_{n}, v_{1,n}]_{n<0},\nn\ee
cf. \cref{thm: quasi-canonical form} and \cref{opdisc}. 

By construction, a polynomial in the $\uu_{i,n}$, i.e. an element of $\Fun\MOp_{\g}(D)$, is an element of $\Fun\Op_{\g}(D)$ if and only if it is in the kernel of the $Q_i$. 

After identifying $\uu_{k,n}$ with $b_{k,n}$, one finds $\twist_{n} = \chcent_{n}$ and $v_{1,n} = \ol\confvec_{n}$, cf. \cref{sec: cvg,sec: afc}.

In view of the definition of $\Wgc$ as the kernel of the screenings, we have arrived at the following statement. 
(For the analogous statement in finite types see \cite[\S8.2.5]{Fre07}.)
\begin{thm}\label{Wpi} We have the commutative diagram:
\be
\begin{tikzcd}
\picc_0 \rar{\sim} & \Fun \MOp_\g(D) \\
\Wgc\uar[hook]  \rar{\sim}& \Fun \Op_\g(D) \uar[hook] 
\end{tikzcd}
\nn\ee
These maps are equivariant with respect to $\Der\O$ and $\Aut\O$.
\qed\end{thm}

\subsection{Global identifications}\label{sec: gi}
\begin{prop}\label{mopconn}
We have an isomorphism of sheaves on $\cp1$,
\be \conn{-\chweyl}(U) \cong \MOp_{\g}(U), \nn\ee
given by
\be \chi(t) \mapsto d + (p_{-1} - \chi(t)) dt .\nn\ee
\end{prop}
\begin{proof}
This follows from \cref{chitrans,utrans}. 
\end{proof}
In particular, 
\be \conn{-\chweyl}(\cp1)_\x \cong \MOp_{\g}(\cp1)_\x .\nn\ee

Thus, given $\chi \in \conn{-\chweyl}(\cp1)_\x$ we have the corresponding Miura oper $\nabla = d + (p_{-1} - \chi(z))dz$ and its underlying oper $[\nabla]\in \Op_{\g}(\cp1)_\x$. The affine connection $\Tkn_\chi$ is manifestly the affine connection defined by $[\nabla]$, cf. \cref{sec: conrho} and \cref{thm: quasi-canonical form}.

\begin{prop}\label{fc} The cohomology classes $[ \coinv_\chi(\hamd_j dt^{j+1}) ] \in  H^1(\cp1 ,\Omega^j, \Tkn_\chi)_\x$ are
the classes of the functions appearing in (any) quasi-canonical form of the oper $[\nabla]$:
\be [\coinv_\chi( \hamd_j dt^j ) ] = [ v_j(t) dt^j] .\nn\ee
\end{prop}
\begin{proof}
The affine connection $\Tkn_\chi$ gives an affine structure on $\cp1\setminus \x$, i.e. an atlas of local charts $t: U \to \CC$ where each local holomorphic coordinate $t$ is such that $dt$ is constant for $\Tkn_\chi$. It is enough to show that the statement holds in such a chart. But in such a chart the statement reduces to a standard statement about densities of $\g$-mKdV Hamiltonians. 

Indeed, in such a chart the component of the connection $\Tkn_\chi$ is vanishing: $\twist(t) = 0$. Hence, given \cref{twistrem}, $\coinv_\chi(\chcent_{-n} v)$ vanishes too, for all $v\in \picc_0$.  But that means $\coinv_\chi$ descends to a well-defined map from the quotient by such modes, and this quotient is what we called $\pifin_0$ above, cf. \eqref{afffin}. So in this coordinate the $\coinv_\chi( \hamd_j )$ are precisely the densities of integrals of motion of Affine Toda field theory or equivalently (see \cite{FFIoM}) the densities of integrals of motion of $\g$-mKdV. On the oper side, the fact that $\twist(t)=0$ means we are effectively working with opers for the derived subalgebra $\smash{\g}'= [\g,\g]$. This is the setting of \cite[Section 6]{DS}. It is shown there  (in Proposition 6.11) that the functions in the quasi-canonical form of such opers are these same densities. (In \cite{DS}, what we are calling the quasi-canonical form of the affine oper appears in Proposition 6.2, and the Miura form appears in Lemma 6.7.)
\end{proof}

\newcommand{\etalchar}[1]{$^{#1}$}
\providecommand{\MR}[1]{}
\providecommand{\bysame}{\leavevmode\hbox to3em{\hrulefill}\thinspace}
\providecommand{\MR}{\relax\ifhmode\unskip\space\fi MR }
\providecommand{\MRhref}[2]{%
  \href{http://www.ams.org/mathscinet-getitem?mr=#1}{#2}
}
\providecommand{\href}[2]{#2}

\end{document}